\documentclass[a4,11pt]{article}
\usepackage[subrefformat=parens]{subcaption}
\usepackage[dvips]{graphicx}
\usepackage{amssymb}
\usepackage{amstext}
\usepackage{amsmath}
\usepackage{amscd}
\usepackage{amsthm}
\usepackage{amsfonts}
\usepackage{enumerate}
\usepackage{graphicx}
\usepackage{latexsym}
\usepackage{mathrsfs}
\usepackage{mathtools}
\usepackage{cases}
\usepackage[svgnames]{xcolor}
\usepackage{verbatim}
\usepackage{bm}
\usepackage{tikz}
\usepackage{caption}
\usepackage[top=3.5cm, bottom=3.5cm, left=3.5cm, right=3.5cm]{geometry}
\usepackage[labelformat=empty,labelsep=none]{subcaption}
\usepackage{authblk}
\usepackage{comment}
\usepackage{hyperref}
\newtheorem{thm}{Theorem}[section]

\newtheorem{lem}[thm]{Lemma}

\theoremstyle{definition}
\newtheorem{rem}[thm]{Remark}

\newtheorem*{claim*}{Claim}
\theoremstyle{remark}

\numberwithin{equation}{section}

\title{\vspace{-3cm}\textbf{Inverse medium scattering problems with Kalman filter techniques II. Nonlinear case}}

\author[1]{\rm Takashi Furuya}
\author[2]{\rm Roland Potthast}
\affil[1]{ {\small Department of Mathematics, Hokkaido University, Japan}}
\affil[ ]{{\small Email: takashi.furuya0101@gmail.com}\vspace{3mm}}

\affil[2]{{\small Data Assimilation Unit, Deutscher Wetterdienst, Germany}}
\affil[ ]{{\small Email: Roland.Potthast@dwd.de}}

\date{}
\begin{document}
\maketitle
\begin{abstract}
In this paper, we study the inverse medium scattering problem to reconstruct unknown inhomogeneous medium from far field patterns of scattered waves. In the first part of our work, the linear inverse scattering problem was discussed, while in the second part, we deal with the nonlinear problem. The main idea is to apply the linear Kalman filter to the linearized problem. There are several ways to linearize, which introduce two reconstruction algorithms. Finally, we give numerical examples to demonstrate our proposed method.
\end{abstract}
\date{{\bf Key words}. Inverse acoustic scattering, Inhomogeneous medium, Far field pattern, Tikhonov regularization method, Kalman filter, Levenberg–Marquardt.}
\section{Introduction}
The inverse scattering problem is the problem to determine unknown scatterers by measuring scattered waves that is generated by sending incident waves far away from scatterers. 
It is of importance for many applications, for example medical imaging, nondestructive testing, remote exploration, and geophysical prospecting. 
Due to many applications, the inverse scattering problem has been studied in various ways.
For further readings, we refer to the following books \cite{Cakoni, Chen, ColtonKress, Kirsch, NakamuraPotthast}, which include the summary of classical and recent progress of the inverse scattering problem.
\par
We begin with the mathematical formulation of the scattering problem. 
Let $k>0$ be the wave number, and let $\theta \in \mathbb{S}^{1}$ be incident direction. 
We denote the incident field $u^{inc}(\cdot, \theta)$ with the direction $\theta$ by the plane wave of the form 
\begin{equation}
u^{inc}(x, \theta):=\mathrm{e}^{ikx \cdot \theta}, \ x \in \mathbb{R}^2. \label{1.1}
\end{equation}
Let $Q$ be a bounded domain and let its exterior $\mathbb{R}^2\setminus  \overline{Q}$ be connected. 
We assume that $q \in L^{\infty}(\mathbb{R}^2)$, which refers to the inhomogeneous medium, satisfies $\mathrm{Im}q \geq 0$, and its support $\mathrm{supp}\ q$ is embed into $Q$, that is $\mathrm{supp}\ q \Subset Q$. 
Then, the direct scattering problem is to determine the total field $u=u^{sca}+u^{inc}$ such that
\begin{equation}
\Delta u+k^2(1+q)u=0 \ \mathrm{in} \ \mathbb{R}^2, \label{1.2}
\end{equation}
\begin{equation}
\lim_{r \to \infty} \sqrt{r} \biggl( \frac{\partial u^{sca}}{\partial  r}-iku^{sca} \biggr)=0, \label{1.3}
\end{equation}
where $r=|x|$. 
The {\it Sommerfeld radiation condition} (\ref{1.3}) holds uniformly in all directions $\hat{x}:=\frac{x}{|x|}$. 
Furthermore, the problem (\ref{1.2})--(\ref{1.3}) is equivalent to the {\it Lippmann-Schwinger integral equation}
\begin{equation}
u(x, \theta)=u^{inc}(x, \theta)+k^2\int_{Q}q(y)u(y, \theta)\Phi(x,y)dy, \label{1.4}
\end{equation}
where $\Phi(x,y)$ denotes the fundamental solution to Helmholtz equation in $\mathbb{R}^2$, that is, 
\begin{equation}
\Phi(x,y):= \displaystyle \frac{i}{4}H^{(1)}_0(k|x-y|), \ x \neq y, \label{1.5}
\end{equation}
where $H^{(1)}_0$ is the Hankel function of the first kind of order one. 
It is well known that there exists a unique solution $u^{sca}$ of the problem (\ref{1.2})--(\ref{1.3}), and it has the following asymptotic behaviour,
\begin{equation}
u^{sca}(x, \theta)=\frac{\mathrm{e}^{ikr}}{\sqrt{r}}\Bigl\{ u^{\infty}(\hat{x},\theta)+O\bigl(1/r \bigr) \Bigr\} , \ r \to \infty, \ \ \hat{x}:=\frac{x}{|x|}. \label{1.6}
\end{equation}
The function $u^{\infty}$ is called the {\it far field pattern} of $u^{sca}$, and it has the form
\begin{equation}
u^{\infty}(\hat{x},\theta)=\frac{k^2}{4\pi}\int_{Q}\mathrm{e}^{-ik \hat{x} \cdot y} u(y, \theta)q(y)dy=:\mathcal{F}_{\theta}q(\hat{x}), \label{1.7}
\end{equation}
where the far field mapping $\mathcal{F}_{\theta}:L^{2}(Q) \to L^{2}(\mathbb{S}^{1})$ is defined in the second equality for each incident direction $\theta \in \mathbb{S}^{1}$. 
For further details of these direct scattering problems, we refer to Chapter 8 of \cite{ColtonKress}. 
\par
We consider the inverse scattering problem to reconstruct the function $q$ from the far field pattern $u^{\infty}(\hat{x}, \theta_n)$ for all directions $ \hat{x} \in \mathbb{S}^{1}$ and several directions $\{ \theta_n \}_{n=1}^{N}\subset \mathbb{S}^{1}$ with some $N \in \mathbb{N}$, and one fixed wave number $k>0$.
It is well known that the function $q$ is uniquely determined from the far field pattern $u^{\infty}(\hat{x}, \theta)$ for all $ \hat{x}, \theta \in \mathbb{S}^{1}$ and one fixed $k>0$ (see, e.g., \cite{bukhgeim2008recovering, novikov1988multidimensional, ramm1988recovery}), but the uniqueness for several incident plane wave is an open question. 
For impenetrable obstacle scattering case, if we assume that the shape of scatterer is a polyhedron or ball, then the uniqueness for a single incident plane wave is proved (see \cite{alessandrini2005determining, cheng2003uniqueness, liu2006uniqueness, liu1997inverse}). 
Recently in \cite{alberti2020infinitedimensional}, they showed the Lipschitz stability for inverse medium scattering with finite measurements $\{ u^{\infty}(\hat{x}_{i}, \theta_{j}) \}_{i,j=1,...,N}$ for large $N \in \mathbb{N}$ under the assumption that the true function belongs to a compact and convex subset of finite-dimensional subspace. 
\par
Our problem for equation (\ref{1.7}) with finite measurements $\{ u^{\infty}(\cdot, \theta_n) \}_{n=1}^{N}$ is not only ill-posed, but also nonlinear, that is, the far field mappings $\mathcal{F}_{\theta}$ is nonlinear because $u(\cdot, \theta)$ in (\ref{1.7}) is a solution for the Lippmann-Schwinger integral equation (\ref{1.4}), which depends on $q$. 
Existing methods for solving nonlinear inverse problem can be roughly categorized into two groups: iterative optimization methods and qualitative methods. 
The iterative optimization method (see e.g., \cite{Bakushinsky, ColtonKress, Giorgi, Hohage, Kaltenbacher}) does not require many measurements, however it require the initial guess which is the starting point of the iteration. 
It must be appropriately chosen by a priori knowledge of the unknown function $q$, otherwise, the iterative solution could not converge to the true function. 
On the other hand, the qualitative method such as the linear sampling method \cite{ColtonKirsch}, the no-response test \cite{Honda}, the probe method \cite{Ikehata}, the factorization method \cite{KirschGrinberg}, and the singular sources method \cite{Potthast}, does not require the initial guess and it is computationally faster than the iterative method. 
However, the disadvantage of the qualitative method is to require uncountable many measurements. 
For the survey of the qualitative method, we refer to \cite{NakamuraPotthast}.
Recently in \cite{ito2012direct, Liu_2018}, they suggested the reconstruction method from a single incident plane wave although the rigorous justifications are lacked.
\par
The well known method to solve the nonlinear problem is the {\it Newton Method} (see e.g., \cite{Bakushinsky, ColtonKress, Kaltenbacher, Kirsch, Kress, NakamuraPotthast}), which is a classical method to construct an iterative solution based on the first-order linearization.
A natural approach applying the Newton method to our situation is to put all available measurements $\{ u^{\infty}(\cdot, \theta_n) \}_{n=1}^{N}$ and all far field mappings $\{ \mathcal{F}_{\theta_{n}} \}_{n=1}^{N}$ into one long vectors $\vec{u}^{\infty}$ and $\vec{\mathcal{F}}$, respectively, and to iteratively solve the linearized big system of $\vec{u}^{\infty}=\vec{\mathcal{F}}q$ by the Tikhonov regularization, in other words to apply {\it Levenberg–Marquardt} scheme to $\vec{u}^{\infty}=\vec{\mathcal{F}}q$. 
However, this is computationally expensive when the number $N$ of measurements is increasing in which we have to construct the bigger system.
\par
In this paper, we propose the reconstruction scheme based on the {\it Kalman filter}. 
The Kalman filter (see the original paper \cite{Kalman}) is the algorithm to estimate the unknown state in the dynamics system by using the time sequential measurements. 
The contributions of this paper are followings.
\begin{itemize}
  \item[(A)] We propose the reconstruction algorithm by combination of linearization and Kalman filter, which is equivalent to the Levenberg–Marquardt (see (\ref{4.10})--(\ref{4.14})).
  \item[(B)] We also propose the reconstruction algorithm based on the {\it Extended Kalman Filter} (see (\ref{5.9})--(\ref{5.13})).
\end{itemize}
The algorithm in (A) is proposed by understanding the Levenberg–Marquardt scheme from the viewpoint of the Kalman filter, and the equivalence is proved by the first part of our work \cite{Furuya}, which showed that the Kalman filter is equivalent to the Tikhonov regularization in the case of the linear inverse problem. 
The Extended Kalman filter (see e.g., \cite{Grewal, grewal2010applications, Jazwinski}) is the nonlinear version of the Kalman filter, which idea is to linearize the nonlinear equation and update the state and weight every time to give one incident measurement. 
The algorithm in (B) is different from that in (A) in term of when to linearize the nonlinear equation, and the number of linearization in (B) is larger than that in (A). 
The figure \ref{KFN and EKF} provides an illustration for the differences of (A) and (B).
The advantages of using Kalman filter is that we do not require to construct the big system equation $\vec{u}^{\infty}=\vec{\mathcal{F}}q$, which reduces computational costs. 
Instead, we update not only state, but also the weight of the norm for the state space, which is associated with the update of the covariance matrices of the state in the statistical viewpoint (see Section 5 in \cite{Furuya}).
Numerical experiments in Section \ref{Numerical examples} show that the reconstruction of (B) is robust to noise and its error decrease more rapidly rather than that of (A) although theoretical interpretations for this result is missing in this paper, which would be the focus of future work. 
\par
This paper is organized as follows. 
In Section \ref{Frechet derivative of the far field mapping}, we recall the Fr\'echet derivative of the far field mapping $\mathcal{F}$ and its properties. 
In Section \ref{Linearized problems}, we consider the linearized problem for nonlinear equation, and recall Levenberg–Marquardt method. 
In Section \ref{Levenberg Marquardt and Kalman filter}, we propose two reconstruction algorithms called the Full data Levenberg–Marquardt (FLM) and the Kalman filter Levenberg–Marquardt (KFL), and show that they are equivalent. 
In Section \ref{Iterative Extended Kalman filter}, we propose the reconstruction algorithm called the iterative Extended Kalman filter (EKF). 
Finally in Section \ref{Numerical examples}, we give numerical examples to demonstrate our algorithms.
\section{Fr\'echet derivative of the far field mapping}\label{Frechet derivative of the far field mapping}
The approach for solving the nonlinear equation (\ref{1.7}) often requires the linearization by the Fr\'echet derivative. 
In this section, we briefly recall the Fr\'echet derivative of the far field mapping and its properties.
We denote by $L^{\infty}_{+}(Q):=\{q \in L^{\infty}(Q): \exists q_{0}>0,\ \mathrm{Im}q \geq q_{0} \ \mathrm{a.e} \ \mathrm{on} \ Q \}$.
\par
We define the far filed mapping $\mathcal{F}_{\theta}: L^{\infty}_{+}(Q) \subset L^{2}(Q) \to L^{2}(\mathbb{S}^{1})$ by 
\begin{equation}
\mathcal{F}_{\theta}q(\hat{x}):=\frac{k^2}{4\pi}\int_{Q}\mathrm{e}^{-ik \hat{x} \cdot y} u_{q}(y, \theta)q(y)dy=:u^{\infty}_{q}(\hat{x}, \theta), \ \hat{x} \in \mathbb{S}^1, \label{2.1}
\end{equation}
where the total field $u_{q}(\cdot, \theta)$ is given by the solving the integral equation of (\ref{1.4}).
\begin{lem}\label{Lemma 2.1}
\begin{description}
\item[(1)] $\mathcal{F}_{\theta} \in C^{1}(L^{\infty}_{+}(Q), L^{2}(\mathbb{S}^{1}))$, that is, for any $q \in L^{\infty}_{+}(Q)$, $\mathcal{F}_{\theta}$ is Fr\'echet differentiable at $q$, and denoting the
Fr\'echet derivative by $\mathcal{F}^{\prime}_{\theta}[q]:L^{2}(Q) \to L^{2}(\mathbb{S}^{1})$, the mapping $q \in L^{\infty}_{+}(Q) \mapsto \mathcal{F}^{\prime}_{\theta}[q] \in \mathcal{L}(L^{2}(Q), L^{2}(\mathbb{S}^{1}))$ is continuous, and its derivative $\mathcal{F}^{\prime}_{\theta}[q]$ at $q$ is given by
\begin{equation}
\mathcal{F}^{\prime}_{\theta}[q]m=v^{\infty}_{q,m}, \label{2.2}
\end{equation} 
where $v^{\infty}_{q,m}$ is the far field pattern of the radiating solution $v=v_{q,m}$ such that
\begin{equation}
\Delta v+k^2(1+q)v=-k^{2}mu_{q}(\cdot, \theta) \ \mathrm{in} \ \mathbb{R}^2. \label{2.3}
\end{equation}
\item[(2)] $\mathcal{F}^{\prime}_{\theta}[\cdot]$ is locally bounded.
\end{description}
\end{lem}
\begin{proof}
First, we recall that from Lemmas 2.2, 2.3, 2.4, and 2.6 in \cite{Bao2005InverseMS}, there exists $C_{1},...,C_{5} >0$ depending on $k$ and $Q$ such that 
\begin{equation}
\left\|  u_{q}(\cdot, \theta) \right\|_{L^{2}(Q)} \leq C_{1}(1+\left\| q \right\|_{L^{\infty}(Q)}), \label{FD_1}
\end{equation}
\begin{equation}
\left\|  u_{q_1}(\cdot, \theta)-u_{q_2}(\cdot, \theta) \right\|_{L^{2}(Q)} \leq C_{2}(1+\left\| q_{2} \right\|_{L^{\infty}(Q)})\left\| q_{1}-q_{2} \right\|_{L^{\infty}(Q)}, \label{FD_2}
\end{equation}
\begin{equation}
\left\|  u_{q+m}(\cdot, \theta)-u_{q}(\cdot, \theta) -v_{q,m}(\cdot, \theta) \right\|_{L^{2}(Q)} \leq C_{3}(1+\left\| q \right\|_{L^{\infty}(Q)})\left\| m \right\|^{2}_{L^{\infty}(Q)}, \label{FD_3}
\end{equation}
\begin{equation}
\left\|  v_{q,m}(\cdot, \theta) \right\|_{L^{2}(Q)} \leq C_{4}(1+\left\| q \right\|_{L^{\infty}(Q)})\left\| m \right\|_{L^{\infty}(Q)}, \label{FD_4}
\end{equation}
\begin{equation}
\left\| v_{q_{1},m}(\cdot, \theta)-v_{q_{2},m}(\cdot, \theta) \right\|_{L^{2}(Q)} \leq C_{5}(1+\left\| q_{2} \right\|_{L^{\infty}(Q)})\left\| q_{1}-q_{2} \right\|_{L^{\infty}(Q)}\left\| m \right\|_{L^{\infty}(Q)}. \label{FD_5}
\end{equation}
{\bf(1)} Concerning Fr\'echet differentiability, by using (\ref{FD_2}) and (\ref{FD_3})
\begin{equation}
\begin{split}
& \left| u^{\infty}_{q+m}(\hat{x}, \theta)-u^{\infty}_{q}(\hat{x}, \theta) -v^{\infty}_{q,m}(\hat{x}, \theta) \right| \\
& \le 
\left| \frac{k^{2}}{4\pi} \int_{Q}\left\{ u_{q+m}(y, \theta)(q(y)-m(y))-u_{q}(y, \theta)q(y) -u_{q}(y, \theta)m(y) -v_{q,m}(y, \theta)q(y)\right\} dy  \right| \\
& \le 
C_{6}\big\{ \left\| u_{q+m}(\cdot, \theta) -u_{q}(\cdot, \theta) \right\|_{L^{2}(Q)}\left\| m \right\|_{L^{\infty}(Q)} 
\\
& \hspace{4cm} + \left\| u_{q+m}(\cdot, \theta) -u_{q}(\cdot, \theta) -v_{q,m}(\cdot, \theta) \right\|_{L^{2}(Q)}\left\| q \right\|_{L^{\infty}(Q)} \big\}
\\
& \le C_{7} (1+\left\| q \right\|_{L^{\infty}(Q)})\left\| m \right\|^{2}_{L^{\infty}(Q)}. \label{Frechet diff for far field}
\end{split}
\end{equation}
Concerning the continuity of the derivative, by using (\ref{FD_2}), (\ref{FD_4}), and (\ref{FD_5})
\begin{equation}
\begin{split}
& \left| v^{\infty}_{q_{1},m}(\hat{x}, \theta)-v^{\infty}_{q_{2},m}(\hat{x}, \theta) \right| \\
& \le 
\left| \frac{k^{2}}{4\pi} \int_{Q}\left\{ u_{q_{1}}(y, \theta)m(y)+v_{q_{1},m}(y, \theta)q_{1}(y) -u_{q_{2}}(y, \theta)m(y) -v_{q_{2},m}(y, \theta)q_{2}(y)\right\} dy  \right| \\
& \le 
C_{8}\big\{ \left\| u_{q_{1}}(\cdot, \theta) -u_{q_{2}}(\cdot, \theta) \right\|_{L^{2}(Q)}\left\| m \right\|_{L^{\infty}(Q)} + \left\| v_{q_{1},m} \right\|_{L^{2}(Q)} \left\| q_{1} -q_{2} \right\|_{L^{\infty}(Q)}
\\
& \hspace{4cm}  + \left\| v_{q_{1},m}(\cdot, \theta) -v_{q_{2},m}(\cdot, \theta)  \right\|_{L^{2}(Q)}\left\| q_{2} \right\|_{L^{\infty}(Q)} \big\} 
\\ 
& \le C_{9} (1+\left\| q_{1} \right\|_{L^{\infty}(Q)}+\left\| q_{2} \right\|_{L^{\infty}(Q)} )\left\| m \right\|_{L^{\infty}(Q)} \left\| q_{1}-q_{2} \right\|_{L^{\infty}(Q)}
\end{split}
\end{equation}
{\bf(2)} Concerning the local boundedness, by using (\ref{FD_1}) and (\ref{FD_4})
\begin{equation}
\begin{split}
& \left| v^{\infty}_{q,m}(\hat{x}, \theta) \right| \\
& \le 
C_{10}\big\{ \left\| u_{q}(\cdot, \theta) \right\|_{L^{2}(Q)}\left\| m \right\|_{L^{\infty}(Q)} + \left\| v_{q,m} \right\|_{L^{2}(Q)} \left\| q \right\|_{L^{\infty}(Q)}\big\} 
\\ 
& \le C_{11} (1+\left\| q \right\|_{L^{\infty}(Q)})^{2}\left\| m \right\|_{L^{\infty}(Q)}
\end{split}
\end{equation}
\end{proof}
We observe the integral kernel of the linear operator $\mathcal{F}^{\prime}_{\theta}[q]$.
The far field pattern $v^{\infty}=v^{\infty}(\cdot, \theta)$ is of the form 
\begin{equation}
v^{\infty}_{q,m}(\hat{x}, \theta)= \frac{k^2}{4\pi}\int_{Q}\mathrm{e}^{-ik \hat{x} \cdot y}\left[ m(y)u_{q}(y,\theta)+q(y)v(y, \theta)\right]dy. \label{2.4}
\end{equation}
Here, we denote the fundamental solution for $-\Delta -k^2(1+q)$ by $\Phi_{q}(x,y)$, which is of the form
\begin{equation}
\Phi_{q}(x, y)=\Phi(x, y) + w(x, y), \ x \neq y, \label{2.5}
\end{equation}
where $w=w(\cdot, y)$ is the unique solution of the following integral equation
\begin{equation}
w(x, y)=k^2\int_{Q}\Phi(x,z)q(z)\left( w(z, y)+\Phi(z,y) \right)dz, \ x \in \mathbb{R}^2. \label{2.6}
\end{equation}
By using the fundamental solution $\Phi_q$, the radiating solution $v=v(\cdot, \theta)$ can be of the form
\begin{equation}
v(x, \theta)=k^2\int_{Q}\Phi_{q}(x,y)m(y)u_{q}(y, \theta)dy, \ x \in \mathbb{R}^2. \label{2.7}
\end{equation}
By combining (\ref{2.4}) and (\ref{2.7}), and using the Fubini's theorem, we conclude that
\begin{equation}
\mathcal{F}^{\prime}_{\theta}[q]m(\hat{x}) = \frac{k^2}{4\pi}\int_{Q} K_{q}(\hat{x}, y) u_{q}(y, \theta)m(y)dy, \  \hat{x} \in \mathbb{S}^1, \label{2.8}
\end{equation}
where the function $K_{q}$ is defined by 
\begin{equation}
K_{q}(\hat{x}, y) := \mathrm{e}^{-ik \hat{x} \cdot y}+k^{2}\int_{Q}\mathrm{e}^{-ik \hat{x} \cdot z} q(z)\Phi_{q}(z, y) dz. \label{2.9}
\end{equation}
We denote the far field mappings $\mathcal{F}: L^{\infty}_{+}(Q) \subset L^{2}(Q) \to L^{2}(\mathbb{S}^{1}\times \mathbb{S}^{1})$ by $\mathcal{F}q(\hat{x}, \theta):=\mathcal{F}_{\theta}q(\hat{x})$, and from (\ref{Frechet diff for far field}), the mapping $\mathcal{F}: L^{\infty}_{+}(Q) \to L^{2}(\mathbb{S}^{1}\times \mathbb{S}^{1})$ is also $C^{1}$. 
The following lemma is proved by the same argument in Section 11 of \cite{ColtonKress} and Section 2 of  \cite{Bao2005InverseMS}.
\begin{lem}\label{prop of Frechet}
\begin{description}
\item[(1)] $\mathcal{F}: \{q \in L^{\infty}(Q): \mathrm{Im}q \geq 0 \ \mathrm{a.e} \ \mathrm{on} \ Q \} \to L^{2}(\mathbb{S}^{1} \times \mathbb{S}^{1})$ is injective.
\item[(2)] $\mathcal{F}\in C^{1}(L^{\infty}_{+}(Q), L^{2}(\mathbb{S}^{1} \times \mathbb{S}^{1}))$, and its derivative  $\mathcal{F}^{\prime}[q]: L^{2}(Q) \to L^{2}(\mathbb{S}^{1} \times \mathbb{S}^{1})$ at $q$ is injective.
\end{description}
\end{lem}
By Theorem 2.1 of \cite{BOURGEOIS2013187}, we have the following stability.
\begin{lem}\label{Lemma 2.3}
Let $W$ be a finite dimensional
subspace of $L^{2}(Q)$, and Let $K$ be a compact and convex subset of $W \cap L^{\infty}_{+}(Q)$. Then, there exists a constant $C>0$ such that
\begin{equation}
\left\| q_{1} - q_{2} \right\|_{L^{2}(Q)} \leq C \left\| \mathcal{F}q_{1} - \mathcal{F}q_{2} \right\|_{L^{2}(\mathbb{S}^{1} \times \mathbb{S}^{1})}, \ q_{1}, q_{2} \in K.
\end{equation}
\end{lem}
Let $\{ \theta_{i}: i \in \mathbb{N} \}$ be dense in $\mathbb{S}^{1}$. We denote $\vec{\mathcal{F}}_{N}:L^{2}(Q) \to L^{\infty}(\mathbb{S}^{1})^{N}$ by  $\vec{\mathcal{F}}_{N}(q):=\left(
\begin{array}{cc}
\mathcal{F}_{\theta_{1}}q \\
\vdots \\
\mathcal{F}_{\theta_{N}}q
\end{array}
\right)$. 
Following lemma is proved by the same argument in Theorem 7 of \cite{alberti2020infinitedimensional}.
\begin{lem}\label{Lipchotz stability}
Let $W$ be a finite dimensional
subspace of $L^{2}(Q)$, and Let $K$ be a compact and convex subset of $W \cap L^{\infty}_{+}(Q)$.
Then, for large $N \in \mathbb{N}$, there exists a constant $C_{N}>0$ such that
\begin{equation}
\left\| q_{1} - q_{2} \right\|_{L^{2}(Q)} \leq C_{N} \left\| \vec{\mathcal{F}}_{N}q_{1} - \vec{\mathcal{F}}_{N}q_{2} \right\|_{L^{2}(\mathbb{S}^{1})^{N}}, \ q_{1}, q_{2} \in K.
\end{equation}
\end{lem}
\begin{proof}
We follow the proof of Theorem 7 of \cite{alberti2020infinitedimensional}.
We first remark that the far-field patterns are analytic functions in $\mathbb{S}^{1} \times \mathbb{S}^{1}$. 
Since $\mathrm{dim}(\mathbb{S}^{1})=1$, and $1>\frac{1}{2}$, $H^{1}(\mathbb{S}^{1})$ is continuously embedded into $C(\mathbb{S}^{1})$, and so it is a reproducing kernel Hilbert space consisting of continuous functions. 
Let $P_{G_{N}}:H^{1}(\mathbb{S}^{1}) \to H^{1}(\mathbb{S}^{1})$ be the projection onto $G_{N}$ defined by
\begin{equation}
G_{N}:= \mathrm{span}\{ k_{\theta_{i}}: i=1,...,N \}
\end{equation}
where the basis $k_{\theta_{i}}$ is unique solution of $f(\theta_{i})=\langle f, k_{\theta_{i}} \rangle_{H^{1}(\mathbb{S}^{1})}$.
Then by Example 2 of \cite{alberti2020infinitedimensional}, $P_{G_{N}} \to I_{H^{1}(\mathbb{S}^{1})}$ as $N \to \infty$ and
\begin{equation}
\left\| P_{G_{N}}f \right\|_{H^{1}(\mathbb{S}^{1})} \leq C_{N} \left\| (f(\theta_{1}),...,f(\theta_{N})) \right\|_{2}. \label{projection ineq}
\end{equation}
Let $P_{L^{2}(\mathbb{S}^{1}; H^{1}(\mathbb{S}^{1}) )}: L^{2}(\mathbb{S}^{1} \times \mathbb{S}^{1}) \to L^{2}(\mathbb{S}^{1} \times \mathbb{S}^{1})$ be the projection onto the Bochner space $L^{2}(\mathbb{S}^{1}; H^{1}(\mathbb{S}^{1}) ) \subset L^{2}(\mathbb{S}^{1} \times \mathbb{S}^{1})$.
We define the bounded linear operator $Q_{N}: L^{2}(\mathbb{S}^{1} \times \mathbb{S}^{1}) \to L^{2}(\mathbb{S}^{1} \times \mathbb{S}^{1})$ by $Q_{N}g(\hat{x},\theta)=P_{G_{N}}[P_{L^{2}(\mathbb{S}^{1}; H^{1}(\mathbb{S}^{1}))}g(\hat{x}, \cdot)](\theta)$. 
Then, $Q_{N}\big|_{L^{2}(\mathbb{S}^{1}; H^{1}(\mathbb{S}^{1}) )} \to I_{L^{2}(\mathbb{S}^{1}; H^{1}(\mathbb{S}^{1}) )}$.
By Lemmas \ref{Lemma 2.1} and \ref{Lemma 2.3}, we can apply Theorem 2 of \cite{alberti2020infinitedimensional} (as $T=L^{2}(\mathbb{S}^{1} \times \mathbb{S}^{1})$, $\tilde{Y}=L^{2}(\mathbb{S}^{1}; H^{1}(\mathbb{S}^{1}))$ ), which implies that there exists a large $N \in \mathbb{N}$ such that by using (\ref{projection ineq})
\begin{equation}
\begin{split}
& \left\| q_{1}-q_{2} \right\|  \le 
C\left\| Q_{N} u^{\infty}_{q_{1}} -Q_{N}u^{\infty}_{q_{2}} \right\|_{L^{2}(\mathbb{S}^{1} \times \mathbb{S}^{1})}\le C \left\| P_{G_N} u^{\infty}_{q_{1}} -P_{G_N}u^{\infty}_{q_{2}} \right\|_{L^{2}(\mathbb{S}^{1}; H^{1}(\mathbb{S}^{1})) }
\\ 
& = C \left\{ \int_{\mathbb{S}^{1}}\left\| P_{G_N} [u^{\infty}_{q_{1}}(\hat{x}, \cdot)] -P_{G_N} [u^{\infty}_{q_{2}}(\hat{x}, \cdot)] \right\|^{2}_{H^{1}(\mathbb{S}^{1})} ds(\hat{x}) \right\}^{1/2}
\\ 
& \le C_{N} \left\{ \sum_{i=1}^{N} \int_{\mathbb{S}^{1}}\left| u^{\infty}_{q_{1}}(\hat{x}, \theta_{i}) - u^{\infty}_{q_{2}}(\hat{x}, \theta_{i}) \right|^{2} ds(\hat{x}) \right\}^{1/2} \leq C_{N} \left\| \vec{\mathcal{F}}_{N}q_{1} - \vec{\mathcal{F}}_{N}q_{2} \right\|_{L^{2}(\mathbb{S}^{1})^{N}}
\end{split}
\end{equation}
for $q_{1}, q_{2} \in K$.
\end{proof}
\section{Linearized problems}\label{Linearized problems}
In this section, we consider the linearized problem for nonlinear equation, and recall Levenberg–Marquardt method.
Let $X$ and $Y$ be Hilbert spaces over complex variables $\mathbb{C}$ which correspond to the state space $L^{2}(Q)$ of the inhomogeneous medium function $q$, and the observation space $L^{2}(\mathbb{S}^{1})$ of the far field pattern $u^{\infty}$, respectively. 
Let $A:X \to Y$ be a nonlinear observation operator which corresponds to the far field mapping $\mathcal{F}$. 
\par
For give $f \in Y$, we seek the solution $\varphi \in X$ such that
\begin{equation}
A(\varphi) = f. \label{3.1}
\end{equation} 
We assume that we have an initial guess $\varphi_0 \in X$, which is a starting point of the algorithm, and is appropriately determined by a priori information of the true solution $\varphi^{true}$ of (\ref{3.1}). 
We also assume that the nonlinear mapping $A$ is Fr\'echet differentiable at $\varphi_0$, which implies that 
\begin{equation}
A(\varphi) = A(\varphi_0) + A^{\prime}[\varphi_0](\varphi - \varphi_0) + r(\varphi - \varphi_0), \label{3.2}
\end{equation}
where the linear bounded operator $A^{\prime}[\varphi_0]:X \to Y$ is the Fr\'echet derivative of the nonlinear mapping $A$ at $\varphi_0$, and $r:X \to Y$ is some mapping corresponding to the remainder term such that $r(h)=o(h)$ as $\left\| h \right\|  \to 0$. 
In the case to seek the solution $\varphi$ close to the initial guess $\varphi_0$, we can omit the remainder term $r$ because its influence is small. 
Then, we have the following linearized problem of (\ref{3.1}).
\begin{equation}
A^{\prime}[\varphi_0](\varphi - \varphi_0) = f-A(\varphi_0). \label{3.3}
\end{equation}
Although the problem become linear, the equation (\ref{3.3}) may be ill-posed because the Fr\'echet derivative $A^{\prime}[\varphi_0]$ of $A$ is not generally invertible. 
Then, by the regularization method (see e.g., Chapter 4 of \cite{ColtonKress} and Chapter 3 of \cite{NakamuraPotthast}), we have the regularized solution $\varphi_{\alpha}$ of (\ref{3.3})
\begin{equation}
\varphi_{\alpha}:=\varphi_0 + \left(\alpha I + A^{\prime}[\varphi_0]^{*}A^{\prime}[\varphi_0]  \right)^{-1}A^{\prime}[\varphi_0]^{*}\left( f - A(\varphi_0) \right), \label{3.8}
\end{equation}
where $\alpha>0$ is a regularization parameter. 
Furthermore, we have an iterative algorithm for $i \in \mathbb{N}_{0}$
\begin{equation}
\varphi_{i+1}=\varphi_{i} + \left(\alpha_{i} I + A^{\prime}[\varphi_{i}]^{*}A^{\prime}[\varphi_{i}]  \right)^{-1}A^{\prime}[\varphi_{i}]^{*}\left( f - A(\varphi_{i}) \right), \label{3.9}
\end{equation}
which is known as the {\it Levenberg–Marquardt} method (see e.g., \cite{Hanke_1997, Kaltenbacher}). 
So far, many type of the Newton method have been studied, for example, the {\it regularized Gauss--Newton method} (see e.g.,\cite{Bakushinsky}) and the {\it Quasi--Newton method} (see e.g., \cite{Nocedal}), and for any other, we refer to \cite{Hohage, Kaltenbacher, potthast2001convergence, Xiao}. 
We remark that the regularization parameter $\alpha_i>0$ in (\ref{3.9}) is chosen such that the morozov discrepancy principle:
\begin{equation}
\left\| f - A(\varphi_{i})- A^{\prime}[\varphi_i] (\varphi_{i+1}(\alpha) - \varphi_{i}) \right\| = \rho \left\| f - A(\varphi_{i}) \right\|, \label{morozov discrepancy principle}
\end{equation}
with some fixed $\rho \in (0, 1)$, where $\varphi_{i+1}(\alpha)$ is defined as in (\ref{3.9}) replacing $\alpha_{i}$ by $\alpha$. 
Following lemma is the convergence. 
\begin{lem}[Theorem 4.2 of \cite{Kaltenbacher}, or Theorem 2.2 of \cite{Hanke_1997}] \label{convergence LM}
Let $0<\rho<1$ and assume that (\ref{3.1}) is solvable in $B_{r}(\varphi_{0})$ where $r>0$ is some constant, and let $\varphi^{\dag} \in B_{r}(\varphi_{0})$ be its solution, i.e., $f=A(\varphi^{\dag})$, and assume that $A^{\prime}[\cdot]$ is uniformly bounded in $B_{r}(\varphi_{0})$, and a {\it tangential cone condition}: there exists $C>0$ such that for $\varphi, \widetilde{\varphi} \in B_{2r}(\varphi_{0})$
\begin{equation}
\left\| A(\varphi) - A(\widetilde{\varphi}) - A^{\prime}[\varphi] (\varphi - \widetilde{\varphi}) \right\| \leq C\left\|\varphi - \widetilde{\varphi}\right\| \left\|A(\varphi) - A(\widetilde{\varphi}) \right\|. \label{tangential cone condition}
\end{equation}
Then, if $\left\| \varphi_{0} - \varphi^{\dag} \right\| < \rho / C$, then the Levenberg–Marquardt method $\varphi_{i}$ with $\{ \alpha_{i} \}$ determined from (\ref{morozov discrepancy principle}) converges to a solution of $f=A(\varphi)$ as $i \to \infty$.
\end{lem}
\section{Levenberg–Marquardt and Kalman filter}\label{Levenberg Marquardt and Kalman filter}
The natural approach for solving the equation (\ref{1.7}) is to put all available measurements $\{ u^{\infty}_{n} \}_{n=1}^{N}$ and all far field mappings $\{ \mathcal{F}_{n} \}_{n=1}^{N}$ where the index $n$ is associated with the incident direction $\theta_n \in \mathbb{S}^{1}$ into one long vector $\vec{u}^{\infty}$ and $\vec{\mathcal{F}}$, respectively, and to employ the Levenberg–Marquardt method (\ref{3.9}) discussed in the Section 3.
In order to study the above general situation, let $f_1,..., f_N \in Y$ be measurements, let $A_1,...,A_N$ be nonlinear observation operators, and let us consider the problem to determine $\varphi \in X$ such that 
\begin{equation}
\vec{A}(\varphi)=\vec{f}, \label{4.1}
\end{equation}
where  $\vec{f}:=\left(
    \begin{array}{cc}
      f_1 \\
      \vdots \\
      f_N
    \end{array}
  \right)$, and $\vec{A}(\varphi):=\left(
    \begin{array}{cc}
      A_1(\varphi) \\
      \vdots \\
      A_N(\varphi)
    \end{array}
  \right)$. 
By applying the Levenberg–Marquardt method (\ref{3.9}) to the above system (\ref{4.1}), we have iterative solution 
\begin{equation}
\varphi^{FLM}_{i+1} := \varphi^{FLM}_{i} + \left(\alpha_{i} I + \vec{A}^{\prime}[\varphi^{FLM}_{i}]^{*}\vec{A}^{\prime}[\varphi^{FLM}_{i}]\right)^{-1}\vec{A}^{\prime}[\varphi^{FLM}_{i}]^{*}\left(\vec{f} - \vec{A}(\varphi^{FLM}_{i}) \right), \label{4.2}
\end{equation}
where $\varphi^{FLM}_{0}:=\varphi_{0}$, and $\vec{A}^{\prime}[\varphi]$ is denoted by $\vec{A}^{\prime}[\varphi]=\left(
    \begin{array}{cc}
      A^{\prime}_{1}[\varphi] \\
      \vdots \\
      A^{\prime}_N[\varphi]
    \end{array}
  \right)$, 
and the regularization parameters $\alpha_i>0$ satisfies the morozov discrepancy principle (\ref{morozov discrepancy principle}).
We call this the {\it Full data Levenberg–Marquardt}. 
Here, $\vec{A}^{\prime}[\varphi_{0}]^{*}$ is a adjoint operator of $\vec{A}^{\prime}[\varphi_{0}]$ with respect to the usual scalar product $\langle \cdot, \cdot \rangle_{X}$ and the weighted scalar product $\langle \cdot, \cdot \rangle_{Y^{N}, R^{-1}}:=\langle \cdot, R^{-1}\cdot \rangle_{Y^{N}}$ where $R: Y \to Y$ is the positive definite symmetric invertible operator, which is interpreted as the covariance matrices of the observation error distribution from a statistical viewpoint in the case when $X$ and $Y$ are Euclidean spaces (see Section 5 of \cite{Furuya}). 
\par
By the same calculation in (3.6) of \cite{Furuya}, we have
\begin{equation}
\vec{A}^{\prime}[\varphi]^{*}=\vec{A}^{\prime}[\varphi]^{H} R^{-1},
\end{equation}
where $\vec{A}^{\prime}[\varphi]^{H}$ is a adjoint operator of $\vec{A}^{\prime}[\varphi]$ with respect to usual scalar products $\langle \cdot, \cdot \rangle_{X}$ and $\langle \cdot, \cdot \rangle_{Y^{N}}$. 
Then, (\ref{4.2}) can be of the form
\begin{equation}
\varphi^{FLM}_{i+1} = \varphi^{FLM}_{i} + \left(\alpha_{i} I + \vec{A}^{\prime}[\varphi^{FLM}_{i}]^{H} R^{-1}\vec{A}^{\prime}[\varphi^{FLM}_{i}]\right)^{-1}\vec{A}^{\prime}[\varphi^{FLM}_{i}]^{H} R^{-1}\left(\vec{f} - \vec{A}(\varphi^{FLM}_{i}) \right), \label{4.4}
\end{equation}
\par
\begin{rem}\label{Remark1}
Let go back to our scattering problem, which $\vec{A}$ corresponds to the far field mapping $\vec{\mathcal{F}}$. 
From Lemma \ref{prop of Frechet}, $\vec{\mathcal{F}}$ is $C^{1}$, and its derivative is locally bounded.
Furthermore, from Lemma \ref{Lipchotz stability}, we have Lipschitz stability on compact convex subset of the finite dimensional subspace, which satisfies a tangential cone condition (\ref{tangential cone condition}).
Therefore by Lemma \ref{convergence LM}, our solution $q^{FLM}_{i}$ converges to true solution in the finite dimensional subspace if the initial geuss $q_{0}$ is very close to true one. 
\end{rem}
However, the algorithm (\ref{4.2}) of the Full  data Levenberg–Marquardt is computationally expensive when the number $N$ of measurements is increasing in which we have to construct the bigger system $\vec{A}\varphi= \vec{f}$.
So, let us consider the alternative approach based on the Kalman filter. The Kalman filter is the linear estimation for the unknown state by the update of the state and its norm using the sequential measurements.
For details of the following derivation, we refer to the first part of our works \cite{Furuya}. 
\par
We consider the following problem for $n=1,...,N$
\begin{equation}
A^{\prime}_{n}[\varphi_0]\varphi = f_{n} - A_{n}(\varphi_0) + A^{\prime}_{n}[\varphi_0]\varphi_0,  \label{4.5}
\end{equation}
which arises from the linearization of the problem $A_{n}(\varphi)=f_n$ at the initial guess $\varphi_0$. 
The above problem (\ref{4.5}) can be applied to the Kalman filter algorithm (see (4.21)--(4.23) in \cite{Furuya}), then we obtain the following algorithm for $n=1,...,N$.
\begin{equation}
\varphi_{0, n}:= \varphi_{0, n-1} + K_{0, n}\left( f_{n} - A_{n}(\varphi_{0,0}) + A^{\prime}_{n}[\varphi_{0,0}]\varphi_{0,0} - A^{\prime}_{n}[\varphi_{0,0}] \varphi_{0, n-1} \right),  \label{4.6}
\end{equation}
\begin{equation}
K_{0, n}:= B_{0, n-1} A^{\prime}_{n}[\varphi_{0,0}]^{H}\left(R + A^{\prime}_{n}[\varphi_{0,0}] B_{0, n-1} A^{\prime}_{n}[\varphi_{0,0}]^{H} \right)^{-1},  \label{4.7}
\end{equation}
\begin{equation}
B_{0, n}:= \left(I - K_{0, n} A^{\prime}_{n}[\varphi_{0,0}]^{H} \right)B_{0, n-1},  \label{4.8}
\end{equation}
where $\varphi_{0, 0}:=\varphi_0$, and $B_{0,0}:=\frac{1}{\alpha_{0}}I$. 
We denote the final state and covariance matrix in (\ref{4.6}) and (\ref{4.8}) by $\varphi_{1, 0}:=\varphi_{0, N}$ and $B_{1,0}:=\frac{1}{\alpha_{1}}I$, which is the initial guess of the next iteration. 
\par
Next, we consider the following problem
\begin{equation}
A^{\prime}_{n}[\varphi_{1, 0}]\varphi = f_{n} - A_{n}(\varphi_{1, 0}) + A^{\prime}_{n}[\varphi_{1, 0}]\varphi_{1, 0}, \label{4.9}
\end{equation}
which arises from the linearization of the problem $A_{n}(\varphi)=f_n$ at $\varphi_{1, 0}$. 
The above problem (\ref{4.9}) can be applied to the Kalman filter algorithm as well, and we obtain the following algorithm for $n=1,...,N$.
\begin{equation}
\varphi_{1, n}:= \varphi_{1, n-1} + K_{1, n}\left( f_{n} - A_{n}(\varphi_{1,0}) + A^{\prime}_{n}[\varphi_{1,0}]\varphi_{1,0} - A^{\prime}_{n}[\varphi_{1,0}] \varphi_{1, n-1} \right),  
\end{equation}
\begin{equation}
K_{0, n}:= B_{0, n-1} A^{\prime}_{n}[\varphi_{1,0}]^{H}\left(R + A^{\prime}_{n}[\varphi_{1,0}] B_{0, n-1} A^{\prime}_{n}[\varphi_{1,0}]^{H} \right)^{-1},  
\end{equation}
\begin{equation}
B_{0, n}:= \left(I - K_{0, n} A^{\prime}_{n}[\varphi_{1,0}]^{H} \right)B_{0, n-1},  
\end{equation}
\par
We can repeat these procedure, then we obtain the following algorithm for $i \in \mathbb{N}_{0}$ and $n=1,...,N$
\begin{equation}
\varphi^{KFL}_{i, n}:= \varphi^{KFL}_{i, n-1} + K_{i, n}\left( f_{n} - A_{n}(\varphi^{KFL}_{i, 0}) + A^{\prime}_{n}[\varphi^{KFL}_{i, 0}]\varphi^{KFL}_{i, 0} - A^{\prime}_{n}[\varphi^{KFL}_{i, 0}] \varphi^{KFL}_{i, n-1} \right),  \label{4.10}
\end{equation}
\begin{equation}
K_{i, n}:= B_{i, n-1} A^{\prime}_{n}[\varphi^{KFL}_{i, 0}]^{H}\left(R + A^{\prime}_{l}[\varphi^{KFL}_{i, 0}] B_{i, n-1} A^{\prime}_{n}[\varphi^{KFL}_{i, 0}]^{H} \right)^{-1},  \label{4.11}
\end{equation}
\begin{equation}
B_{i, n}:= \left(I - K_{i, n} A^{\prime}_{n}[\varphi^{KFL}_{i, 0}]^{H} \right)B_{i, n-1}.  \label{4.12}
\end{equation}
When the iteration time $i$ is raised by one, the final state is renamed as 
\begin{equation}
\varphi^{KFL}_{i, 0}:= \varphi^{KFL}_{i-1, N},  \label{4.13}
\end{equation}
and the weight is initialized as
\begin{equation}
B_{i, 0}:= \frac{1}{\alpha_{i}}I, \label{4.14}
\end{equation}
where the regularization parameters $\alpha_i>0$ satisfies the morozov discrepancy principle (\ref{morozov discrepancy principle}). 
We call this the {\it Kalman filter Levenberg–Marquardt}.
We remark that the algorithm has two indexes $i$ and $n$, where $i$ is associated with the iteration step, and $n$ measurement step, respectively.
\par
Finally in this section, we show the following equivalent theorem, which is the nonlinear iteration version of Theorem 4.3 in \cite{Furuya}.
\begin{thm}\label{equivalence for KFN and FTN}
For measurements $f_1,...,f_N$, nonlinear mappings $A_1,...,A_N$, and the initial guess $\varphi_0 \in X$, and the initial regularization parameter $\alpha_{0}>0$, the Kalman filter Levenberg–Marquardt given by (\ref{4.10})--(\ref{4.14}) is equivalent to the Full data Levenberg–Marquardt given by (\ref{4.4}), that is, we have
\begin{equation}
\varphi^{KFL}_{i, N}=\varphi^{FLM}_{i+1},  \label{4.15}
\end{equation}
for all $i \in \mathbb{N}_0$.
\end{thm}
\begin{proof}
We will prove (\ref{4.15}) by the induction.
By applying Theorem 4.3 of \cite{Furuya} to the linearized problem $A^{\prime}_{n}[\varphi_0]\varphi = f_{n} - A_{n}(\varphi_0) + A^{\prime}_{n}[\varphi_0]\varphi_0$ for $n=1,...,N$ with the initial guess $\varphi_0$ and the regularization parameter $\alpha_0>0$, we have $\varphi^{KFL}_{0, N}=\varphi^{FLM}_{1}$, which is the case of $i=0$.
\par
Let us assume that (\ref{4.15}) in the case of $i-1$ holds, that is, we have $\varphi^{KFL}_{i-1, N}(=\varphi^{KFL}_{i, 0})=\varphi^{FLM}_{i}=:\varphi_{i}$.
Again, we apply Theorem 4.3 of \cite{Furuya} to the linearized problem $A^{\prime}_{n}[\varphi_i]\varphi = f_{n} - A_{n}(\varphi_i) + A^{\prime}_{n}[\varphi_i]\varphi_i$ for $n=1,...,N$ with the initial guess $\varphi_i=\varphi^{KFL}_{i, 0}=\varphi^{FLM}_{i}$ and the regularization parameter $\alpha_i>0$, then we have $\varphi^{KFL}_{i, N}=\varphi^{FLM}_{i+1}$.
Theorem \ref{equivalence for KFN and FTN} has been shown.
\end{proof}
\section{Iterative Extended Kalman filter}\label{Iterative Extended Kalman filter}
The usual Kalman filter is the linear optimal estimation for solving the linear system.
However in realistic applications, most systems are nonlinear, so many studies of the nonlinear estimation have been done.
The {\it Extended Kalman filter}, which is one of the nonlinear version of the Kalman filter, is to apply the linear Kalman filter to the linearized equation at the current state for every time to observe one measurement.
For further readings of the Extended Kalman filter, we refer to \cite{Grewal, grewal2010applications, Jazwinski}, and there also exists other types of the nonlinear Kalman filter such as the Unscented Kalman Filter (\cite{Julier}) which based on the Monte Carlo sampling without employing the linearization approximation. 
In this section, we introduce the algorithm based on the Extended Kalman filter.
\par
First, let us start with the linearized problem of $A_{1}(\varphi)=f_{1}$ at the initial guess $\varphi_0$.
\begin{equation}
A^{\prime}_{1}[\varphi_0]\varphi = f_{1} - A(\varphi_0) + A^{\prime}_{1}[\varphi_0]\varphi_0. \label{5.1}
\end{equation}
By the same argument in Section 4 of \cite{Furuya} replacing $A_1$ and $f_1$ by $A^{\prime}_{1}[\varphi_0]$ and $f_{1} - A(\varphi_0) + A^{\prime}_{1}[\varphi_0]\varphi_0$, respectively, we have the following solution of (\ref{5.1}).
\begin{equation}
\varphi_{1}:= \varphi_{0} + K_{1}\left( f_{1} - A_{1}(\varphi_{0}) \right), \label{5.2}
\end{equation}
\begin{equation}
K_{1}:= B_{0} A^{\prime}_{1}[\varphi_0]^{H}\left(R + A^{\prime}_{1}[\varphi_0] B_{0} A^{\prime}_{1}[\varphi_0]^{H} \right)^{-1}, \label{5.3}
\end{equation}
\begin{equation}
B_{1}:= \left(I - K_{1} A^{\prime}_{1}[\varphi_0]^{H} \right)B_{0}, \label{5.4}
\end{equation}
where $B_{0}:=\frac{1}{\alpha_0}I$ and $\alpha_0>0$ is an initial regularization parameter. 
\par
Next, we consider linearized problem of $A_{2}(\varphi)=f_{2}$ at $\varphi_{1}$.
\begin{equation}
A^{\prime}_{2}[\varphi_{1}]\varphi = f_{2} - A_{2}(\varphi_{1}) + A^{\prime}_{2}[\varphi_{1}]\varphi_{1}, \label{5.5}
\end{equation}
Then, by the same argument in Section 4 of \cite{Furuya}, we have the following solution of (\ref{5.5}).
\begin{equation}
\varphi_{2}:= \varphi_{1} + K_{2}\left( f_{2} - A_{2}(\varphi_{1}) \right), 
\end{equation}
\begin{equation}
K_{2}:= B_{1} A^{\prime}_{2}[\varphi_1]^{H}\left(R + A^{\prime}_{2}[\varphi_1] B_{2} A^{\prime}_{2}[\varphi_1]^{H} \right)^{-1}, 
\end{equation}
\begin{equation}
B_{2}:= \left(I - K_{2} A^{\prime}_{2}[\varphi_1]^{H} \right)B_{1},
\end{equation}
\par
We can repeat them, then we have the following algorithm.
\begin{equation}
\varphi_{n}:= \varphi_{n-1} + K_{n}\left( f_{n} - A_{n}(\varphi_{n-1}) \right), \label{5.6}
\end{equation}
\begin{equation}
K_{n}:= B_{n-1} A^{\prime}_{n}[\varphi_{n-1}]^{H}\left(R + A^{\prime}_{n}[\varphi_{n-1}] B_{n-1} A^{\prime}_{n}[\varphi_{n-1}]^{H} \right)^{-1}, \label{5.7}
\end{equation}
\begin{equation}
B_{n}:= \left(I - K_{n} A^{\prime}_{n}[\varphi_{n-1}]^{H} \right)B_{n-1}, \label{5.8}
\end{equation}
for $n=1,...,N$. 
In order to obtain the iterative algorithm, we repeat the arguments in the above (\ref{5.1})--(\ref{5.8}) as the initial guess is $\varphi_{N}$.
Finally, we obtain the following iterative algorithm for $i \in \mathbb{N}_{0}$ and $n=1,...,N$.
\begin{equation}
\varphi^{EKF}_{i,n}:= \varphi^{EKF}_{i, n-1} + K_{i, n}\left( f_{n} - A_{n}(\varphi^{EKF}_{i, n-1}) \right), \label{5.9}
\end{equation}
\begin{equation}
K_{i, n}:= B_{i, n-1} A^{\prime}_{n}[\varphi^{EKF}_{i, n-1}]^{H}\left(R + A^{\prime}_{l}[\varphi^{EKF}_{i, n-1}] B_{i, n-1} A^{\prime}_{n}[\varphi^{EKF}_{i, n-1}]^{H} \right)^{-1}, \label{5.10}
\end{equation}
\begin{equation}
B_{i, n}:= \left(I - K_{i, n} A^{\prime}_{n}[\varphi^{EKF}_{i, n-1}]^{H} \right)B_{i, n-1}. \label{5.11}
\end{equation}
When the iteration time $i$ is raised by one, the final state is renamed as 
\begin{equation}
\varphi^{EKF}_{i, 0}:= \varphi^{EKF}_{i-1, N},  \label{5.12}
\end{equation}
and the weight as
\begin{equation}
B_{i, 0}:= B_{i-1,N}.  \label{5.13}
\end{equation}
We call this the {\it iteratively Extended Kalman Filter}.  Figure \ref{KFN and EKF} provides an illustration for the difference of Kalman filter Levenberg–Marquardt (KFL, left) and iterative Extended Kalman filter (EKF, right). 
When the state moves horizontally, measurements are used, and when it moves vertically, linearization are done.
There are differences in term of when to linearize the nonlinear equation, and the number of linearization in EKF is larger than that in KFL.
\par
\begin{rem}\label{Remark2}
By Theorem \ref{equivalence for KFN and FTN}, Kalman filter Levenberg–Marquardt (\ref{4.10})--(\ref{4.14}) is equivalent to Full data Levenberg–Marquardt (\ref{4.4}), which implies that in our scattering problem, Kalman filter Levenberg–Marquardt $q^{KFL}_{n,N}$ converges to the finite dimensional true solution (see Remark \ref{Remark1}). 
Although we will not prove the rigorous convergence, iteratively Extended Kalman filter (\ref{5.9})-(\ref{5.13}) could be expected to have the convergence because there exists several references \cite{guo2002fast, boutayeb1995convergence, krener2002convergence}, which discuss the convergence of Extended Kalman filter in the context of dynamic Kalman filter in the setting of the Euclidean space.
They could be extended to our scattering setting, in particular to infinite dimensional Hilbert spaces over complex variable.
\end{rem}
\begin{figure}[h]
\hspace{0.0cm}
  \includegraphics[keepaspectratio, scale=0.7]
  {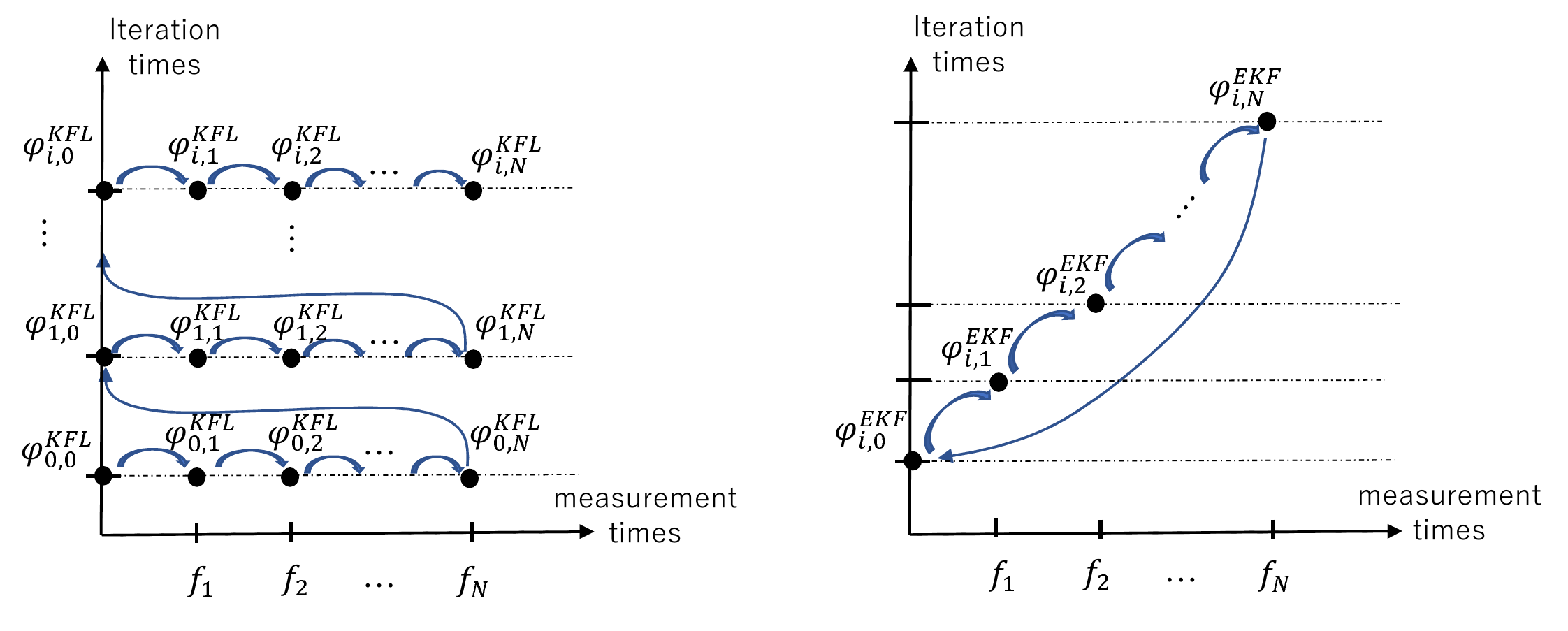}
 \caption{difference of KFL (left) and EKF (right)}
\label{KFN and EKF}
\end{figure}
\section{Numerical examples}\label{Numerical examples}
In this section, we provide numerical examples for the Kalman filter algorithm.
Our inverse scattering problem is to solve the nonlinear integral equation
\begin{equation}
\mathcal{F}_{n}q=u^{\infty}(\cdot, \theta_n), \label{Nonlinear integral equation}
\end{equation}
for $n=1,...,N$ where the operator $\mathcal{F}_{n}:L^{2}(Q) \to L^{\infty}(\mathbb{S}^{1})$ is defined by 
\begin{equation}
\mathcal{F}_{n}q(\hat{x}):=\mathcal{F}q(\hat{x}, \theta_{n}):=\frac{k^2}{4\pi}\int_{Q}\mathrm{e}^{-ik\hat{x} \cdot y}u_{q}(y, \theta_{n}) q(y)dy, 
\end{equation}
where the incident direction is given by $\theta_n:=\left(\mathrm{cos}(2\pi n/N), \mathrm{sin}(2\pi n/N) \right)$ for each $n=1,...,N$. 
Here, $u_{q}(\cdot, \theta_{n})$ is the solution of Lippmann-Schwinger integral equation (\ref{1.4}), which is numerically computed based on Vainikko's method \cite{saranen2001periodic, vainikko2000fast}, which is a fast solution method for the Lippmann–Schwinger equation based on periodization, fast Fourier
transform techniques and multi-grid methods.
We assume that the support of the function $q$ is included in $[-S, S]^2$ with some $S>0$. The linearized problem at $q$ of (\ref{Nonlinear integral equation}) is 
\begin{equation}
\mathcal{F}^{\prime}_{n}[q_{0}](q-q_{0})= u^{\infty}_{n} -  \mathcal{F}_{n}q_{0} + \mathcal{F}^{\prime}_{n}[q_{0}]q_{0}, \label{Lineared Nonlinear}
\end{equation}
where  Fr\'echet derivative $\mathcal{F}^{\prime}_{l}[q_{0}]$ is of the form
\begin{equation}
\mathcal{F}^{\prime}_{n}[q_{0}]q(\hat{x}) = \frac{k^2}{4\pi}\int_{Q} K_{q_{0}}(\hat{x}, y) u_{q_{0}}(y, \theta_{n})q(y)dy, \  \hat{x} \in \mathbb{S}^1. 
\end{equation}
where $K_{q_{0}}(\hat{x}, y)$ is defined by (\ref{2.8}).
\par
The linearized equation (\ref{Lineared Nonlinear}) is discretized by
\begin{equation}
\bm{\mathcal{F}^{\prime}_{n}[q_{0}]} \bm{q} = \bm{u^{\infty}_{n}} -  \bm{\mathcal{F}_{n}q_{0}} + \bm{\mathcal{F}^{\prime}_{n}[q_{0}]q_{0}},
\end{equation}
where
\begin{equation}
\bm{\mathcal{F}^{\prime}_{n}[q_{0}]} = \frac{k^2 S^{2}}{4\pi M^{2}} \left(
K_{q_{0}}(\hat{x}_j, y_{m_{1},m_{2}})u_{q_{0}}(y_{m_{1},m_{2}}, \theta_{n}) \right)_{j=1,...,J, \ -M \leq m_{1},m_{2} \leq M-1} \in \mathbb{C}^{J\times (2M)^2}. 
\end{equation} 
where $y_{m_{1},m_{2}}:=\left( \frac{(2m_{1}+1)S}{2M}, \frac{(2m_{2}+1)S}{2M} \right)$, and $M \in \mathbb{N}$ is a number of the division of $[0,S]$ (i.e., the function $q$ is discretized by piecewise constant on $[-S, S]^{2}$ which is decomposed by squares with the length $\frac{S}{M}$), and $\hat{x}_j:=\left(\mathrm{cos}(2\pi j/J),  \mathrm{sin}(2\pi j/J) \right)$, and $J \in \mathbb{N}$ is a number of the division of $[0,2\pi]$ and 
\begin{equation}
\bm{q} = \left( q(y_{m_{1},m_{2}})  \right)_{-M \leq m_{1},m_{2} \leq M-1} \in \mathbb{C}^{(2M)^2},
\end{equation} 
and
\begin{equation}
\bm{u_{n}^{\infty}} = \left( u^{\infty}(\hat{x}_j, \theta_{n}) \right)_{j=1,...,J}+\bm{\epsilon_{n}}  \in \mathbb{C}^{J}.
\end{equation}
The noise $\bm{\epsilon_{n}} \in \mathbb{C}^{J}$ is sampling from a complex Gaussian distribution $\mathcal{C}\mathcal{N}(0, \sigma^{2}\bm{I})$, which is equivalent to $\bm{\epsilon_{n}}=\bm{\epsilon^{re}_{n}}+i\bm{\epsilon^{im}_{n}}$ where $\bm{\epsilon^{re}_{n}}, \bm{\epsilon^{im}_{n}} \in \mathbb{R}^{J}$ are independently identically distributed from $\mathcal{N}(0, \sigma^{2} \bm{I})$.
\par
Here, we always fix discretization parameters as $J=30$, $M=8$, $S=3$, $N=30$, and weight $\bm{R} \in \mathbb{R}^{J \times J}$, which is the covariance matrix of the observation error distribution, as $R=r^{2}I$, and $r=3$. 
\par 
We consider true functions as the characteristic function 
\begin{equation}
q^{true}_{j}(x):=\left\{ \begin{array}{ll}
0.1 & \quad \mbox{for $x \in B_j$}  \\
0 & \quad \mbox{for $x \notin B_j$}
\end{array} \right.,  \label{6.6}
\end{equation} 
where the support $B_j$ of the true function is considered as the following two types.
\begin{equation} 
B_1:=\left\{(x_1, x_2) : x^{2}_{1}+x^{2}_{2} <1.5   \right\},  \label{6.7}
\end{equation}
\begin{equation}
B_2:=\left\{(x_1, x_2):\begin{array}{cc}
      (x_{1}+1.5)^2+(x_{2}+1.5)^{2} < (1.0)^{2}\ or \\
      1 < x_1 < 2,\ -2 < x_2 < 2\ or \\
      -2 < x_1 < 2,\ -2.0 < x_2 < -1.0 
    \end{array}
\right\},  \label{6.8}
\end{equation}
In Figure \ref{true}, the blue closed curve is the boundary $\partial B_j$ of the support $B_j$, and the green brightness indicates the value of the true function on each cell divided into $(2M)^2=256$ in the sampling domain $[-S, S]^2=[-3, 3]^2$.
Here, we always employ the initial guess $q_0$ as 
\begin{equation}
q_0\equiv0.  \label{6.9}
\end{equation}
\par
Figures \ref{KFNn001} and \ref{KFNn01} show the reconstruction by the  Kalman filter Levenberg–Marquardt (KFL) discussed in (\ref{4.10})--(\ref{4.14}) with noisy $\sigma=0.01$ and $0.1$, respectively, while Figures \ref{EKFn001} and \ref{EKFn01} show the reconstruction by and the iterative Extended Kalman filter (EKF) discussed in (\ref{5.9})--(\ref{5.13}) with noisy $\sigma=0.01$ and $0.1$, respectively.
The first and second columns in Figures \ref{KFNn001} and \ref{KFNn01} correspond to visualization for discrepancy constant $\rho=0.4$, and $0.8$, respectively, while those in Figures \ref{EKFn001} and \ref{EKFn01} correspond to visualization for initial regularization parameter $\alpha_{0}=50$, and $5$, respectively, for different two shapes $B_1$ and $B_2$, and for different two wave numbers $k=3$ and $k=7$.
The third column corresponds to the graph of the Mean Square Error (MSE) defined by
\begin{equation}
e_{i}:=\left\|q^{true}-q_{i} \right\|^2,  \label{6.11}
\end{equation}
where $q_{i}$ is associated with the state of $i$th iteration step. 
The horizontal axis is with respect to number of iterations, and the vertical axis is the value of MSE. 
We observe that the error of Kalman filter Levenberg–Marquardt blows up in some case (wave number $k=3$ and noise $\sigma=0.1$), in which iterative Extended Kalman filter does not. 
We also obverse that the error of iterative Extended Kalman filter decreases more rapidly than that of Kalman filter Levenberg–Marquardt.
Therefore, the result of iterative Extended Kalman filter is better than that of Kalman filter Levenberg–Marquardt in our experiments.
It would be interesting to provide the ratio of convergence for two methods to justify these numerical experiments.
\section*{Acknowledgments}
This work of the first author was supported by Grant-in-Aid for JSPS Fellows (No.21J00119), Japan Society for the Promotion of Science.

\bibliographystyle{plain}
\bibliography{KF.bib}

\begin{figure}[h]
  \begin{minipage}[b]{0.5\linewidth}
  \centering
  \includegraphics[keepaspectratio, scale=0.5]
  {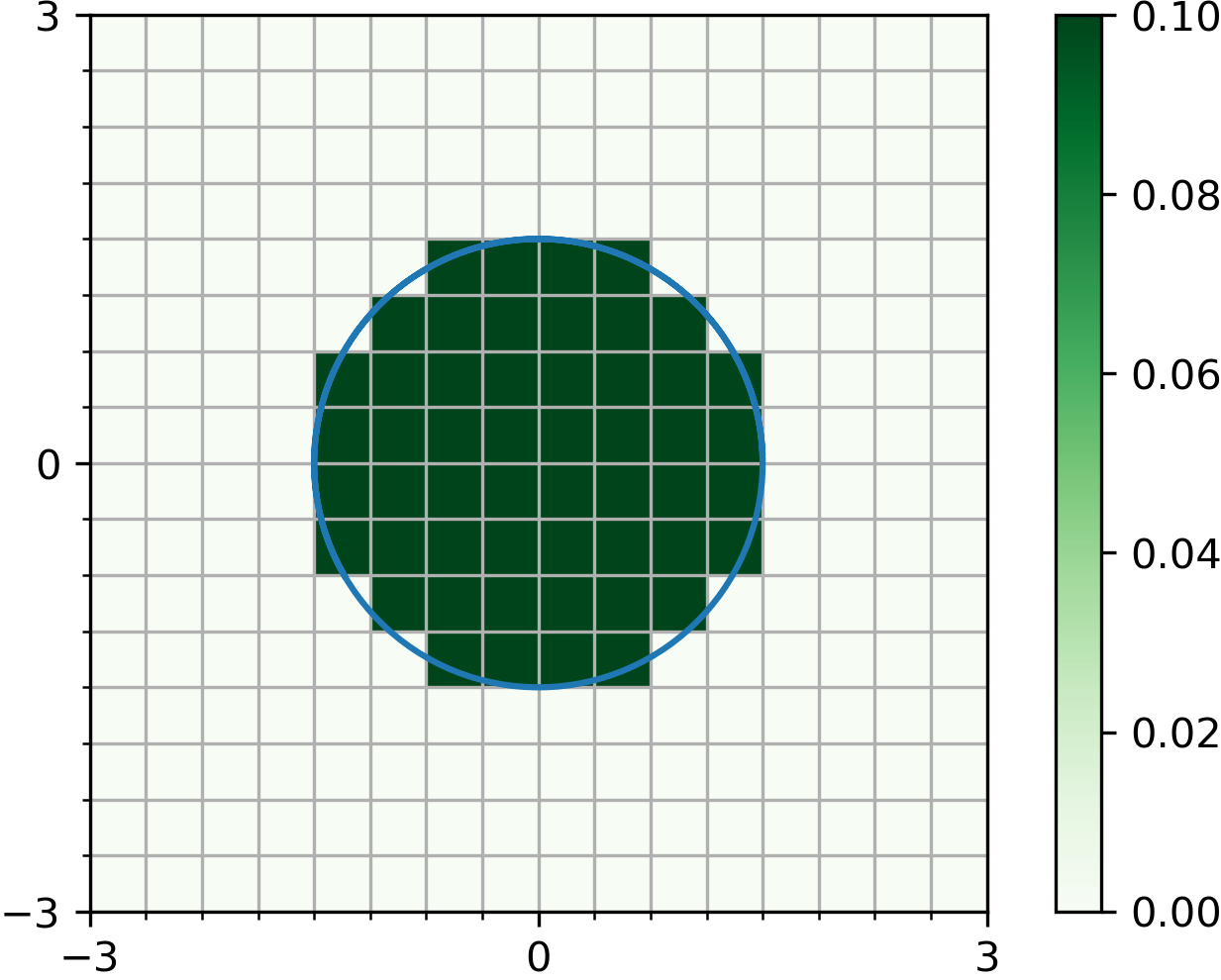}
  \subcaption{$q^{true}_{1}$}
 \end{minipage}
 \begin{minipage}[b]{0.5\linewidth}
  \centering
  \includegraphics[keepaspectratio, scale=0.5]
  {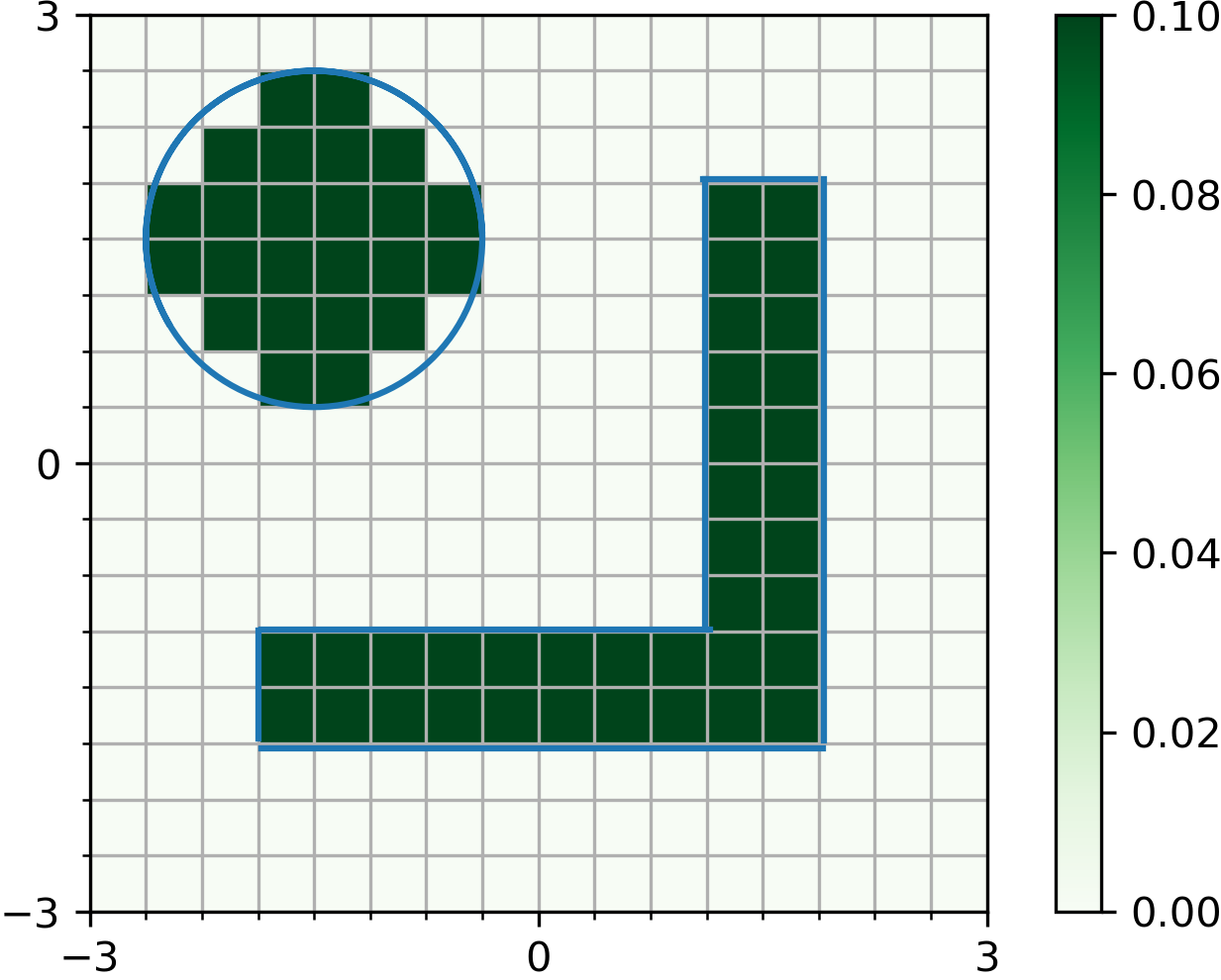}
  \subcaption{$q^{true}_{2}$}
 \end{minipage}
 \caption{true functions}
 \label{true}
\end{figure}

\begin{figure}[h]
\begin{tabular}{c}
\hspace{-2.5cm}
 \begin{minipage}[b]{0.4\linewidth}
  \centering
  \includegraphics[keepaspectratio, scale=0.45]
  {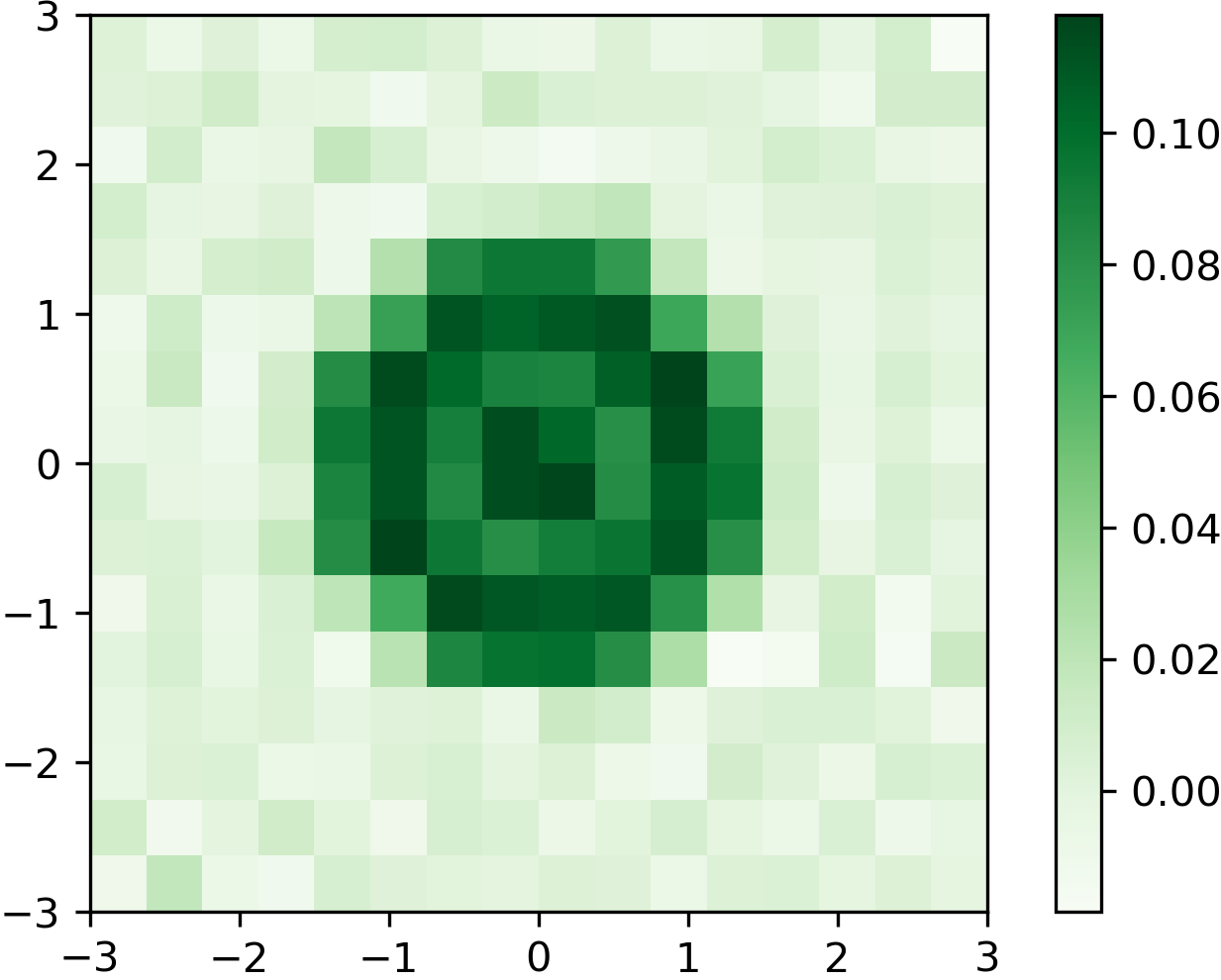}
  \subcaption{$B_1$, $k=3$, $\rho=0.4$}
 \end{minipage}
 \begin{minipage}[b]{0.4\linewidth}
  \centering
  \includegraphics[keepaspectratio, scale=0.45]
  {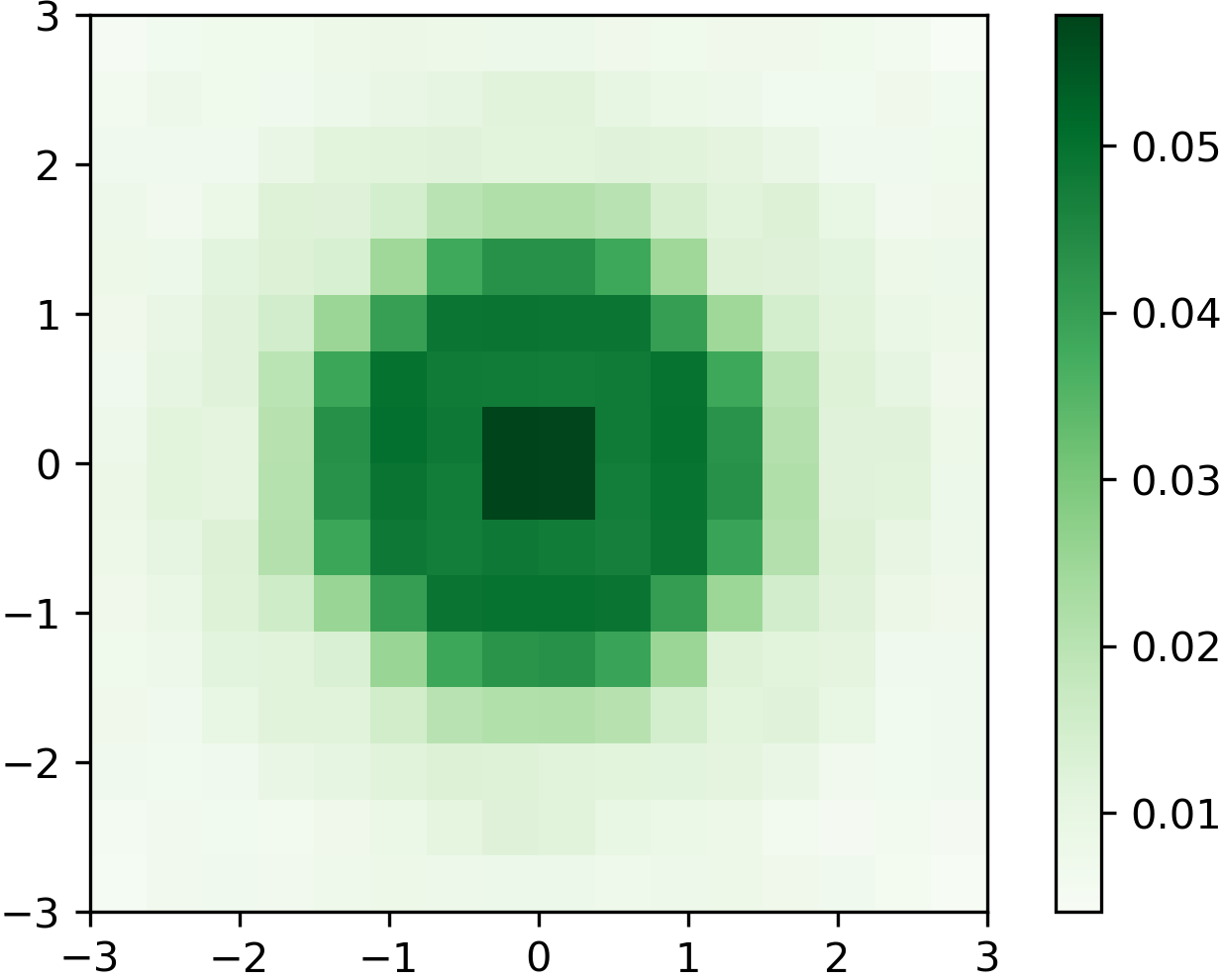}
  \subcaption{$B_1$, $k=3$, $\rho=0.8$}
 \end{minipage}
 \begin{minipage}[b]{0.4\linewidth}
  \centering
  \includegraphics[keepaspectratio, scale=0.45]
  {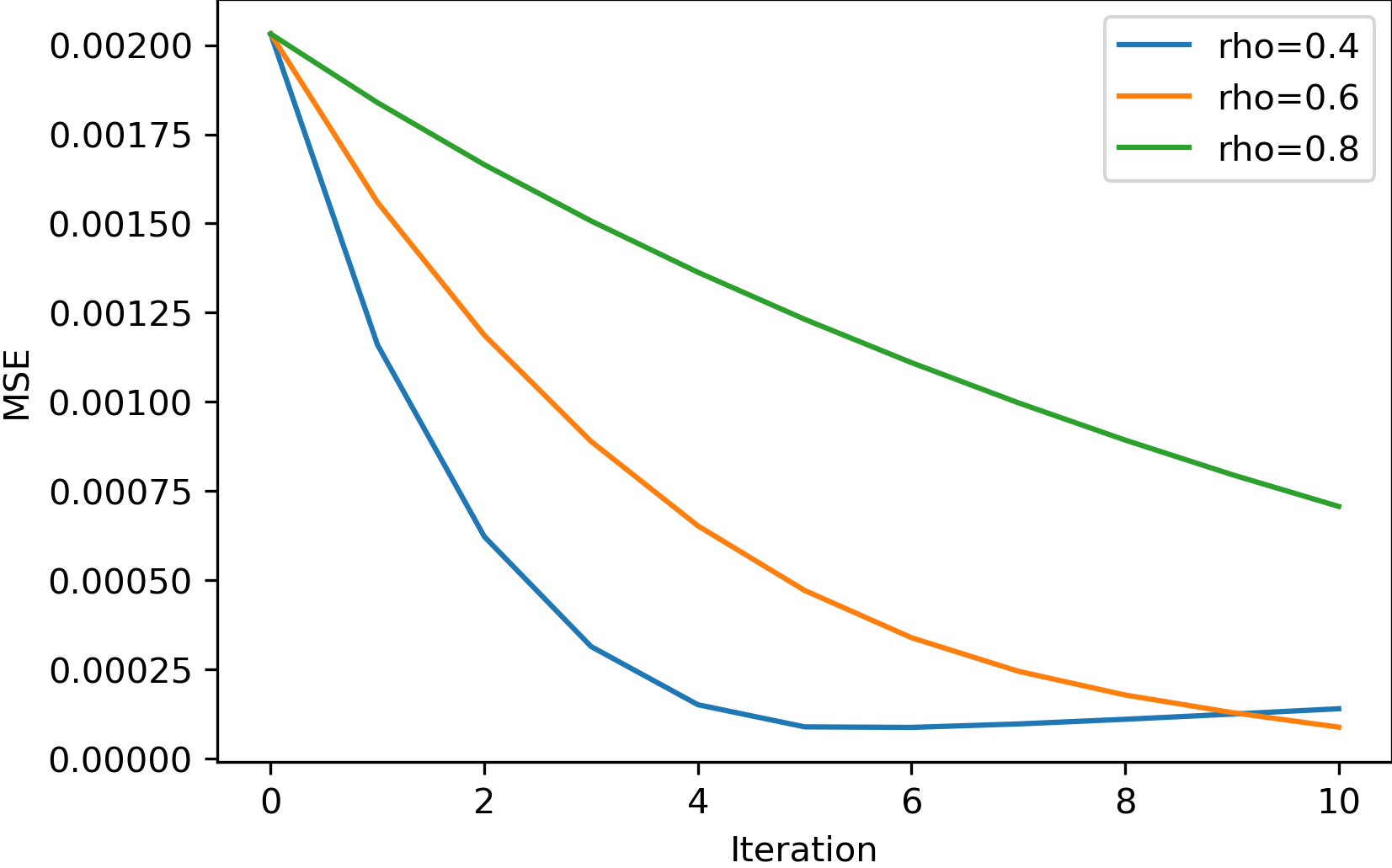}
  \subcaption{$B_1$, $k=3$, error graph}
 \end{minipage}
\end{tabular}

\begin{tabular}{c}
\hspace{-2.5cm}
 \begin{minipage}[b]{0.4\linewidth}
  \centering
  \includegraphics[keepaspectratio, scale=0.45]
  {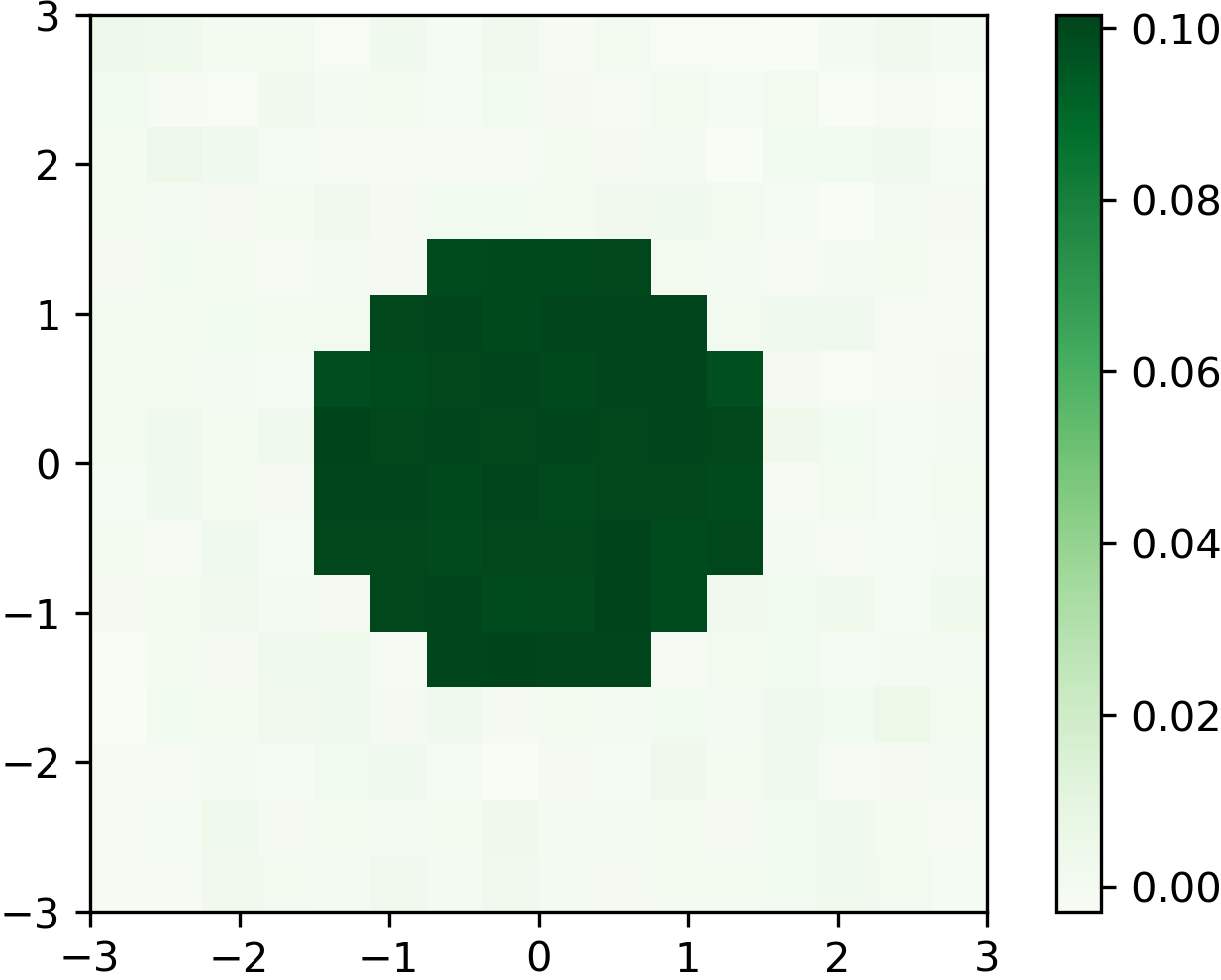}
  \subcaption{$B_1$, $k=7$, $\rho=0.4$}
 \end{minipage}
 \begin{minipage}[b]{0.4\linewidth}
  \centering
  \includegraphics[keepaspectratio, scale=0.45]
  {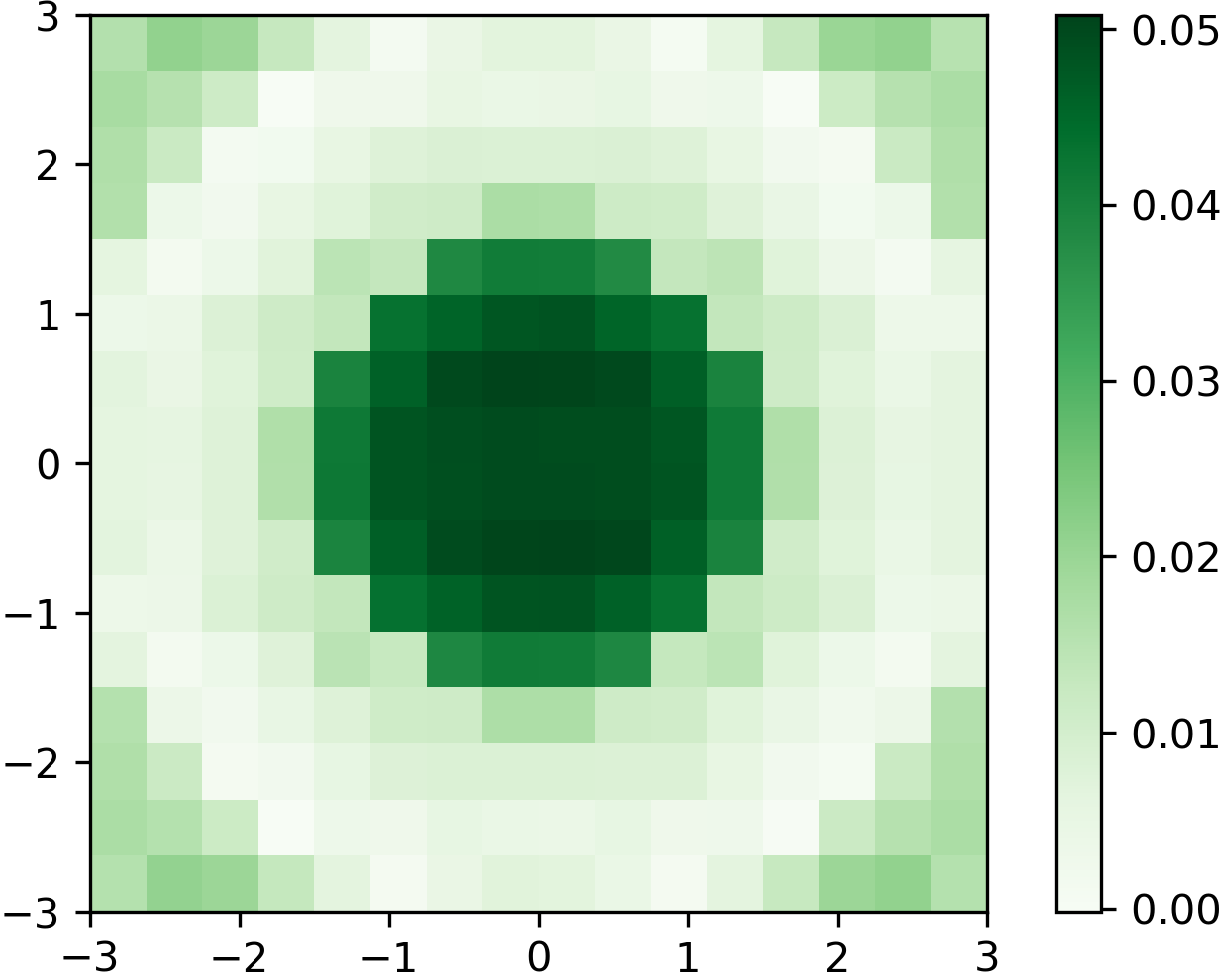}
  \subcaption{$B_1$, $k=7$, $\rho=0.8$}
 \end{minipage}
 \begin{minipage}[b]{0.4\linewidth}
  \centering
  \includegraphics[keepaspectratio, scale=0.45]
  {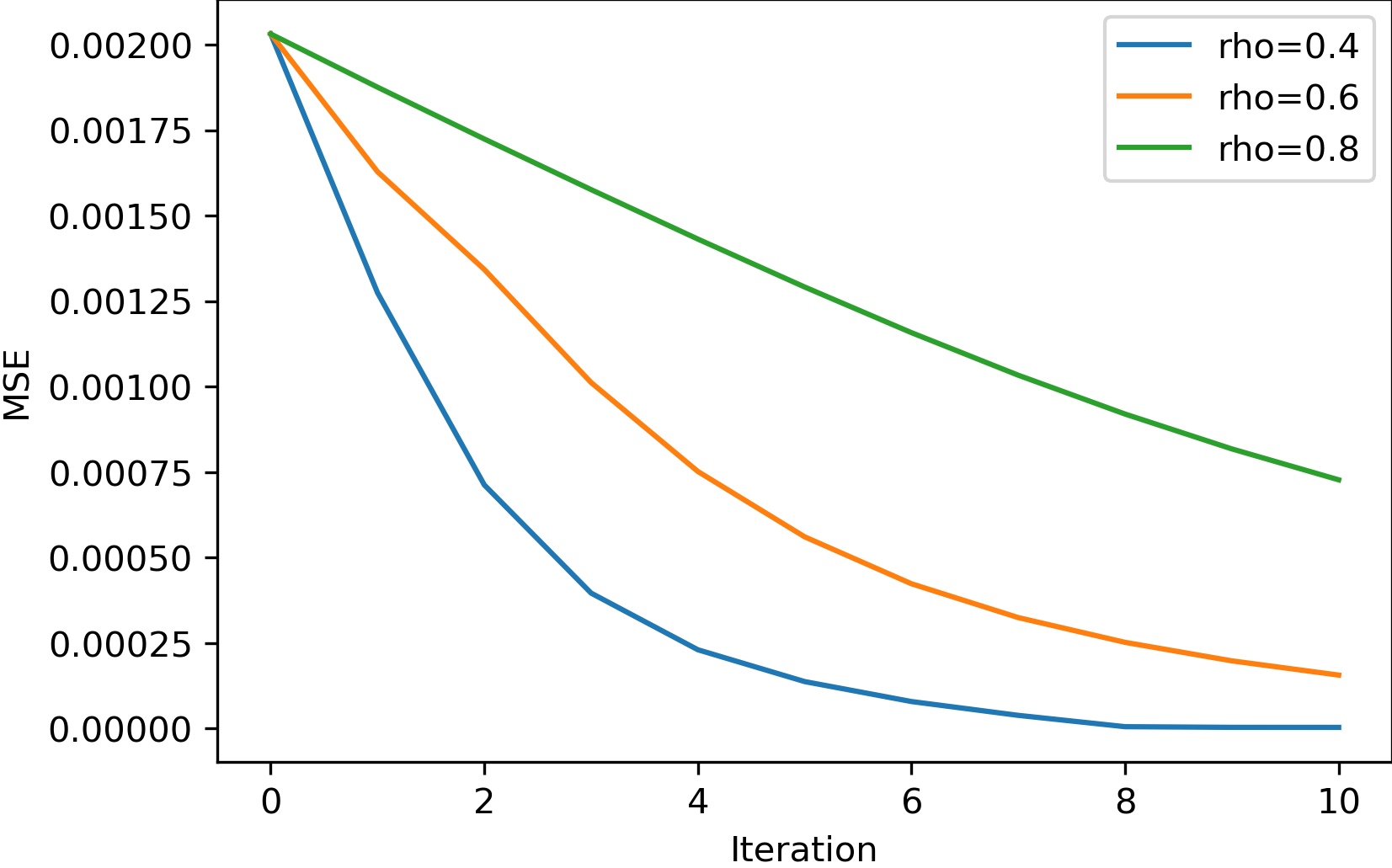}
  \subcaption{$B_1$, $k=7$, error graph}
 \end{minipage}
\end{tabular}

\begin{tabular}{c}
\hspace{-2.5cm}
 \begin{minipage}[b]{0.4\linewidth}
  \centering
  \includegraphics[keepaspectratio, scale=0.45]
  {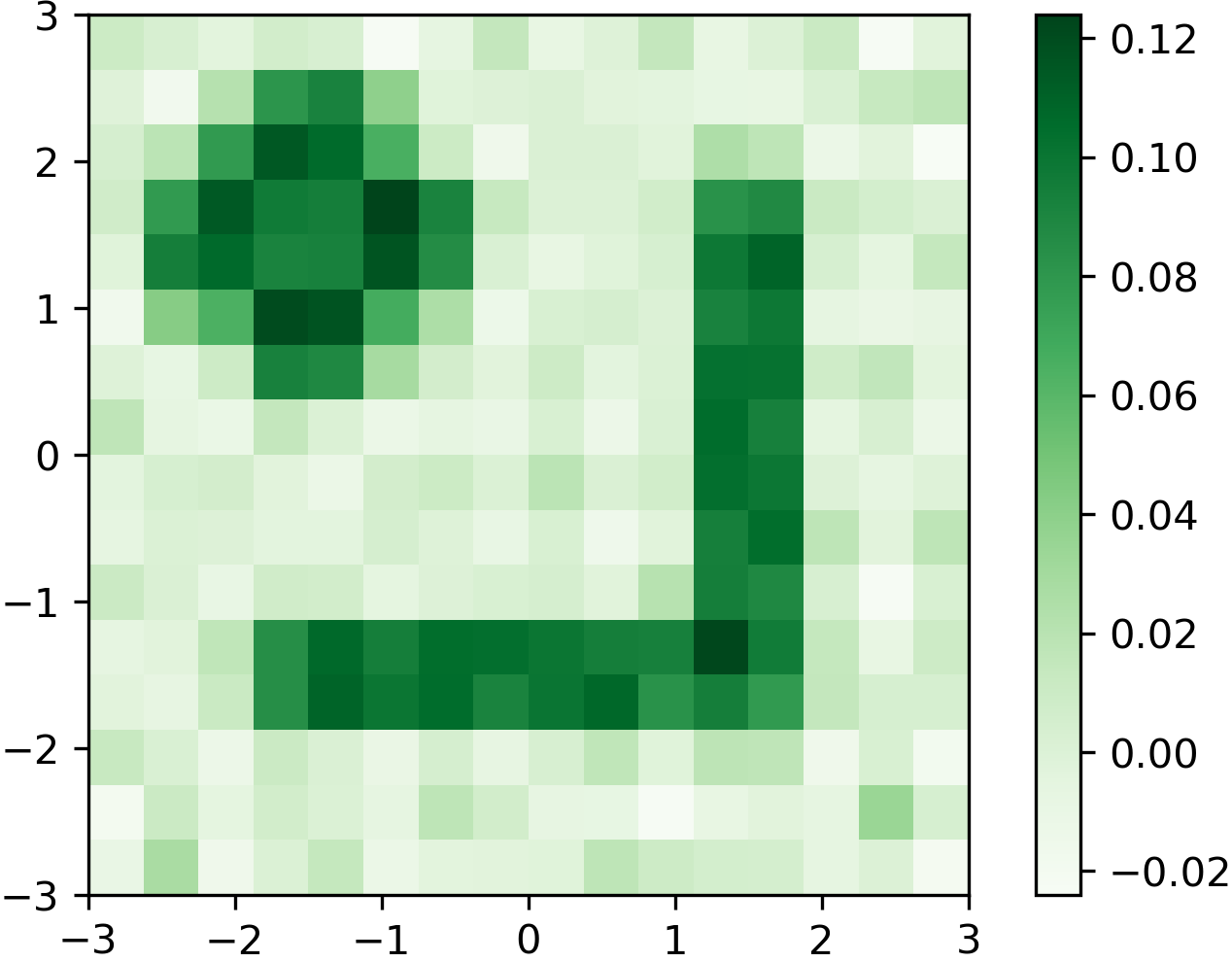}
  \subcaption{$B_2$, $k=3$, $\rho=0.4$}
 \end{minipage}
 \begin{minipage}[b]{0.4\linewidth}
  \centering
  \includegraphics[keepaspectratio, scale=0.45]
  {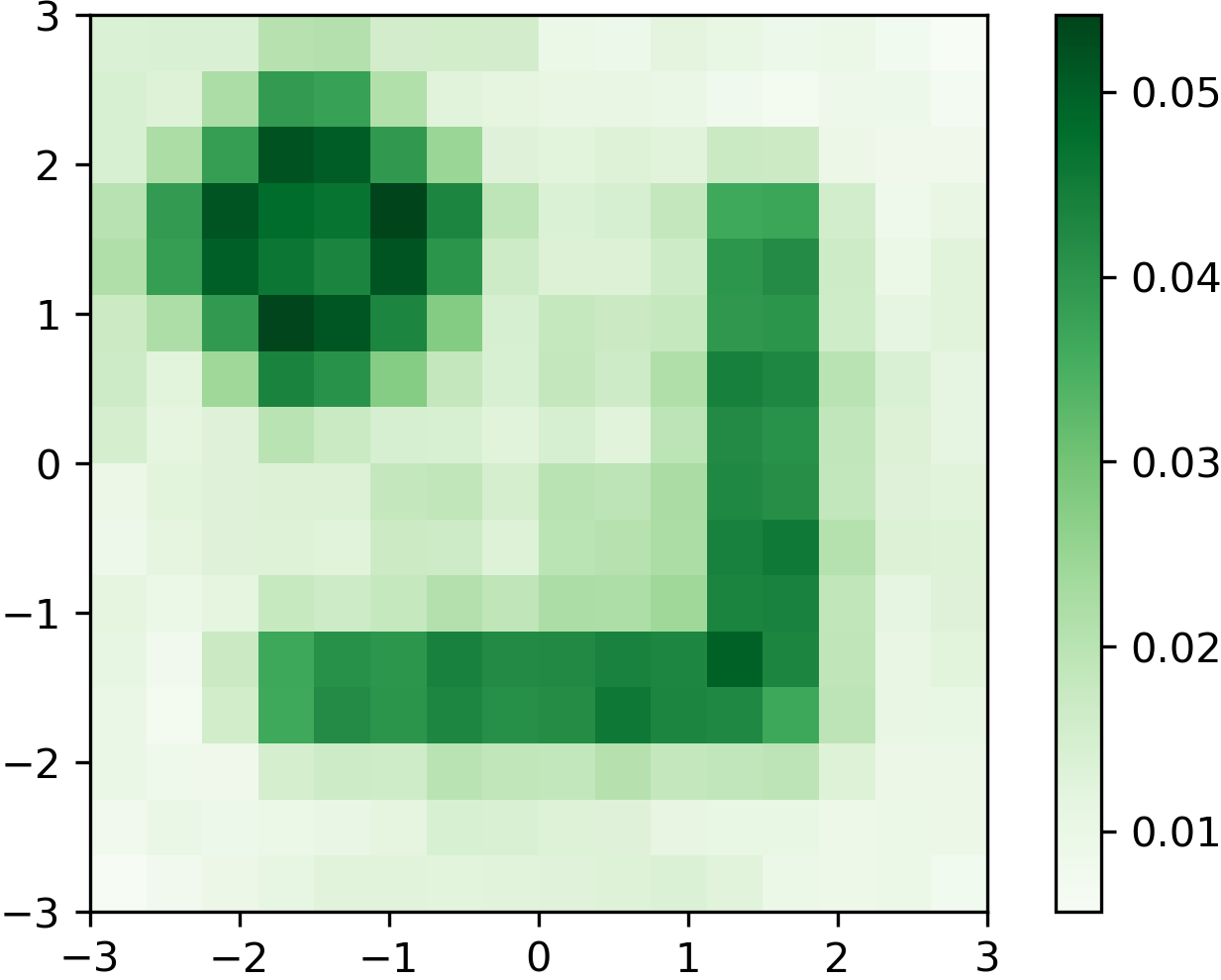}
  \subcaption{$B_2$, $k=3$, $\rho=0.8$}
 \end{minipage}
 \begin{minipage}[b]{0.4\linewidth}
  \centering
  \includegraphics[keepaspectratio, scale=0.45]
  {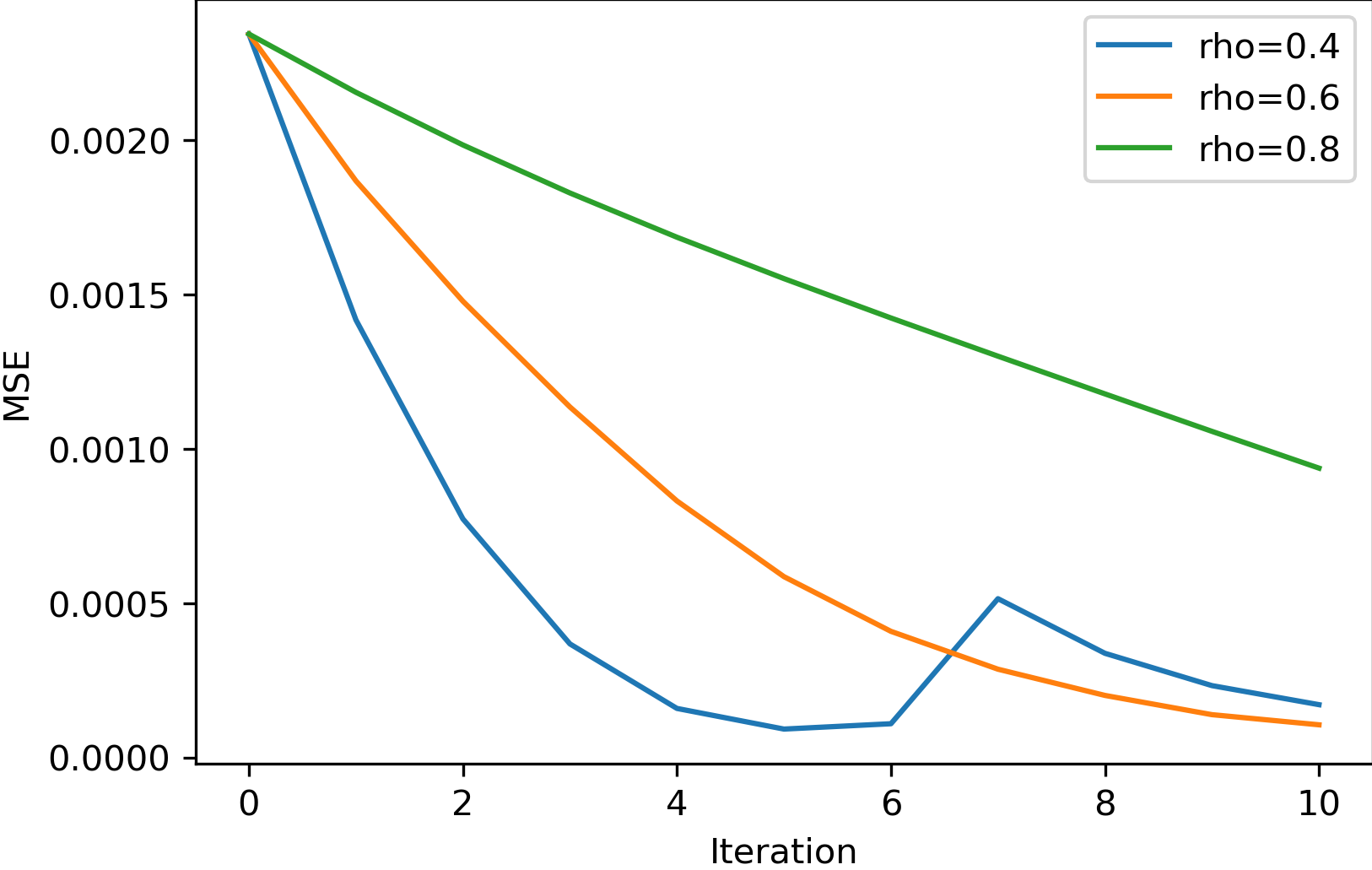}
  \subcaption{$B_2$, $k=3$, error graph}
 \end{minipage}
\end{tabular}

\begin{tabular}{c}
\hspace{-2.5cm}
 \begin{minipage}[b]{0.4\linewidth}
  \centering
  \includegraphics[keepaspectratio, scale=0.45]
  {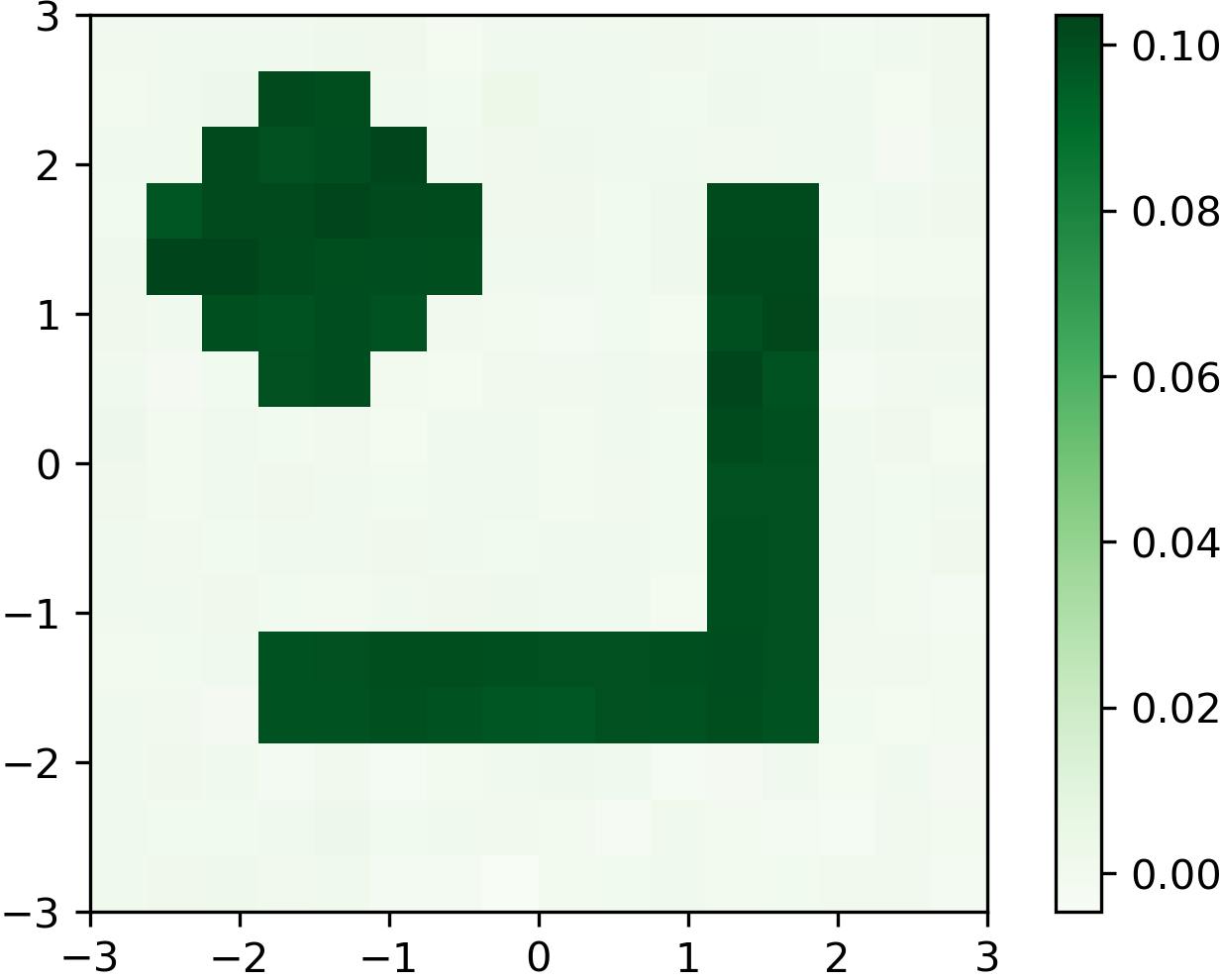}
  \subcaption{$B_2$, $k=7$, $\rho=0.4$}
 \end{minipage}
 \begin{minipage}[b]{0.4\linewidth}
  \centering
  \includegraphics[keepaspectratio, scale=0.45]
  {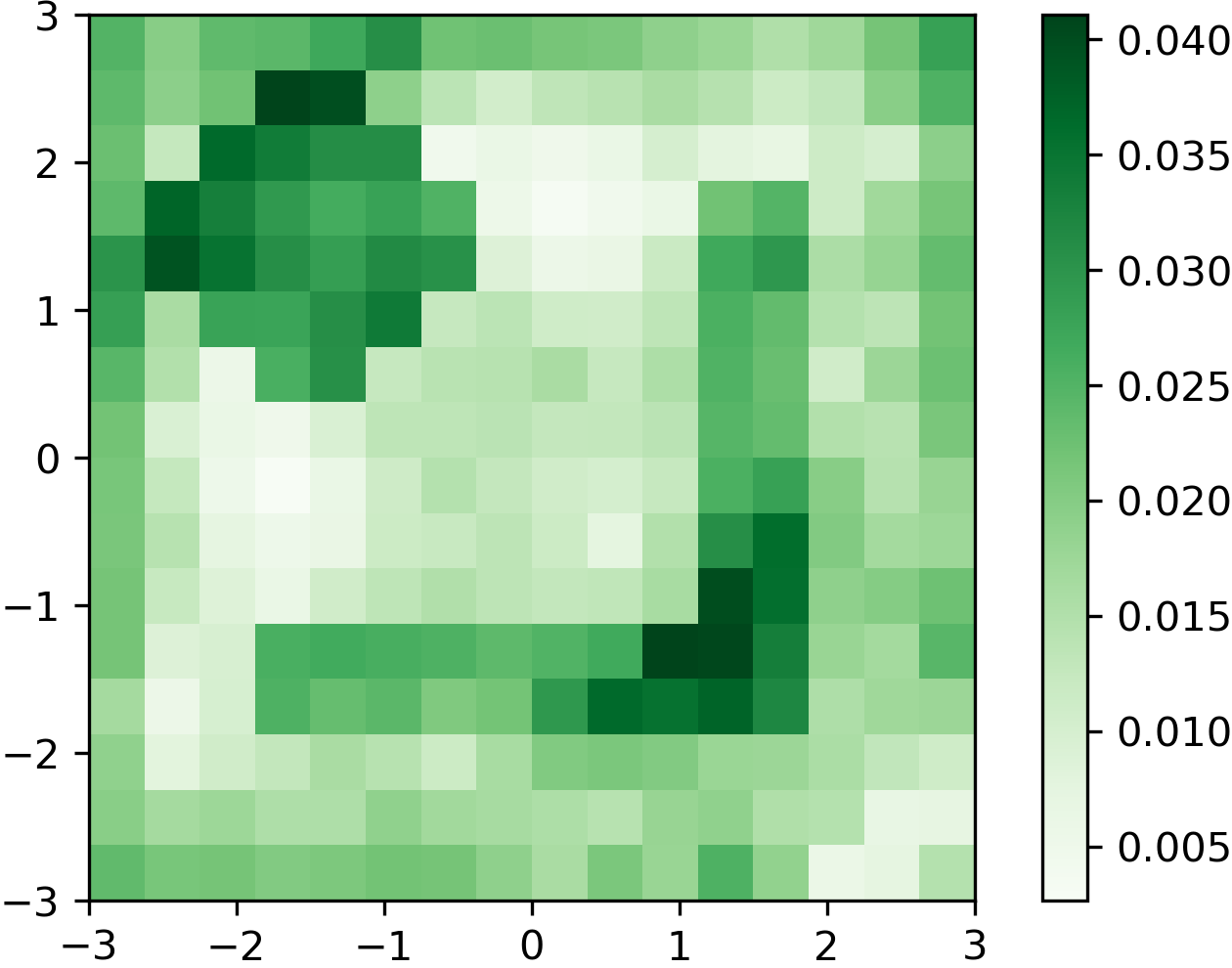}
  \subcaption{$B_2$, $k=7$, $\rho=0.8$}
 \end{minipage}
 \begin{minipage}[b]{0.4\linewidth}
  \centering
  \includegraphics[keepaspectratio, scale=0.45]
  {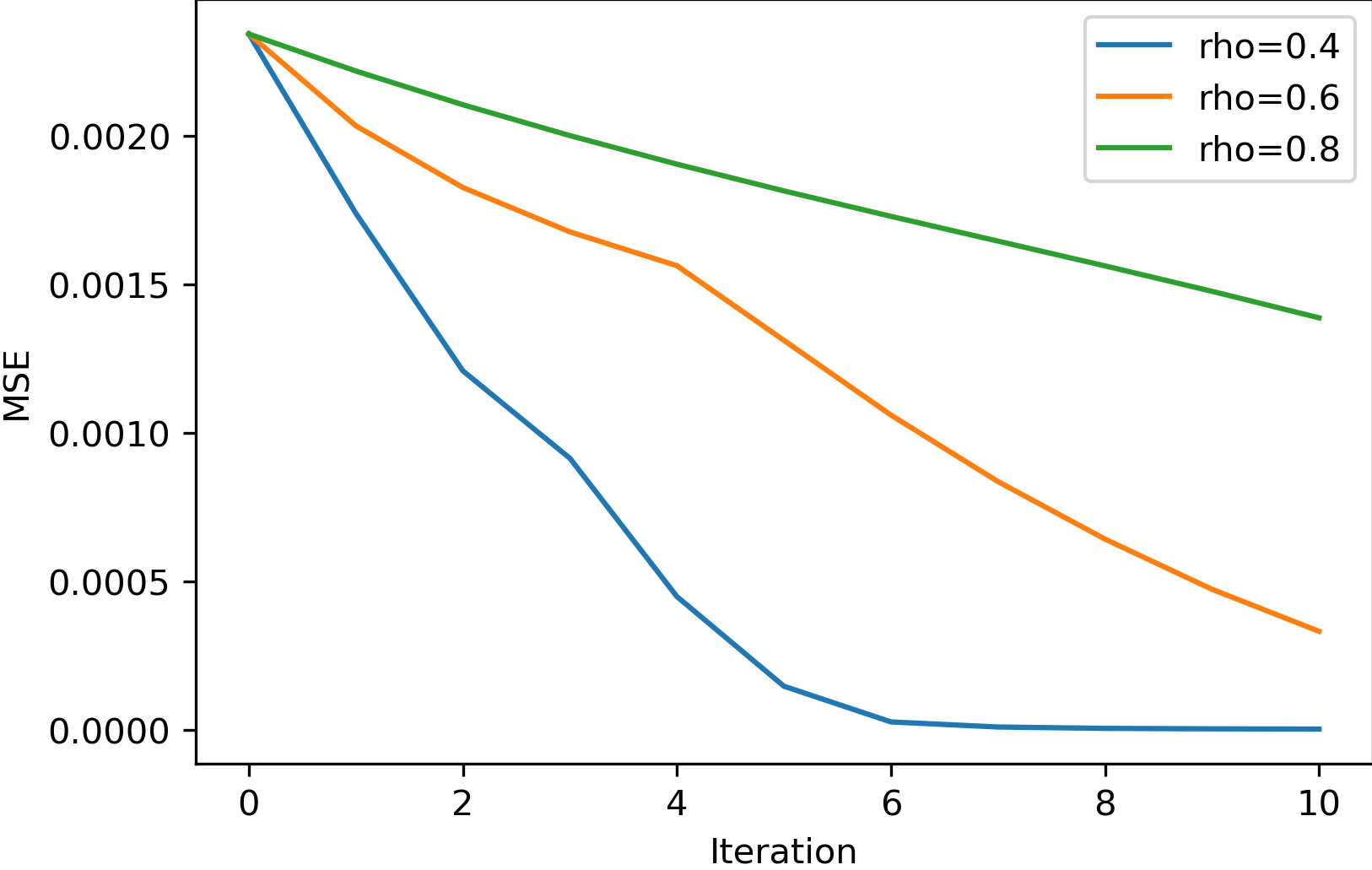}
  \subcaption{$B_2$, $k=7$, error graph}
 \end{minipage}
\end{tabular}
\caption{KFL (nosiy $\sigma=0.01$)}
\label{KFNn001}
\end{figure}

\begin{figure}[h]
\begin{tabular}{c}
\hspace{-2.5cm}
 \begin{minipage}[b]{0.4\linewidth}
  \centering
  \includegraphics[keepaspectratio, scale=0.45]
  {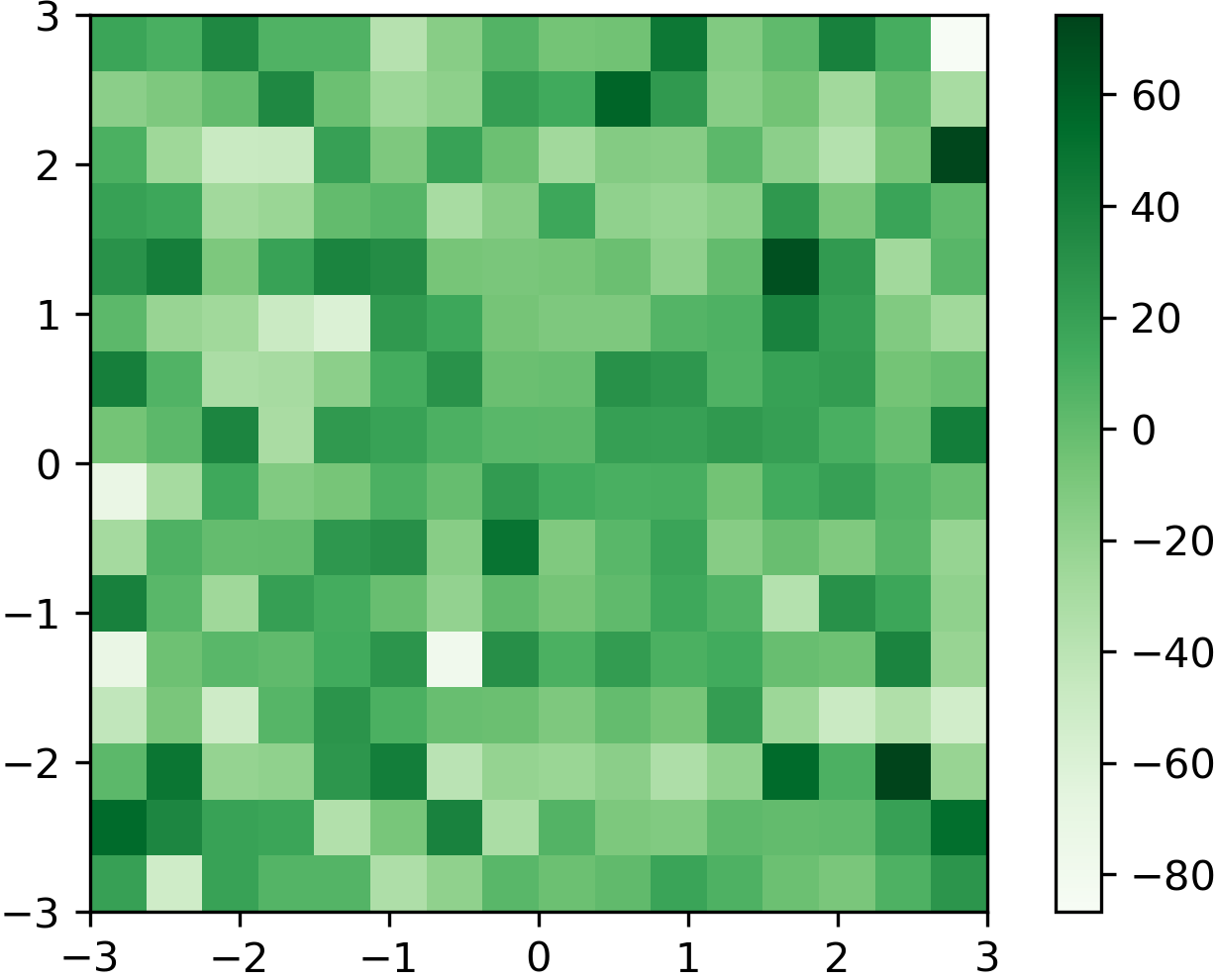}
  \subcaption{$B_1$, $k=3$, $\rho=0.4$}
 \end{minipage}
 \begin{minipage}[b]{0.4\linewidth}
  \centering
  \includegraphics[keepaspectratio, scale=0.45]
  {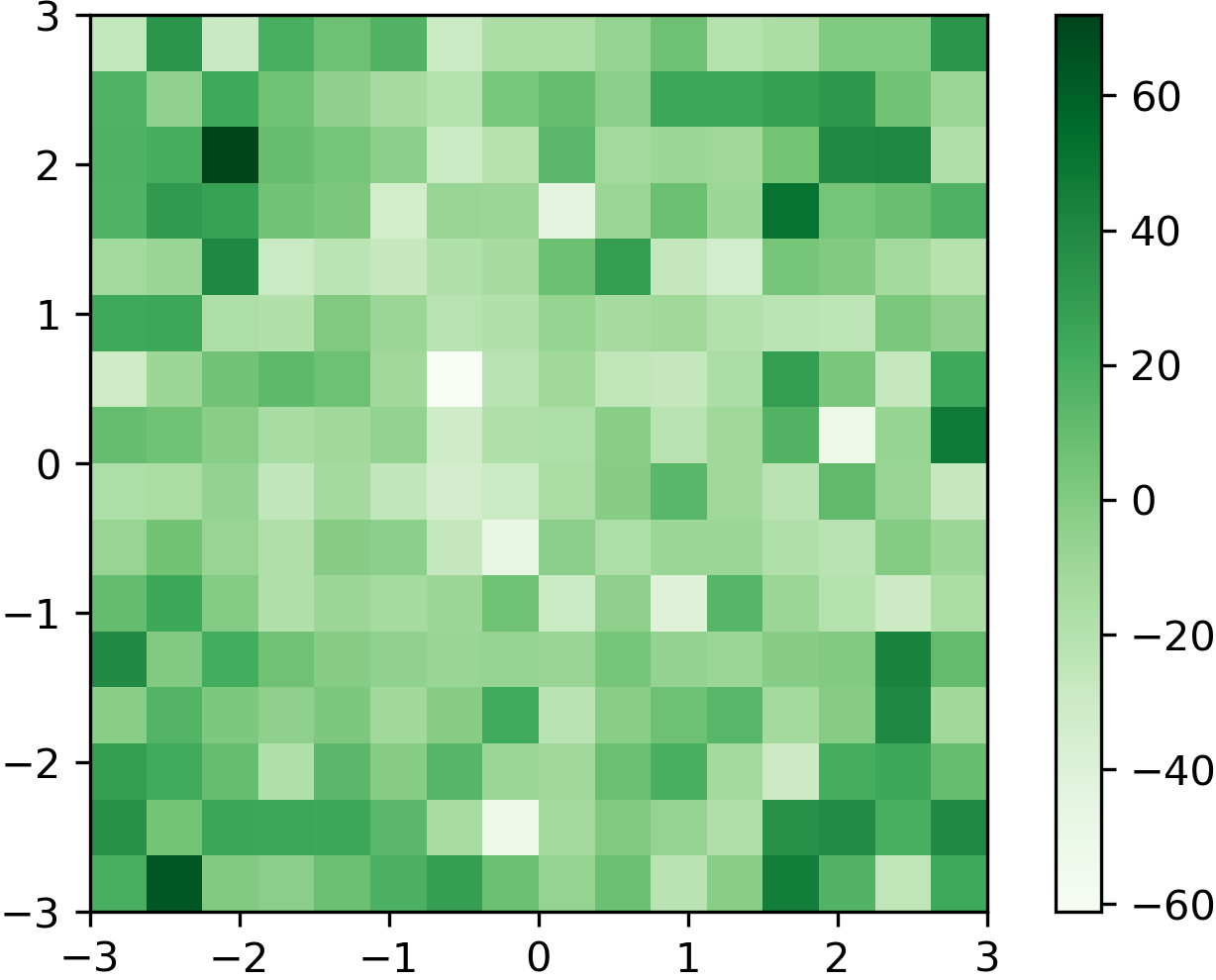}
  \subcaption{$B_1$, $k=3$, $\rho=0.8$}
 \end{minipage}
 \begin{minipage}[b]{0.4\linewidth}
  \centering
  \includegraphics[keepaspectratio, scale=0.45]
  {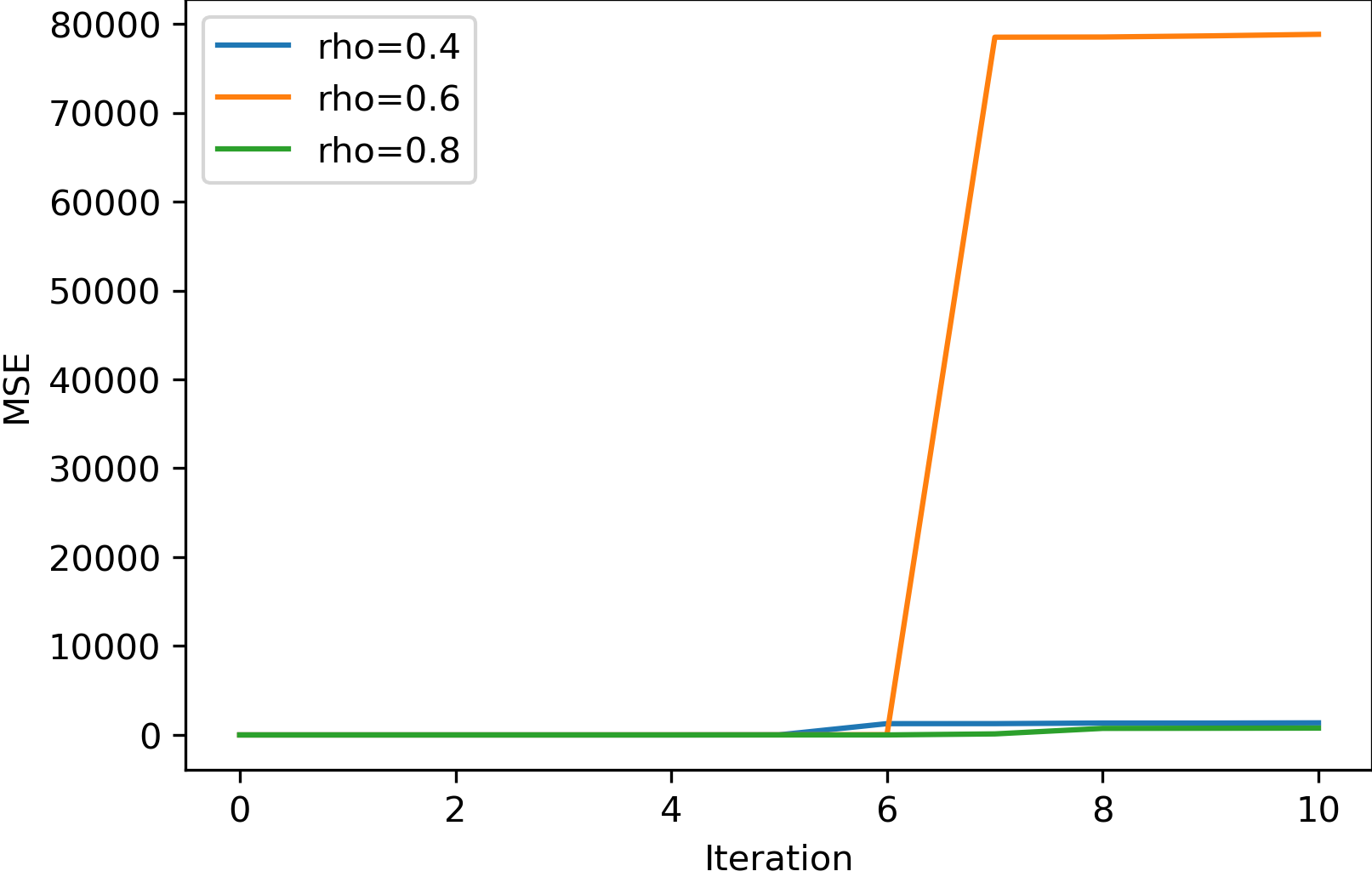}
  \subcaption{$B_1$, $k=3$, error graph}
 \end{minipage}
\end{tabular}

\begin{tabular}{c}
\hspace{-2.5cm}
 \begin{minipage}[b]{0.4\linewidth}
  \centering
  \includegraphics[keepaspectratio, scale=0.45]
  {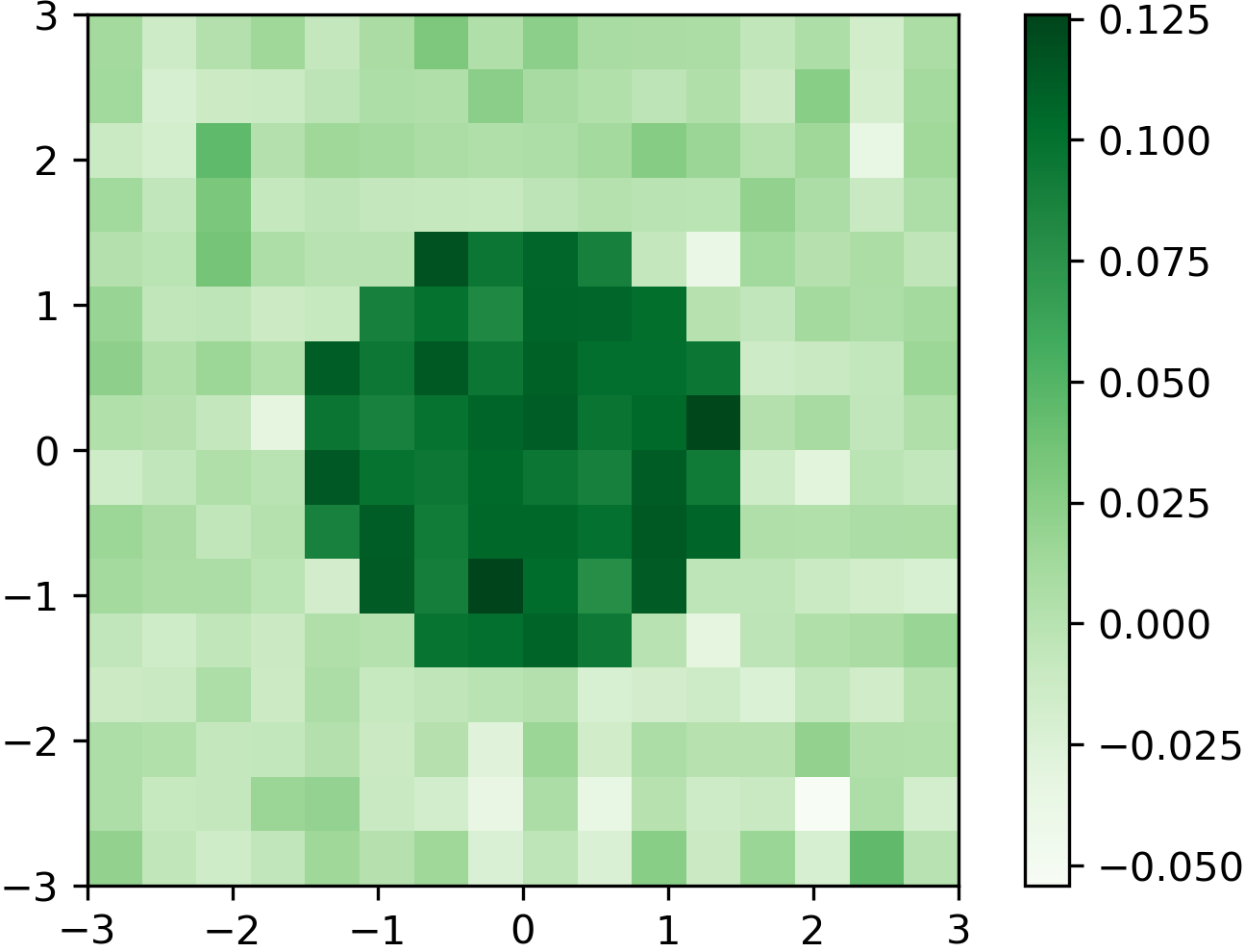}
  \subcaption{$B_1$, $k=7$, $\rho=0.4$}
 \end{minipage}
 \begin{minipage}[b]{0.4\linewidth}
  \centering
  \includegraphics[keepaspectratio, scale=0.45]
  {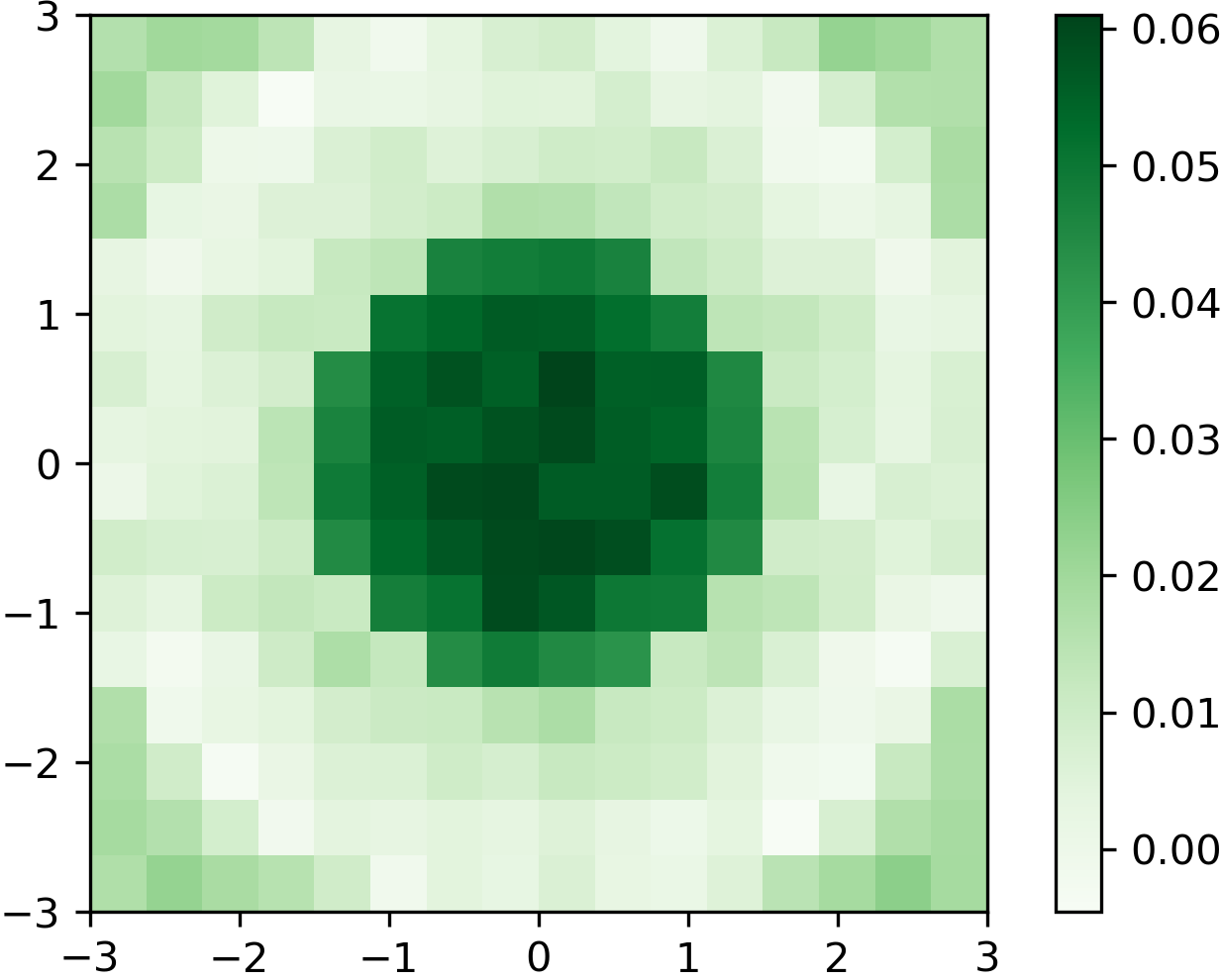} 
  \subcaption{$B_1$, $k=7$, $\rho=0.8$}
 \end{minipage}
 \begin{minipage}[b]{0.4\linewidth}
  \centering
  \includegraphics[keepaspectratio, scale=0.45]
  {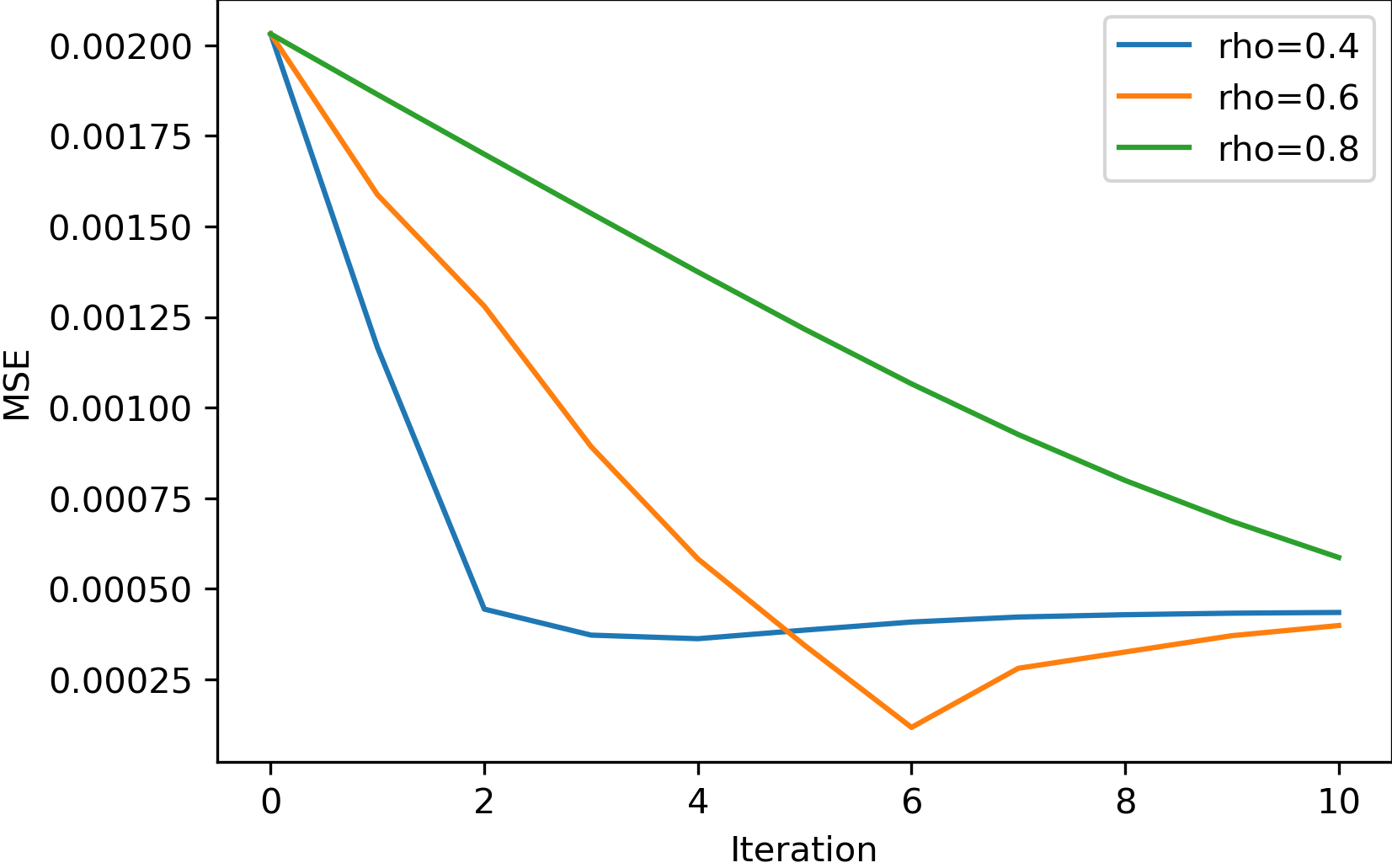}
  \subcaption{$B_1$, $k=7$, error graph}
 \end{minipage}
\end{tabular}

\begin{tabular}{c}
\hspace{-2.5cm}
 \begin{minipage}[b]{0.4\linewidth}
  \centering
  \includegraphics[keepaspectratio, scale=0.45]
  {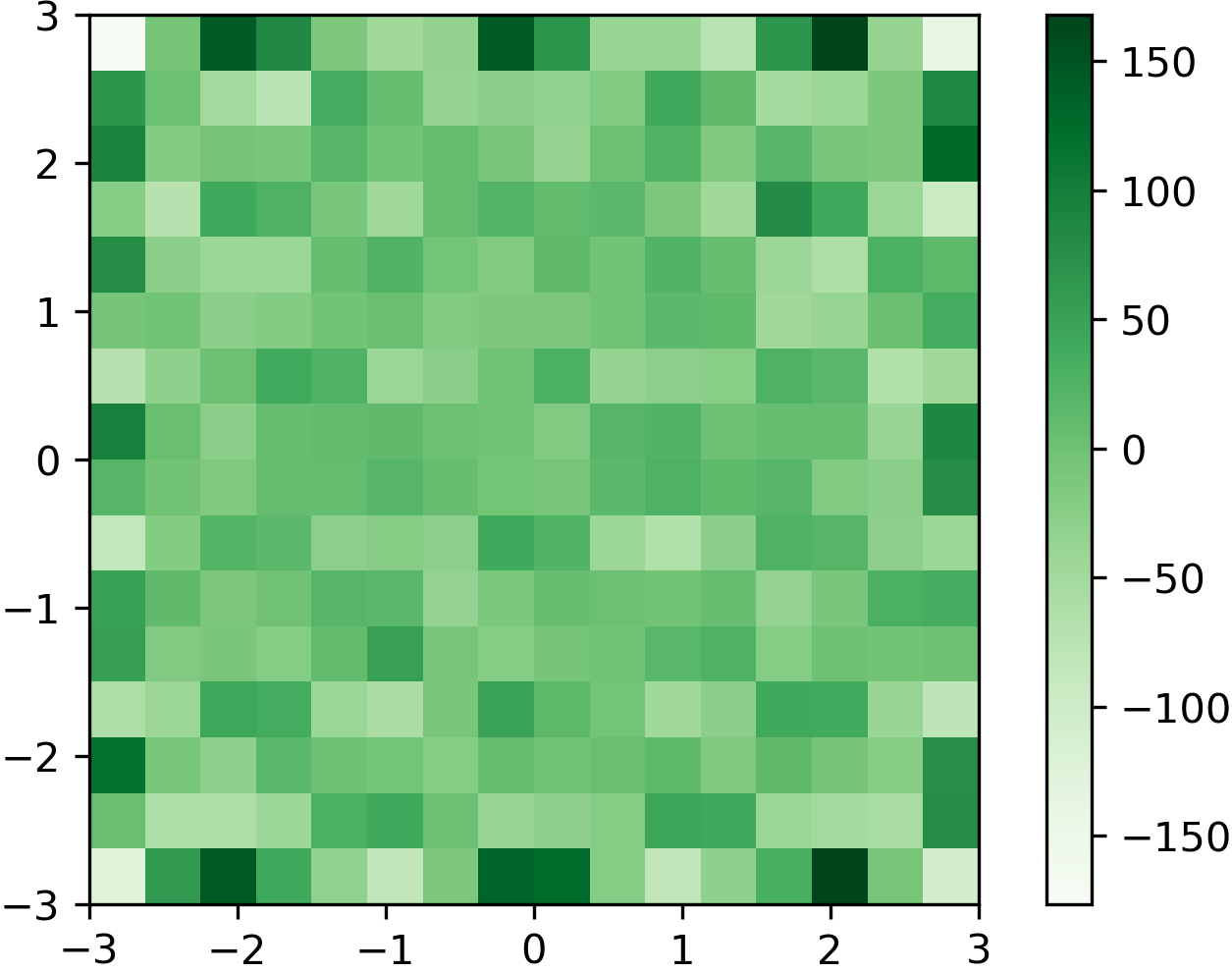}
  \subcaption{$B_2$, $k=3$, $\rho=0.4$}
 \end{minipage}
 \begin{minipage}[b]{0.4\linewidth}
  \centering
  \includegraphics[keepaspectratio, scale=0.45]
  {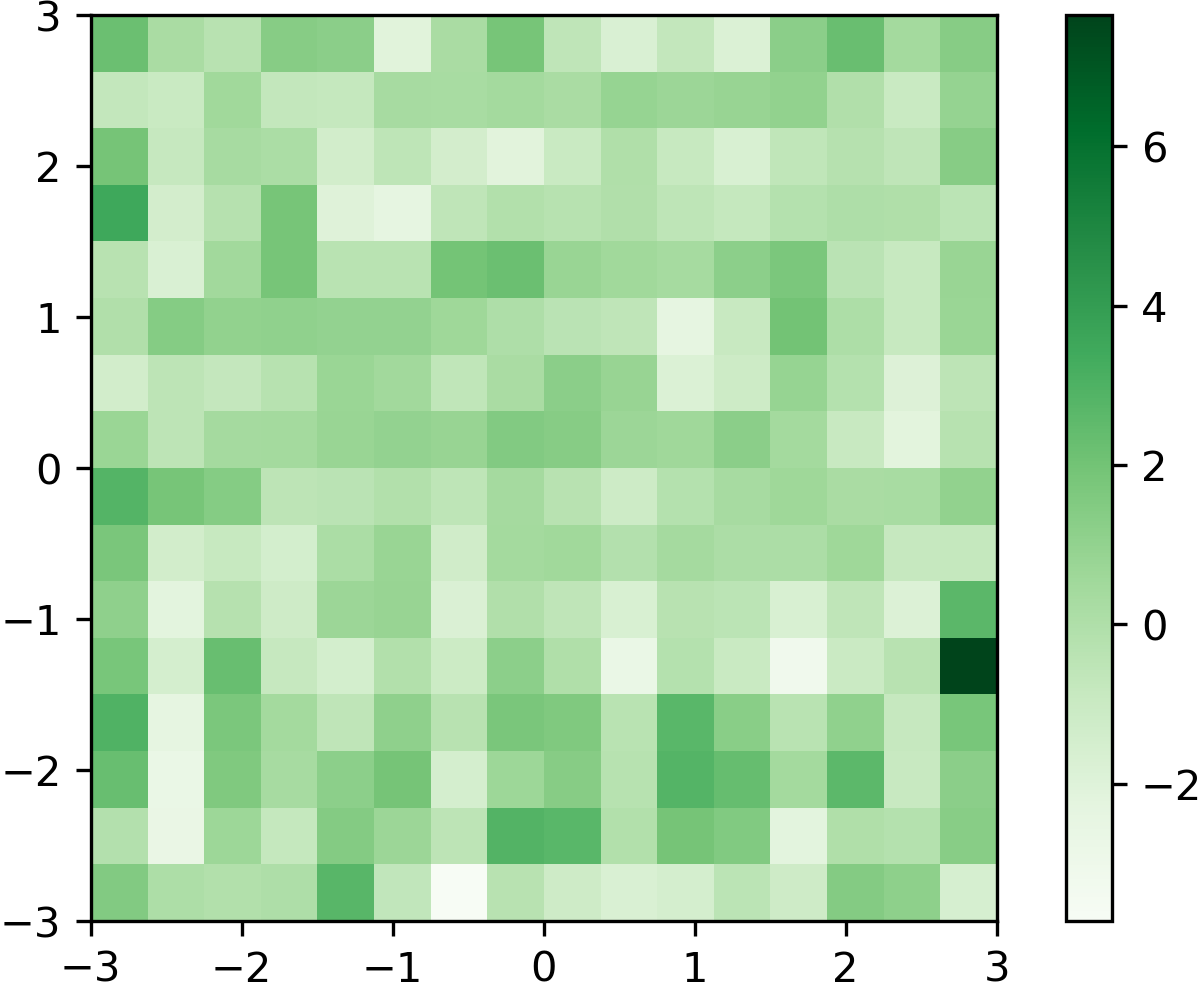}
  \subcaption{$B_2$, $k=3$, $\rho=0.8$}
 \end{minipage}
 \begin{minipage}[b]{0.4\linewidth}
  \centering
  \includegraphics[keepaspectratio, scale=0.45]
  {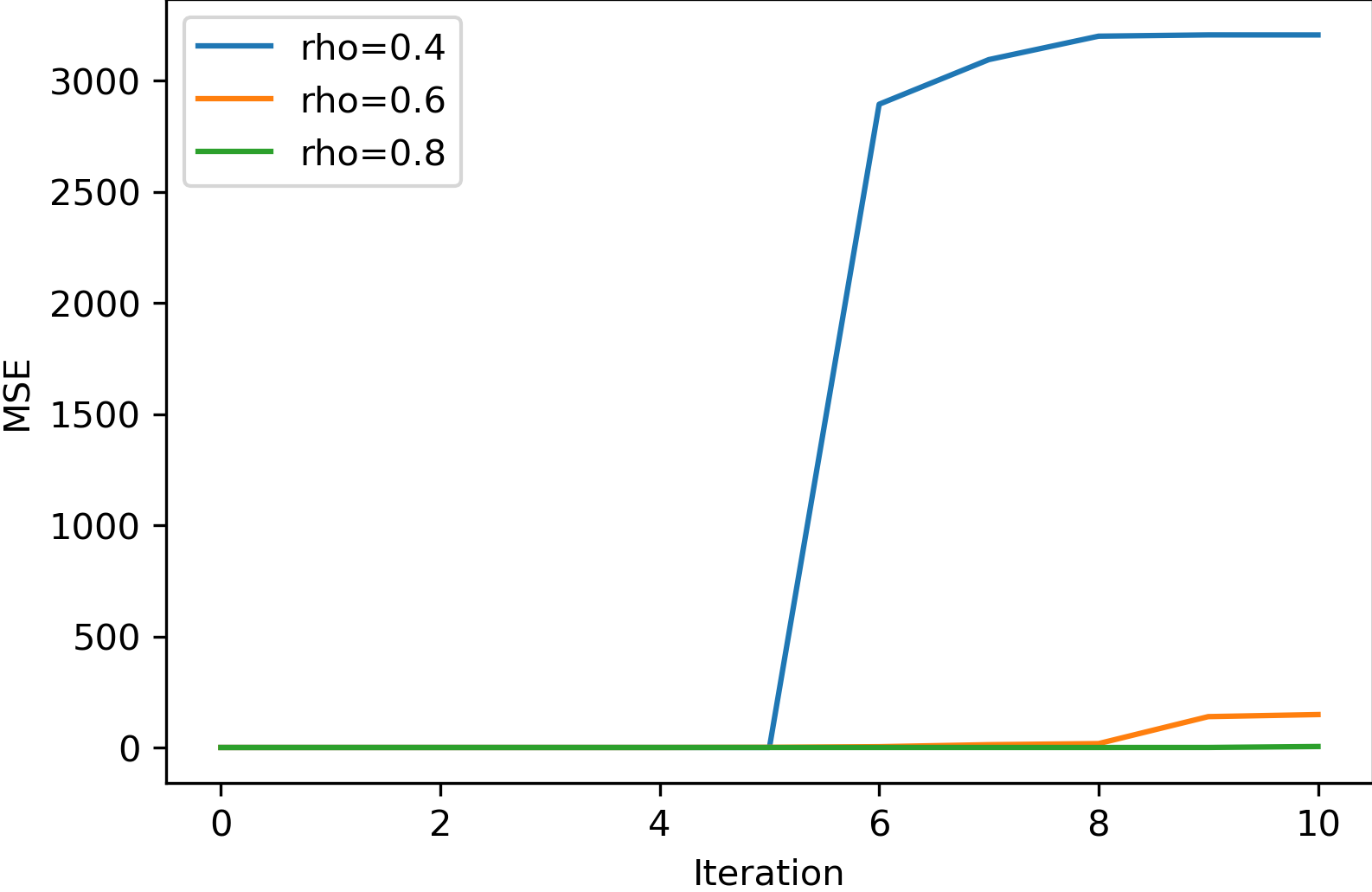}
  \subcaption{$B_2$, $k=3$, error graph}
 \end{minipage}
\end{tabular}

\begin{tabular}{c}
\hspace{-2.5cm}
 \begin{minipage}[b]{0.4\linewidth}
  \centering
  \includegraphics[keepaspectratio, scale=0.45]
  {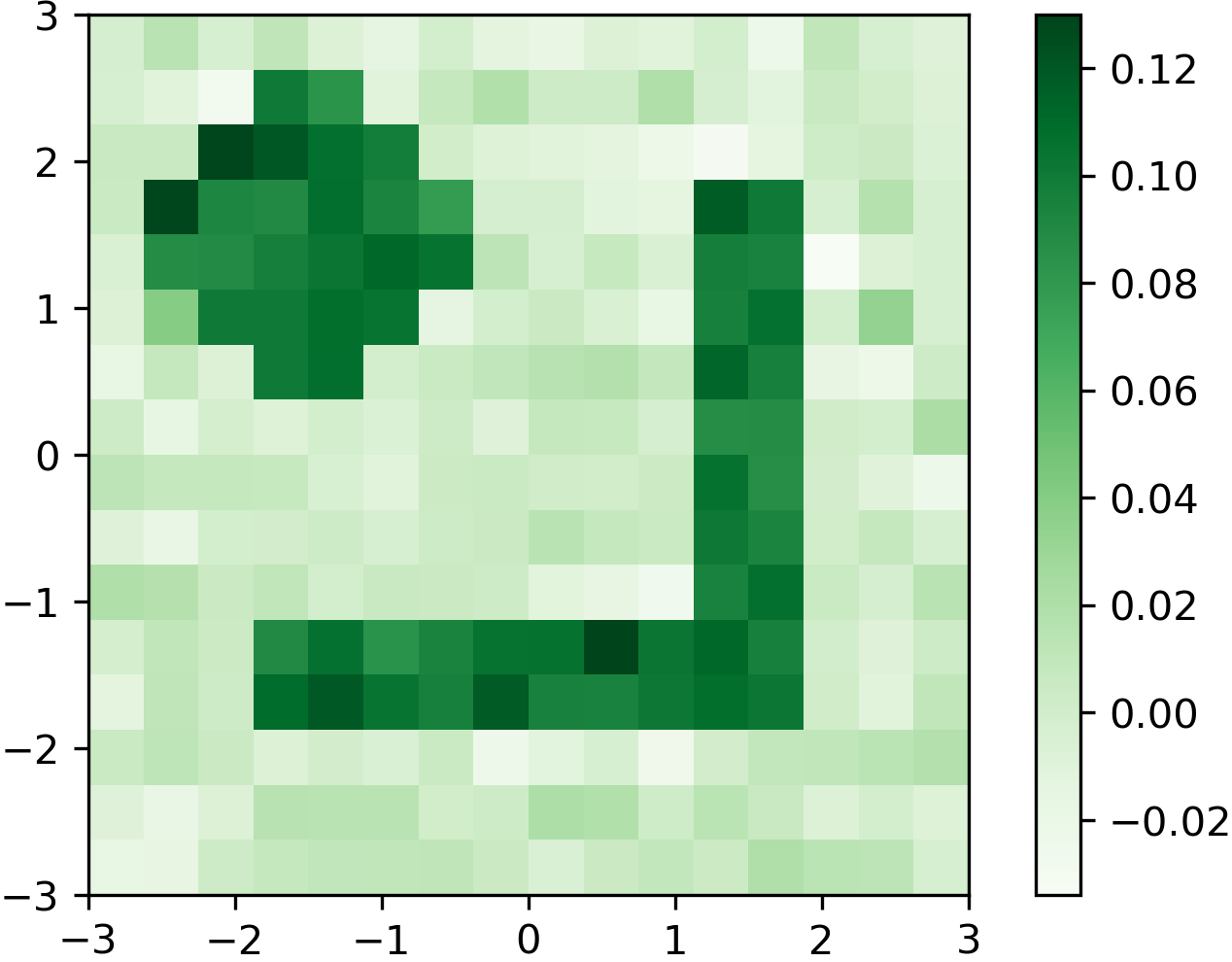}
  \subcaption{$B_2$, $k=7$, $\rho=0.4$}
 \end{minipage}
 \begin{minipage}[b]{0.4\linewidth}
  \centering
  \includegraphics[keepaspectratio, scale=0.45]
  {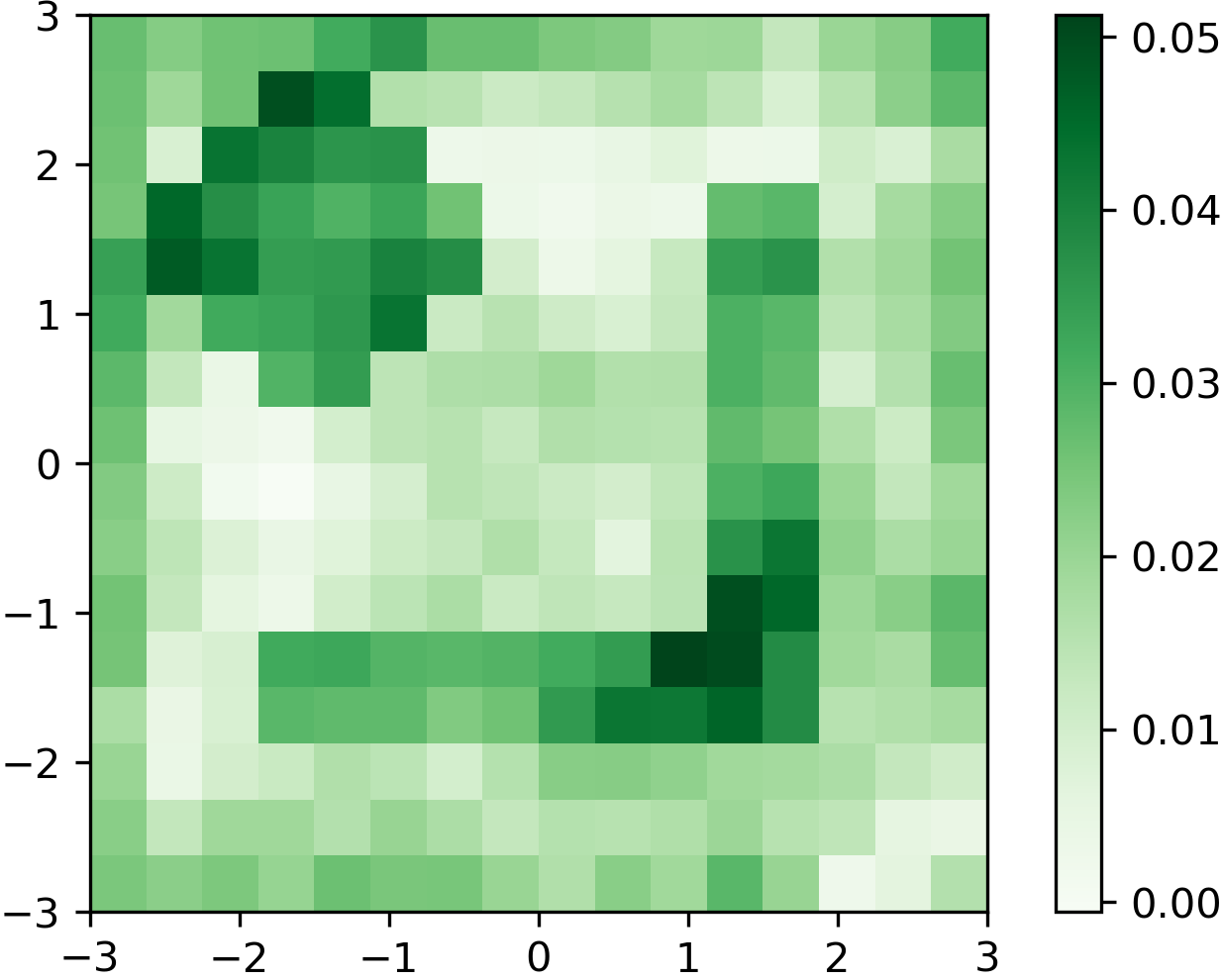}
  \subcaption{$B_2$, $k=7$, $\rho=0.8$}
 \end{minipage}
 \begin{minipage}[b]{0.4\linewidth}
  \centering
  \includegraphics[keepaspectratio, scale=0.45]
  {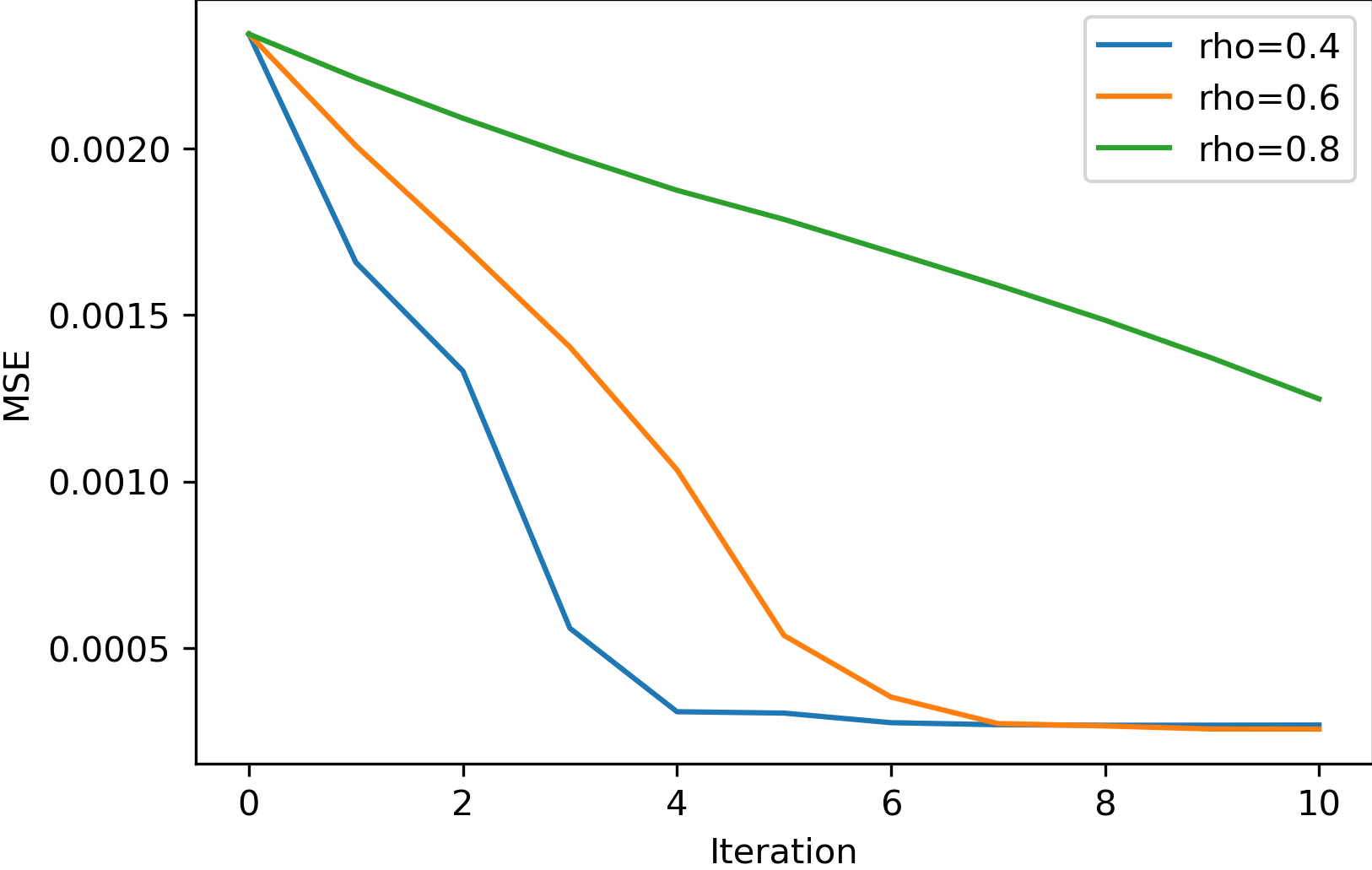}
  \subcaption{$B_2$, $k=7$, error graph}
 \end{minipage}
\end{tabular}
\caption{KFL (nosiy $\sigma=0.1$)}
\label{KFNn01}
\end{figure}

\begin{figure}[h]
\begin{tabular}{c}
\hspace{-2.5cm}
 \begin{minipage}[b]{0.4\linewidth}
  \centering
  \includegraphics[keepaspectratio, scale=0.45]
  {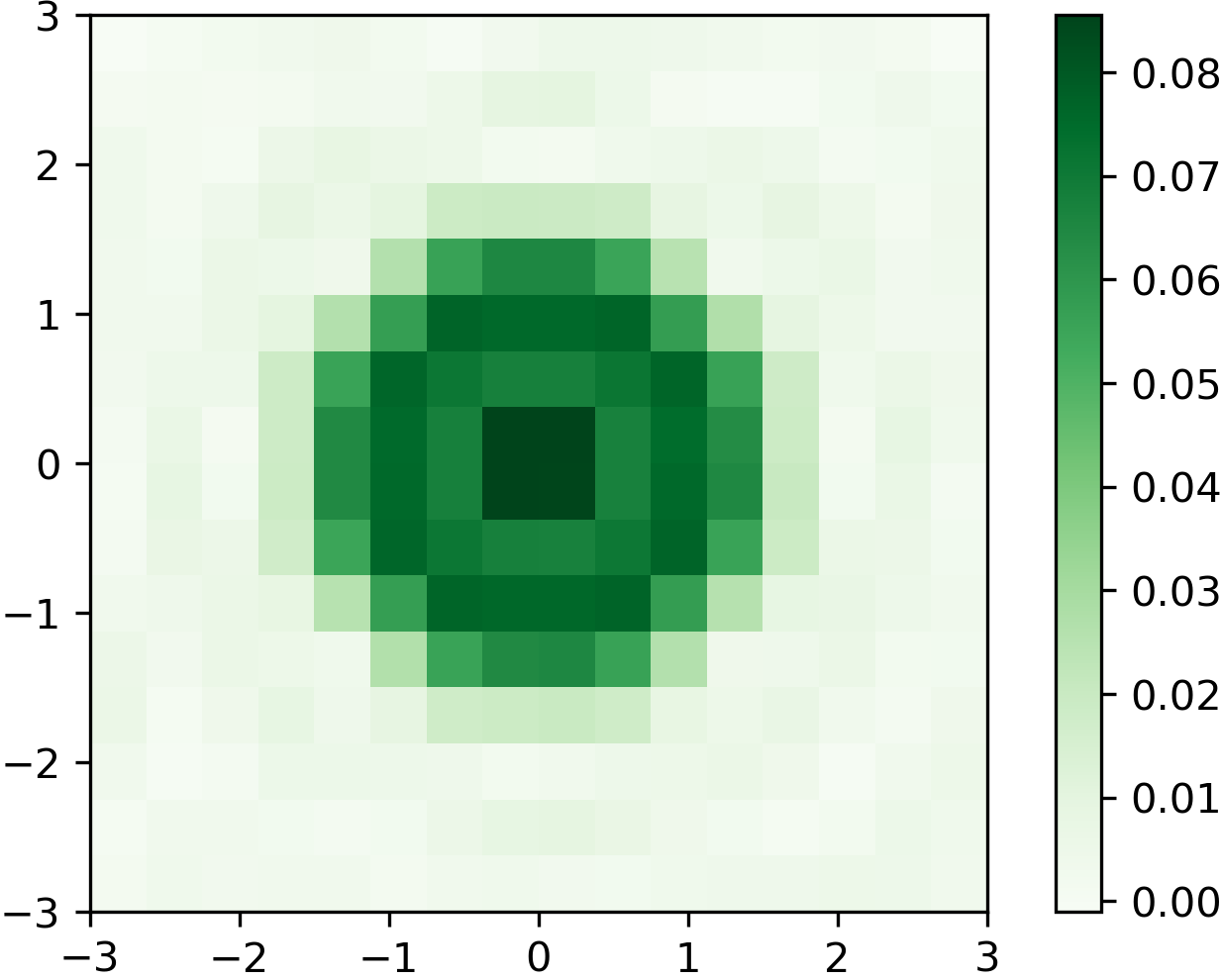}
  \subcaption{$B_1$, $k=3$, $\alpha=50$}
 \end{minipage}
 \begin{minipage}[b]{0.4\linewidth}
  \centering
  \includegraphics[keepaspectratio, scale=0.45]
  {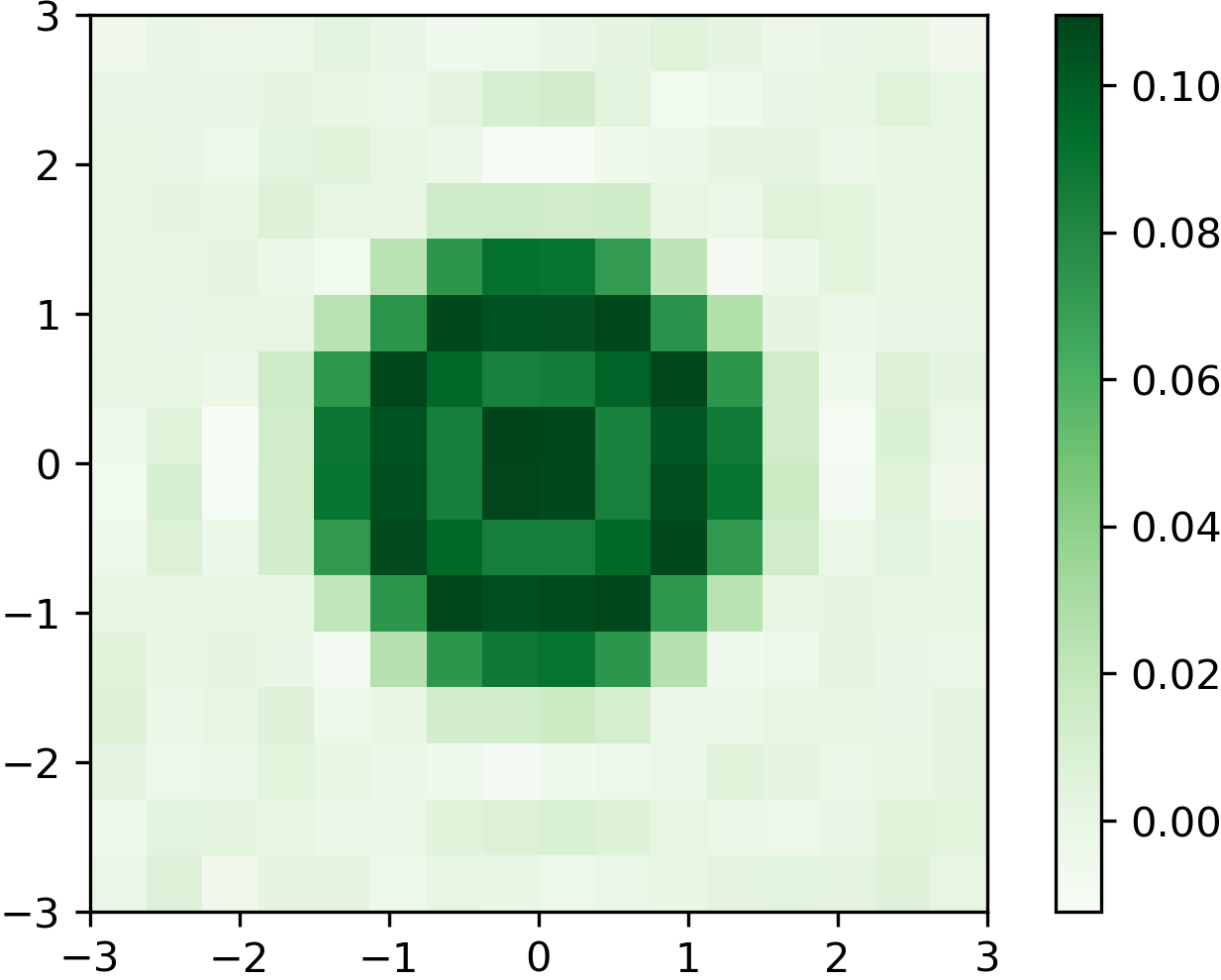}
  \subcaption{$B_1$, $k=3$, $\alpha=5$}
 \end{minipage}
 \begin{minipage}[b]{0.4\linewidth}
  \centering
  \includegraphics[keepaspectratio, scale=0.45]
  {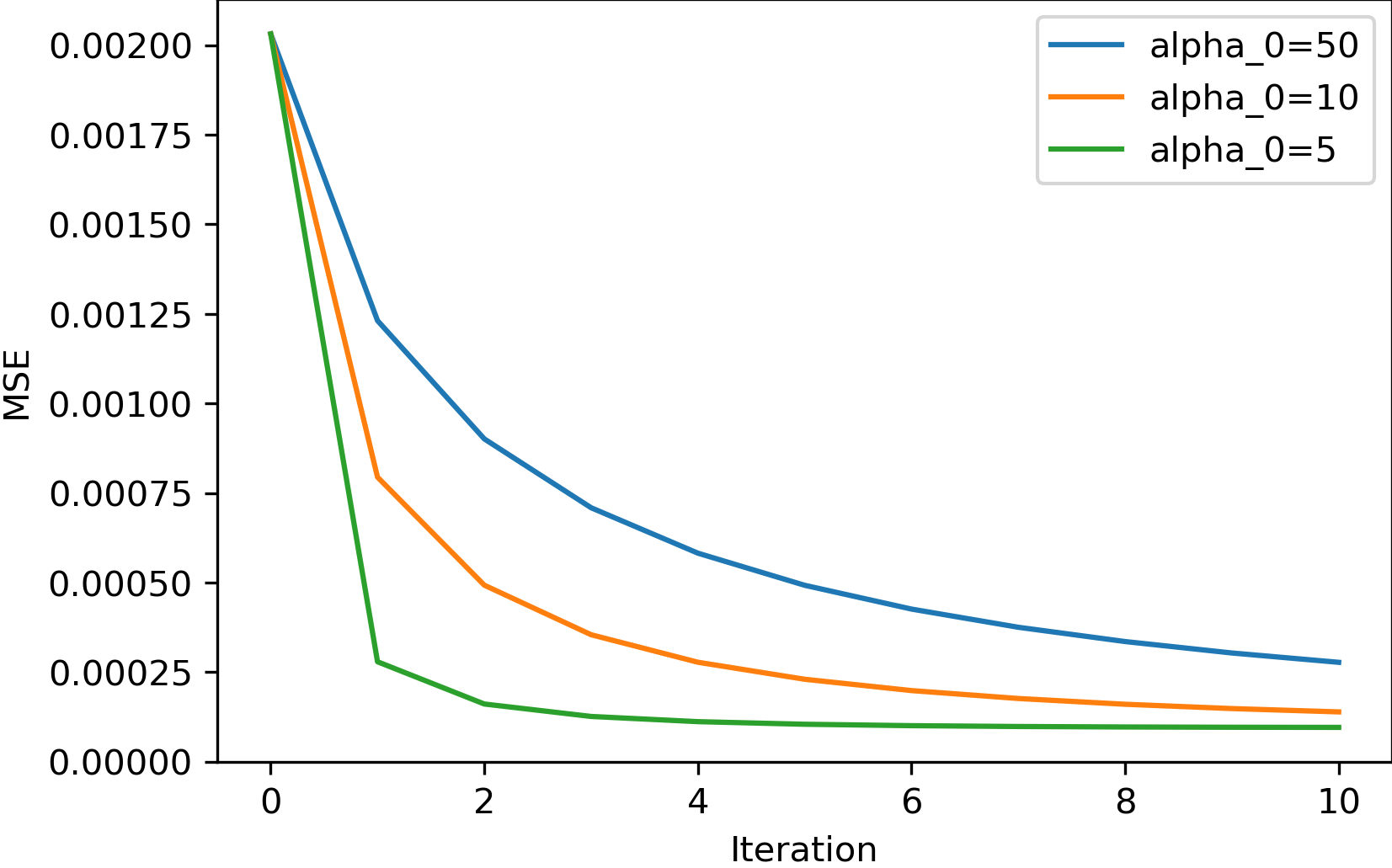}
  \subcaption{$B_1$, $k=3$, error graph}
 \end{minipage}
\end{tabular}

\begin{tabular}{c}
\hspace{-2.5cm}
 \begin{minipage}[b]{0.4\linewidth}
  \centering
  \includegraphics[keepaspectratio, scale=0.45]
  {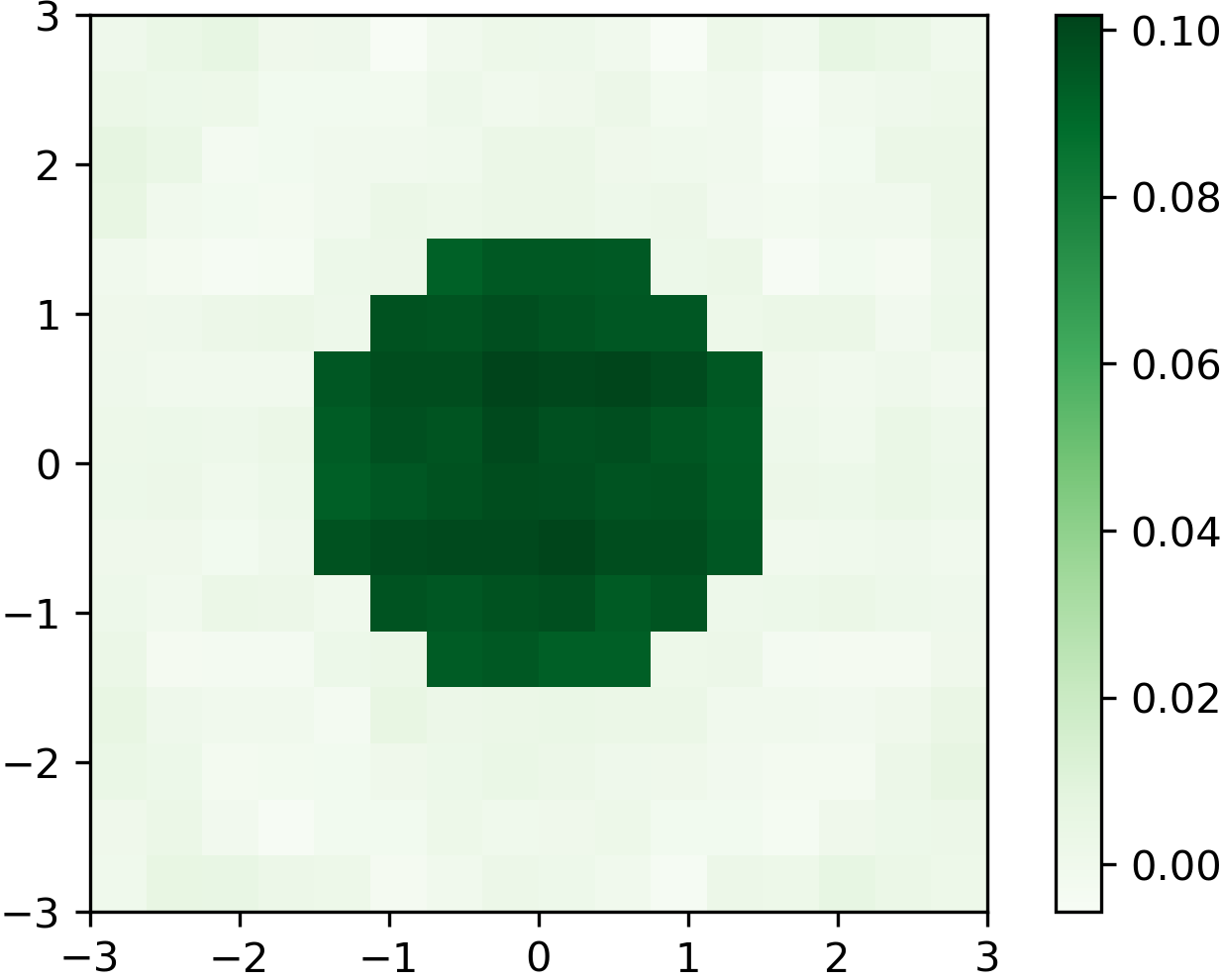}
  \subcaption{$B_1$, $k=7$, $\alpha=50$}
 \end{minipage}
 \begin{minipage}[b]{0.4\linewidth}
  \centering
  \includegraphics[keepaspectratio, scale=0.45]
  {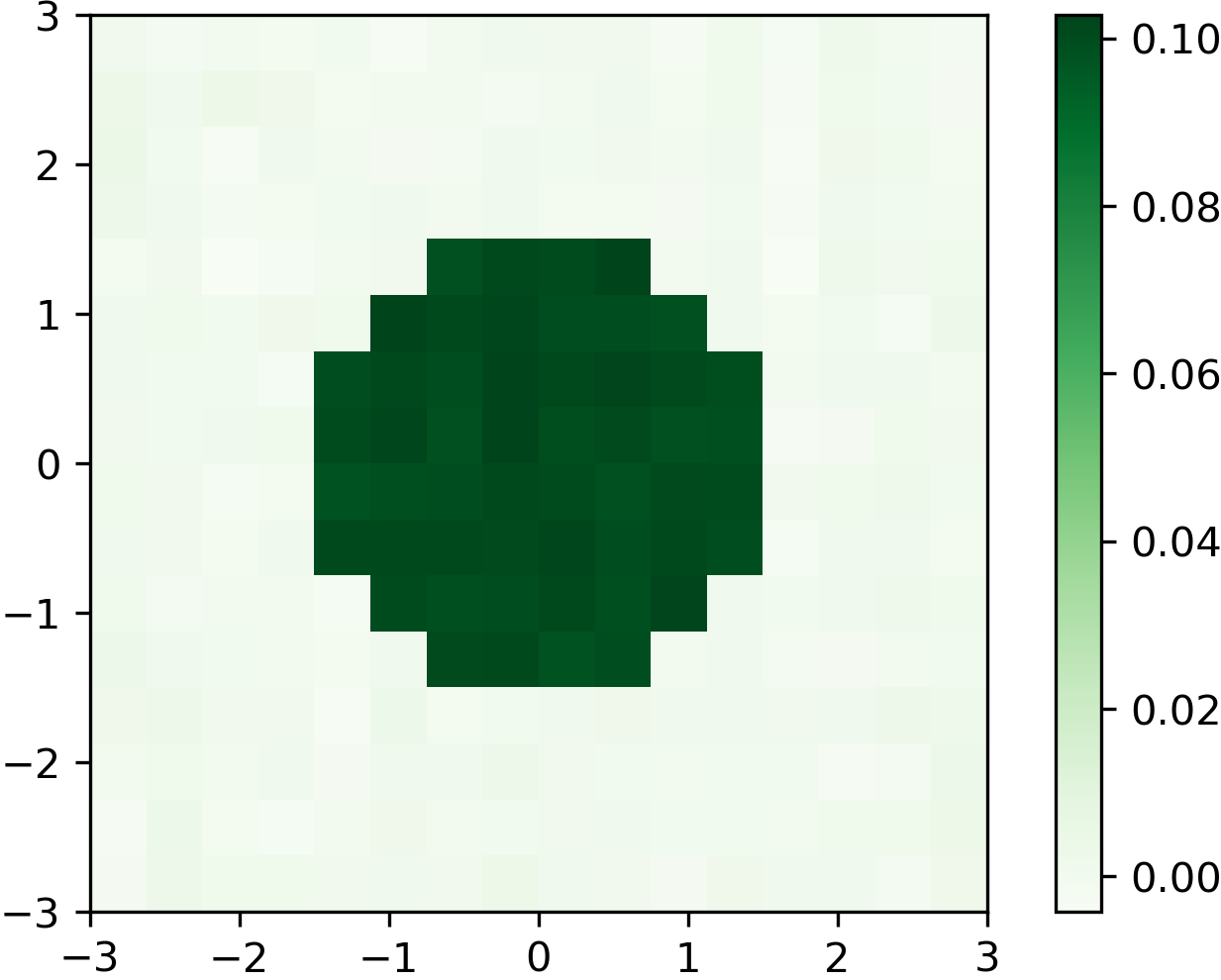}
  \subcaption{$B_1$, $k=7$, $\alpha=5$}
 \end{minipage}
 \begin{minipage}[b]{0.4\linewidth}
  \centering
  \includegraphics[keepaspectratio, scale=0.45]
  {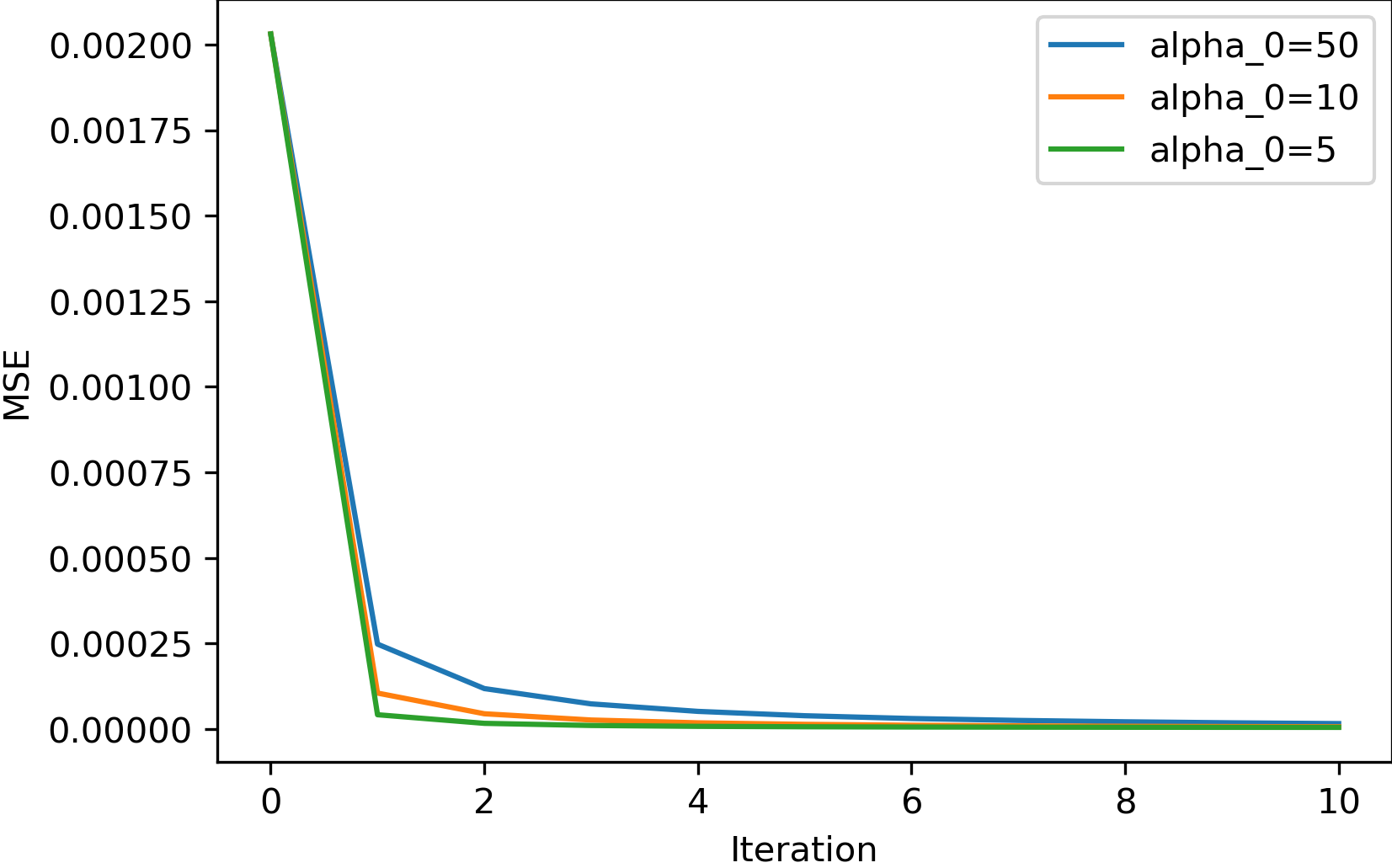}
  \subcaption{$B_1$, $k=7$, error graph}
 \end{minipage}
\end{tabular}

\begin{tabular}{c}
\hspace{-2.5cm}
 \begin{minipage}[b]{0.4\linewidth}
  \centering
  \includegraphics[keepaspectratio, scale=0.45]
  {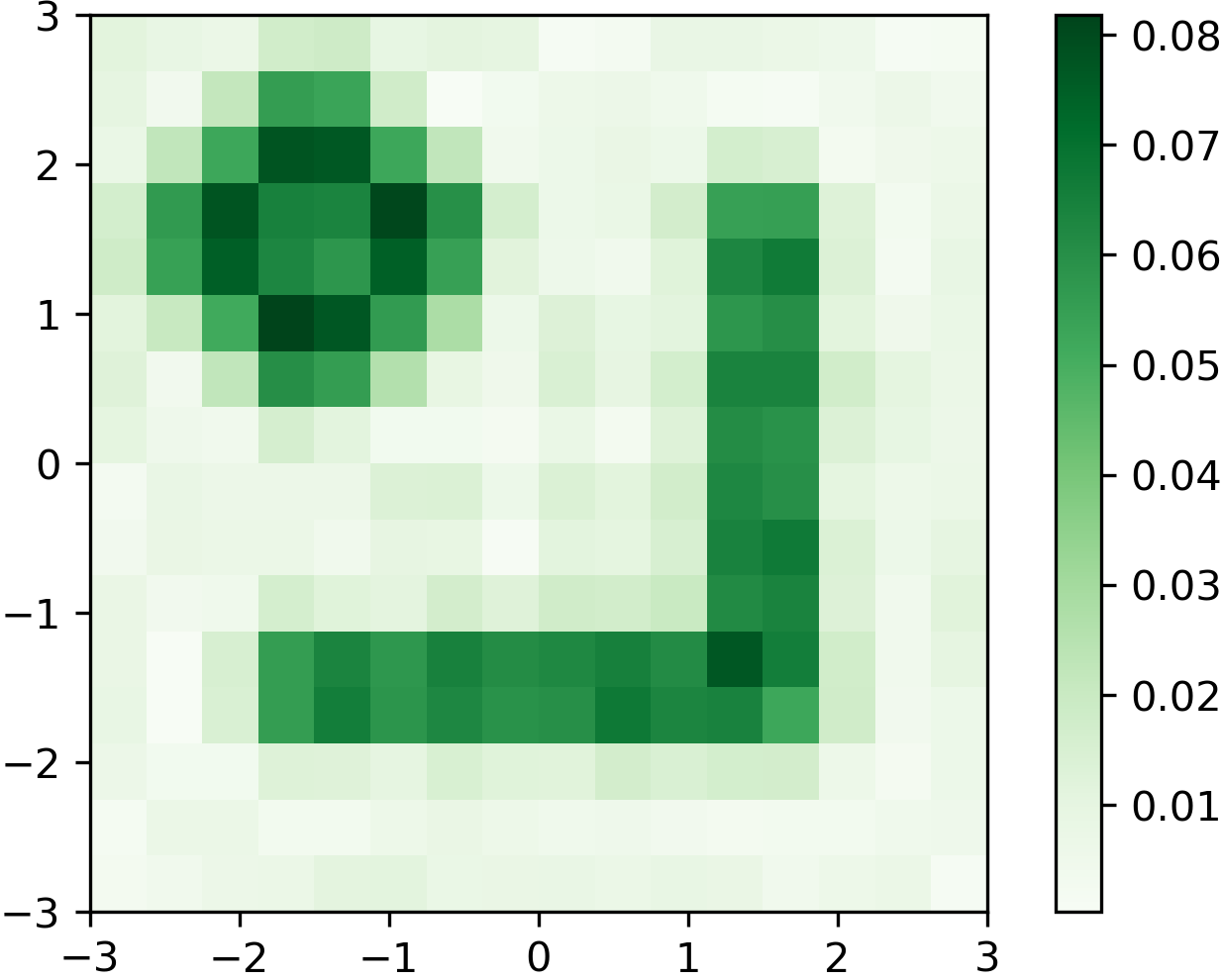}
  \subcaption{$B_2$, $k=3$, $\alpha=50$}
 \end{minipage}
 \begin{minipage}[b]{0.4\linewidth}
  \centering
  \includegraphics[keepaspectratio, scale=0.45]
  {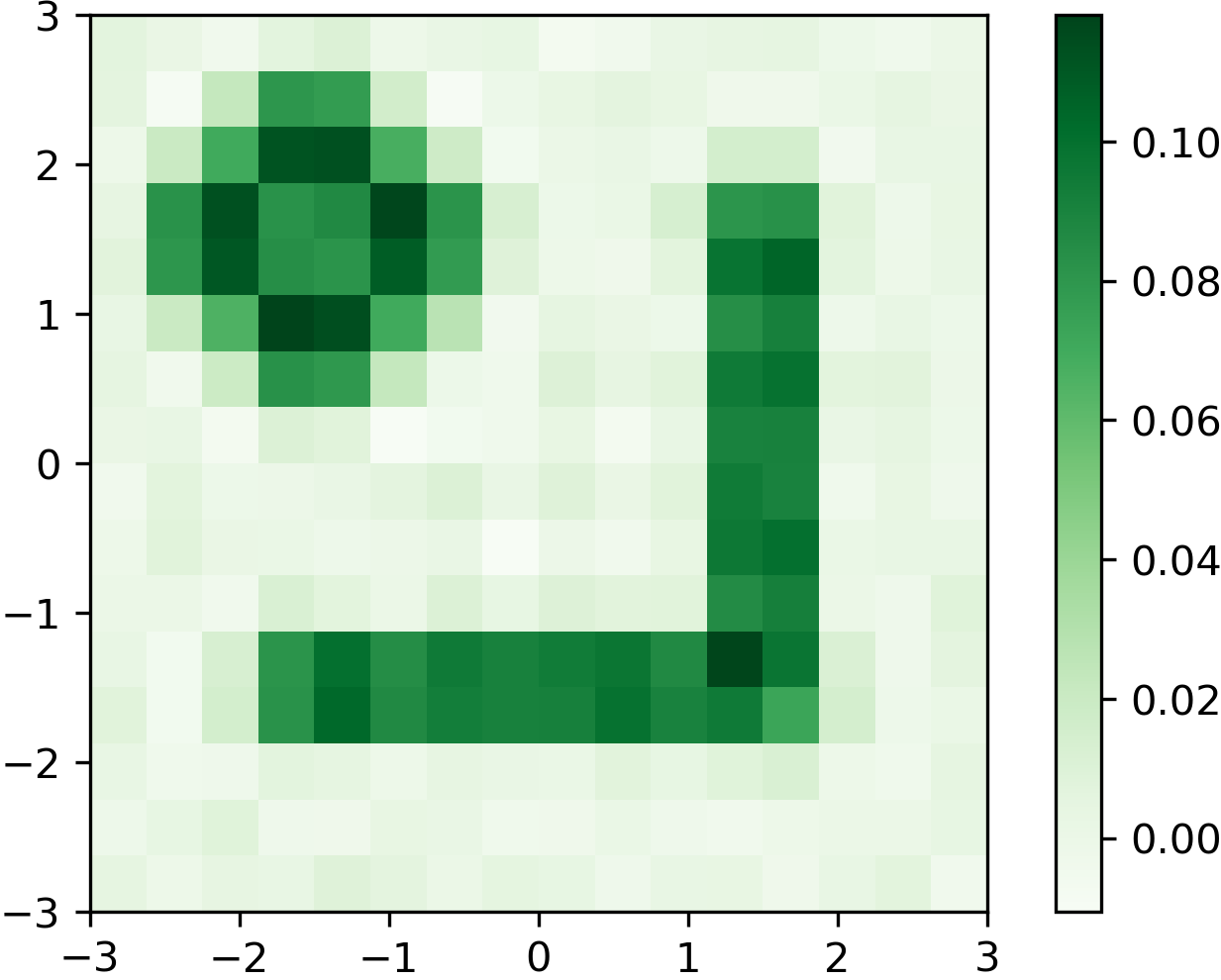}
  \subcaption{$B_2$, $k=3$, $\alpha=5$}
 \end{minipage}
 \begin{minipage}[b]{0.4\linewidth}
  \centering
  \includegraphics[keepaspectratio, scale=0.45]
  {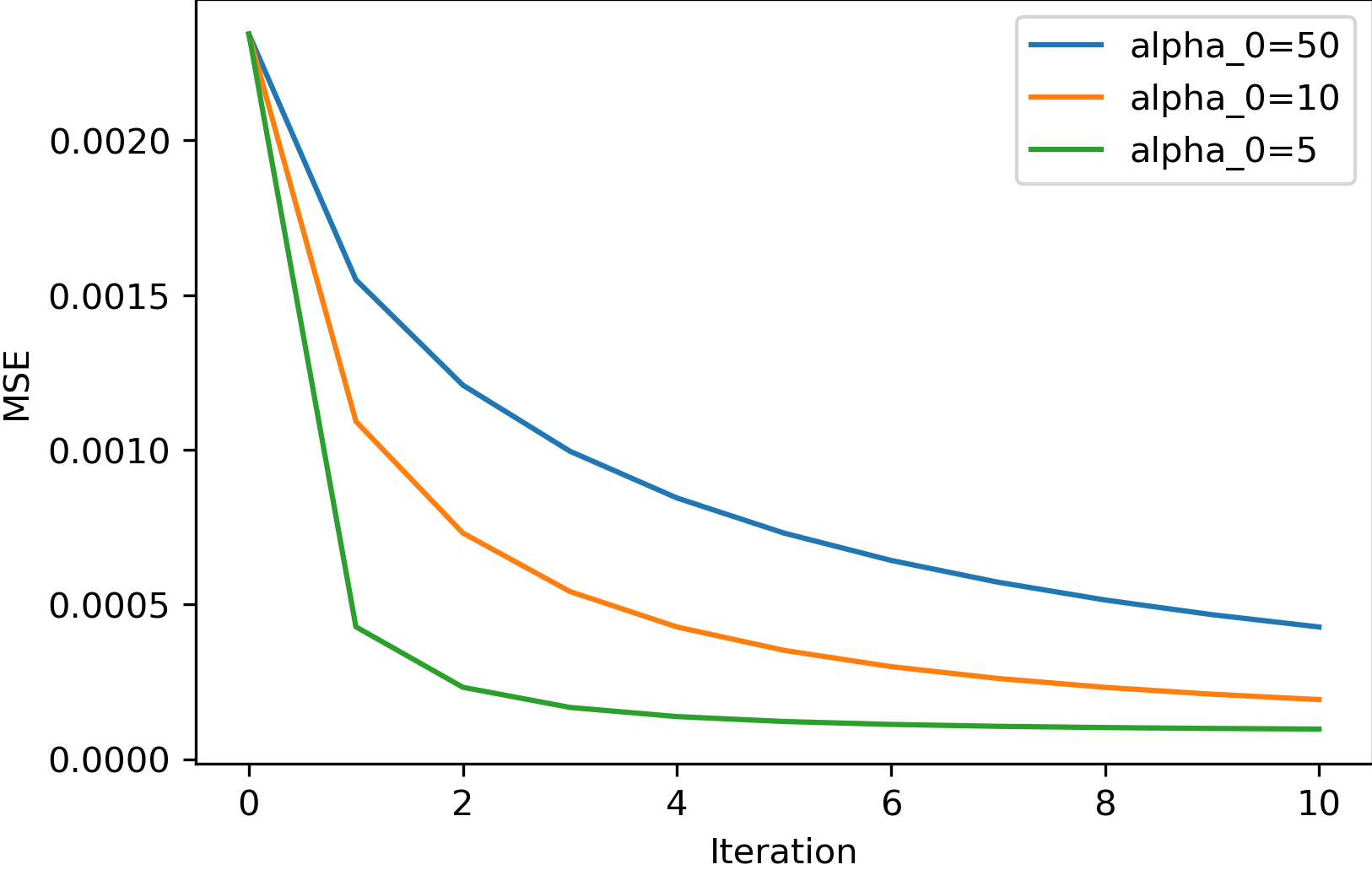}
  \subcaption{$B_2$, $k=3$, error graph}
 \end{minipage}
\end{tabular}

\begin{tabular}{c}
\hspace{-2.5cm}
 \begin{minipage}[b]{0.4\linewidth}
  \centering
  \includegraphics[keepaspectratio, scale=0.45]
  {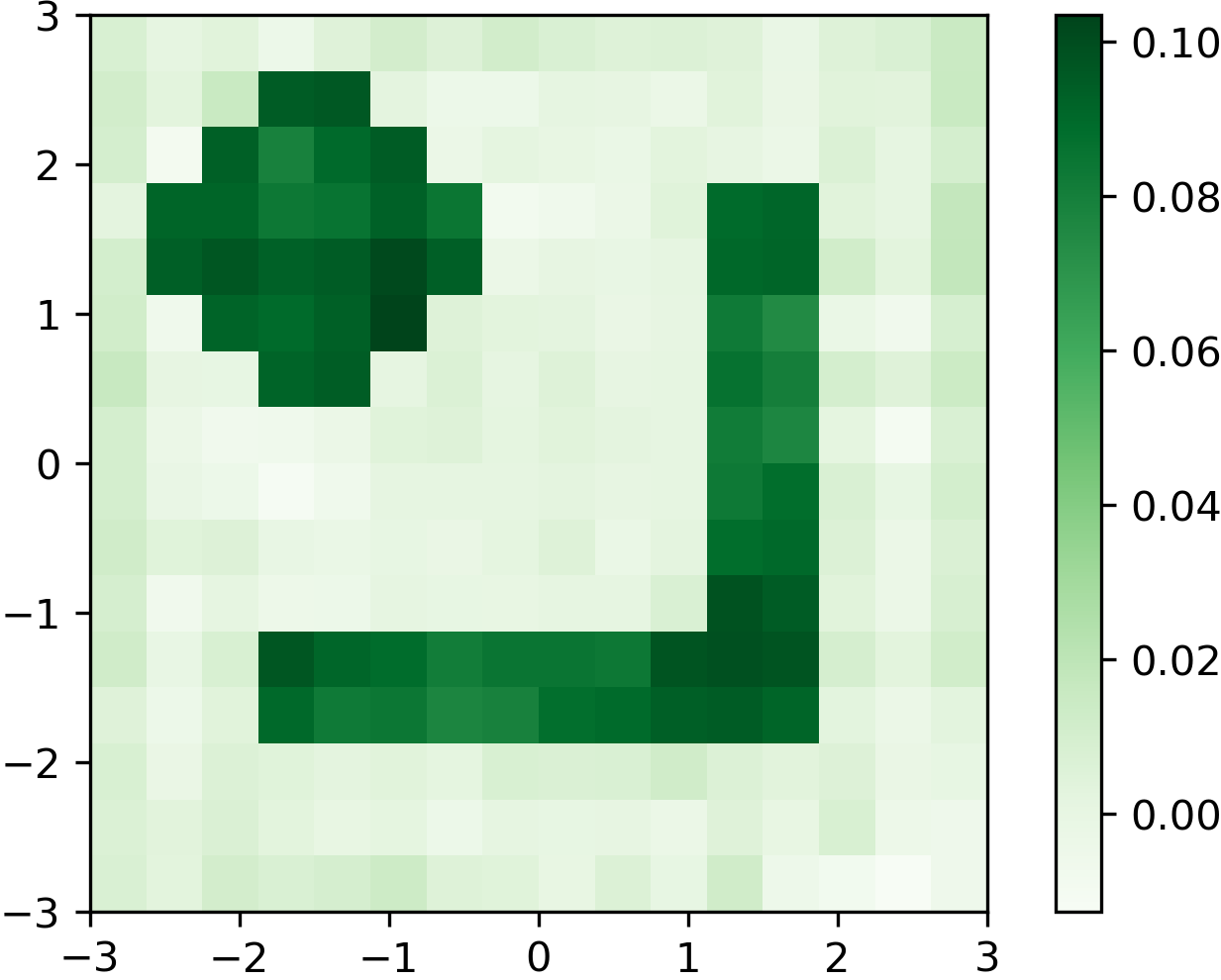}
  \subcaption{$B_2$, $k=7$, $\alpha=50$}
 \end{minipage}
 \begin{minipage}[b]{0.4\linewidth}
  \centering
  \includegraphics[keepaspectratio, scale=0.45]
  {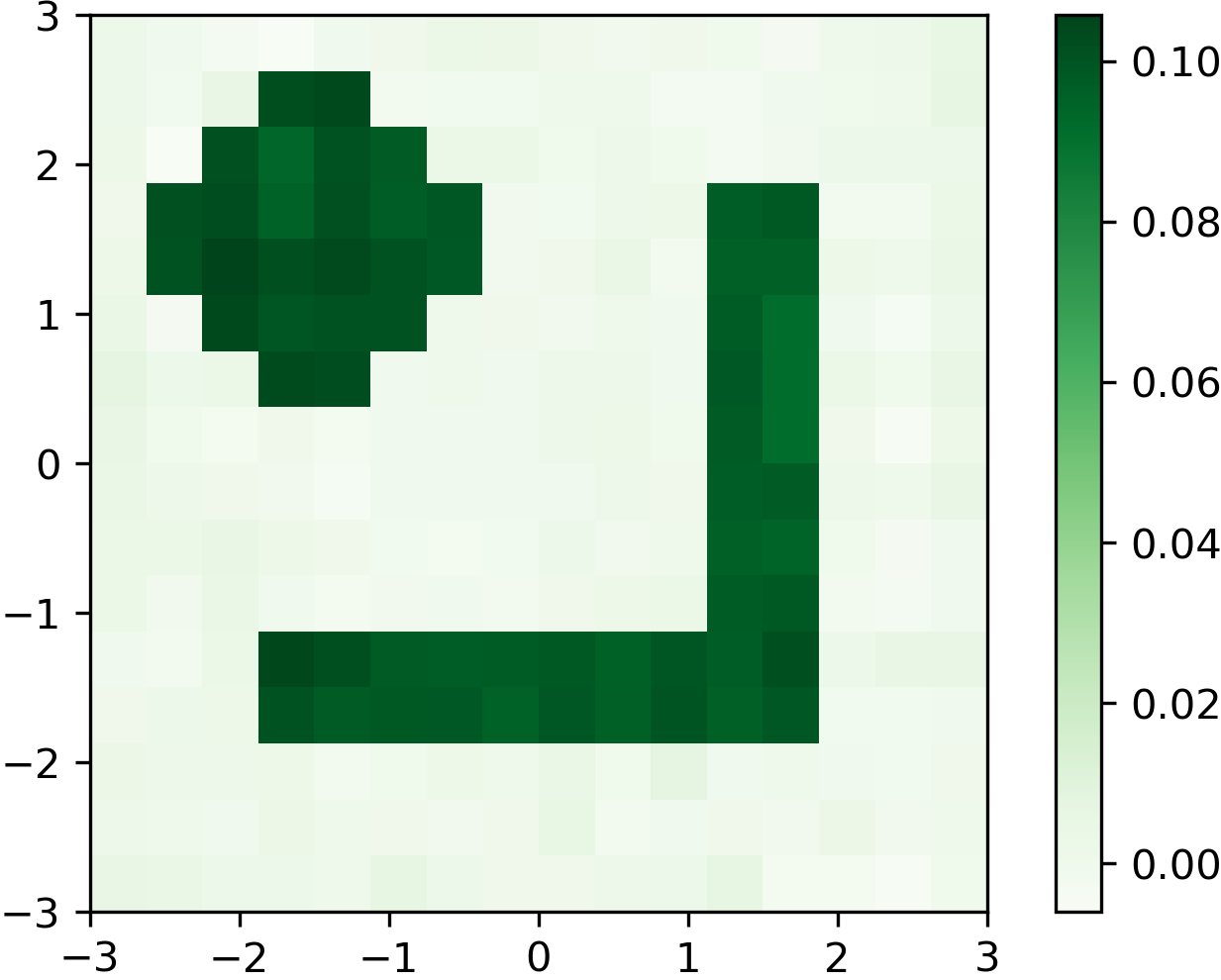}
  \subcaption{$B_2$, $k=7$, $\alpha=5$}
 \end{minipage}
 \begin{minipage}[b]{0.4\linewidth}
  \centering
  \includegraphics[keepaspectratio, scale=0.45]
  {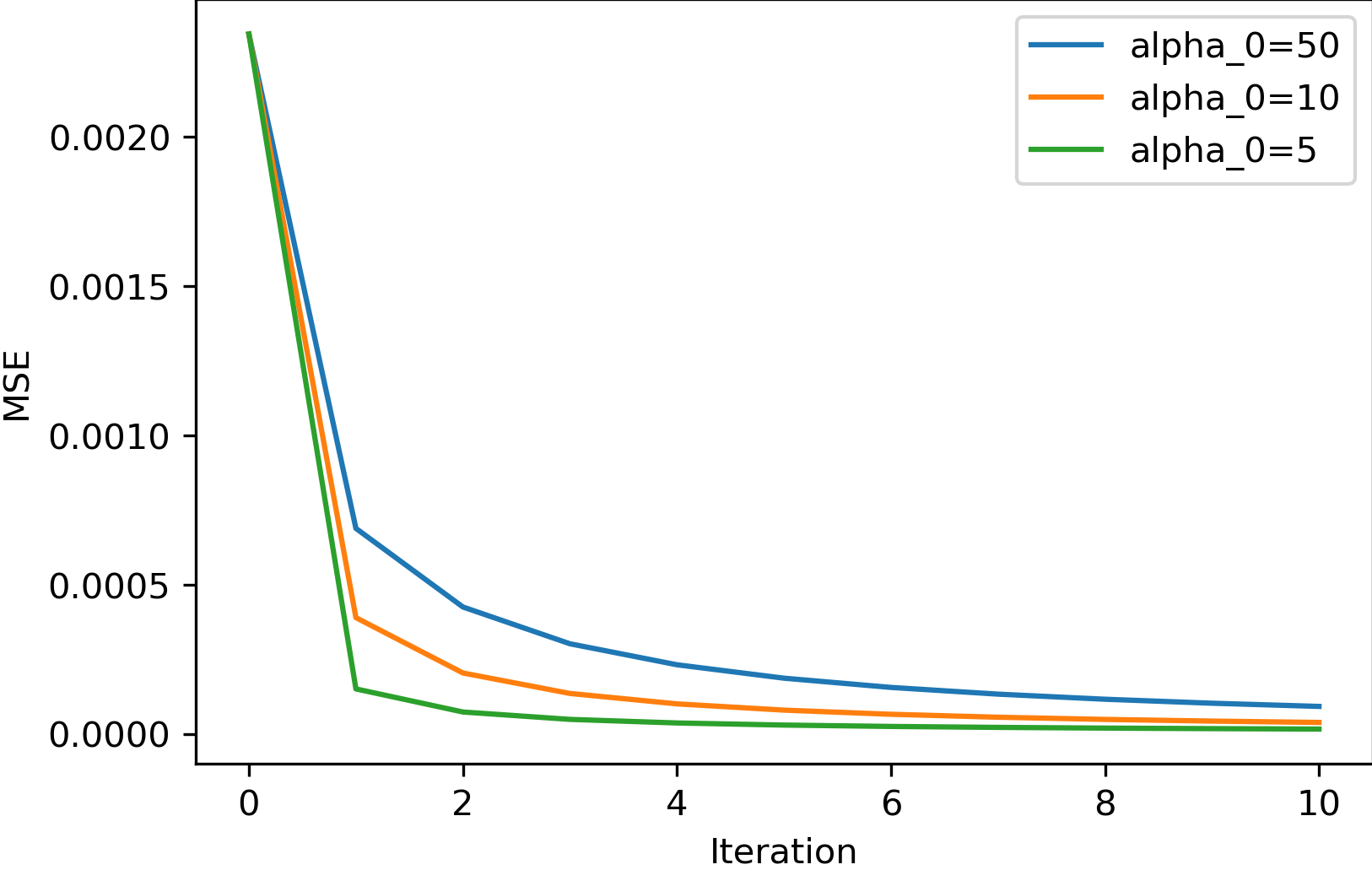}
  \subcaption{$B_2$, $k=7$, error graph}
 \end{minipage}
\end{tabular}
\caption{EKF (nosiy $\sigma=0.01$)}
\label{EKFn001}
\end{figure}

\begin{figure}[h]
\begin{tabular}{c}
\hspace{-2.5cm}
 \begin{minipage}[b]{0.4\linewidth}
  \centering
  \includegraphics[keepaspectratio, scale=0.45]
  {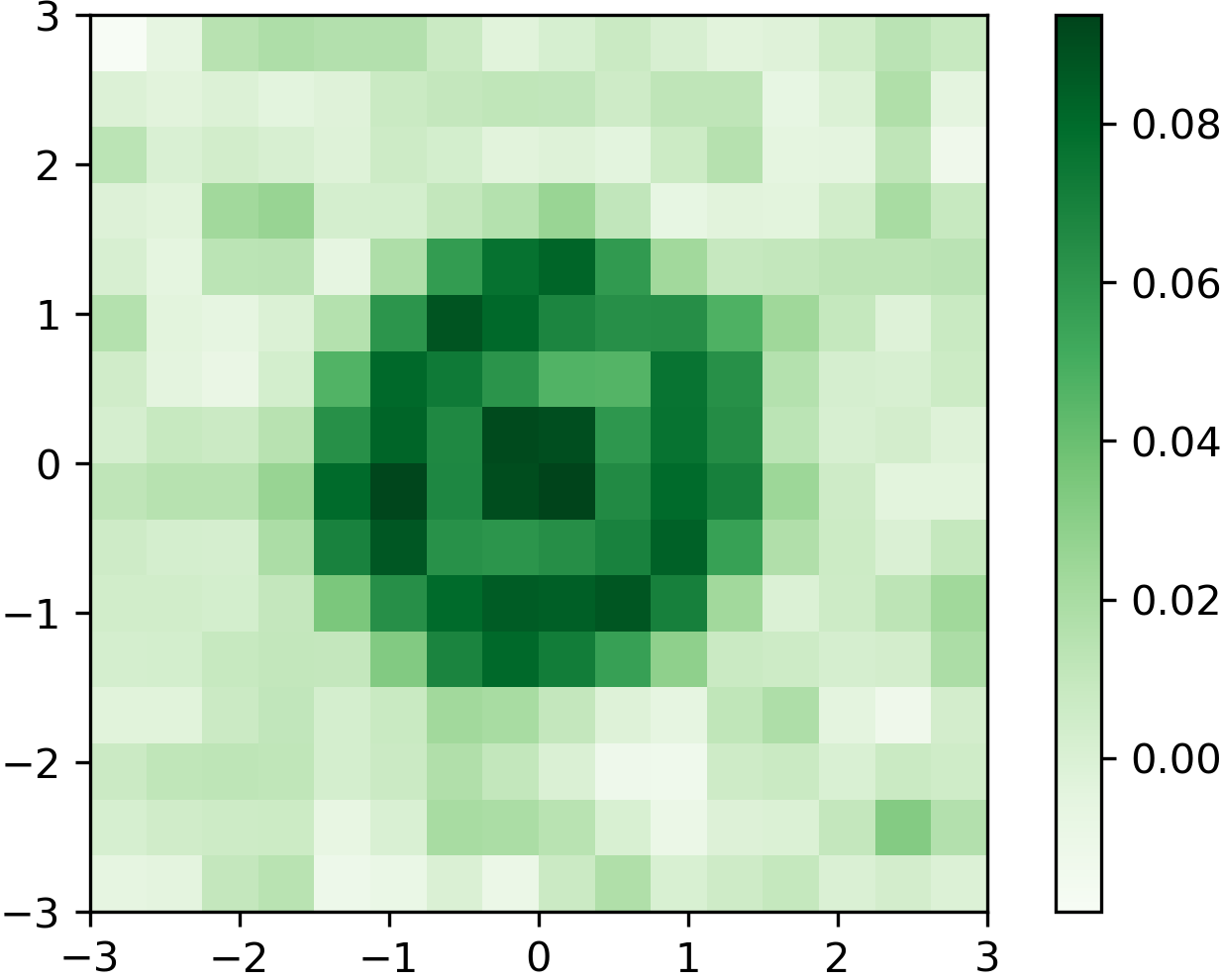}
  \subcaption{$B_1$, $k=3$, $\alpha=50$}
 \end{minipage}
 \begin{minipage}[b]{0.4\linewidth}
  \centering
  \includegraphics[keepaspectratio, scale=0.45]
  {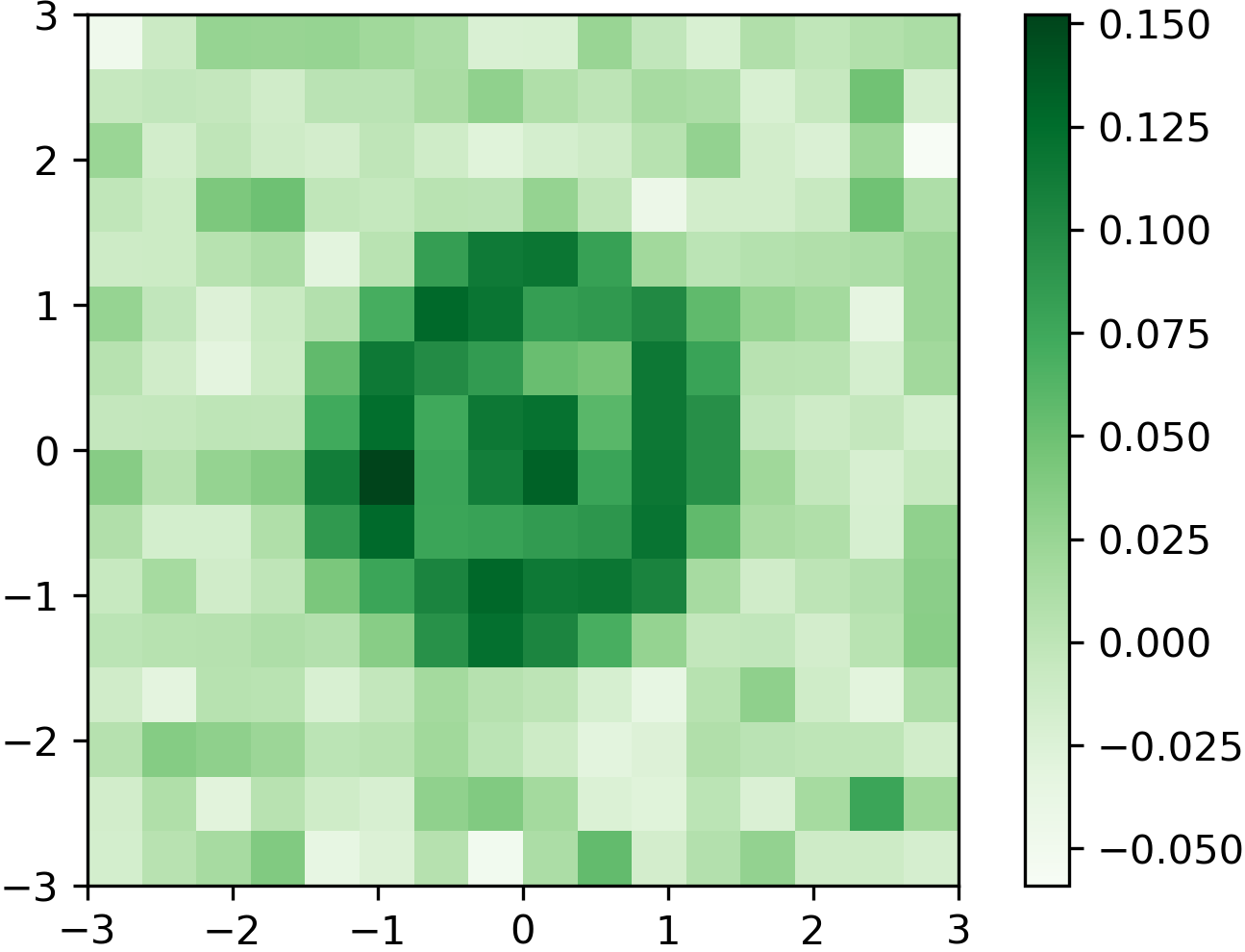}
  \subcaption{$B_1$, $k=3$, $\alpha=5$}
 \end{minipage}
 \begin{minipage}[b]{0.4\linewidth}
  \centering
  \includegraphics[keepaspectratio, scale=0.45]
  {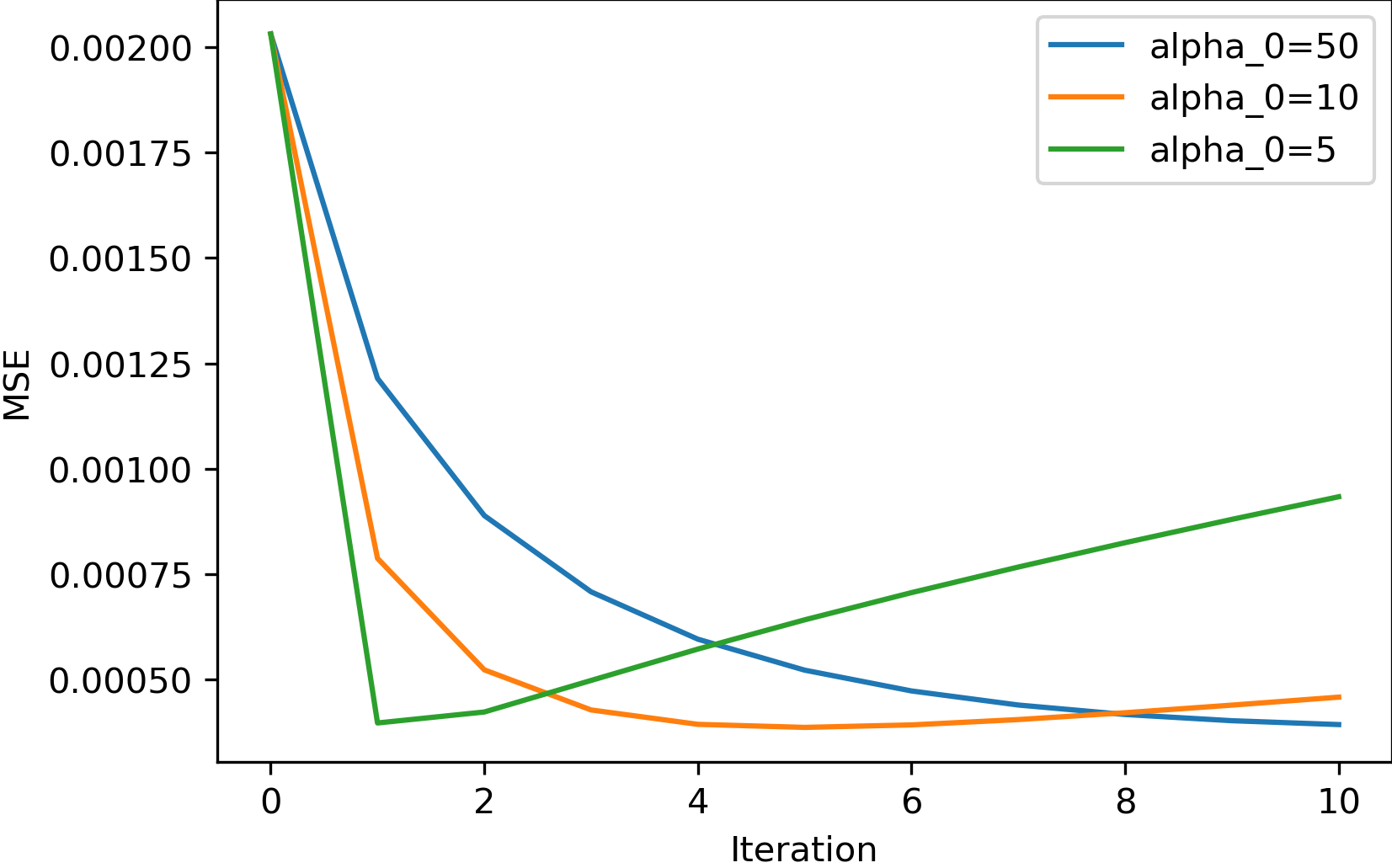}
  \subcaption{$B_1$, $k=3$, error graph}
 \end{minipage}
\end{tabular}

\begin{tabular}{c}
\hspace{-2.5cm}
 \begin{minipage}[b]{0.4\linewidth}
  \centering
  \includegraphics[keepaspectratio, scale=0.45]
  {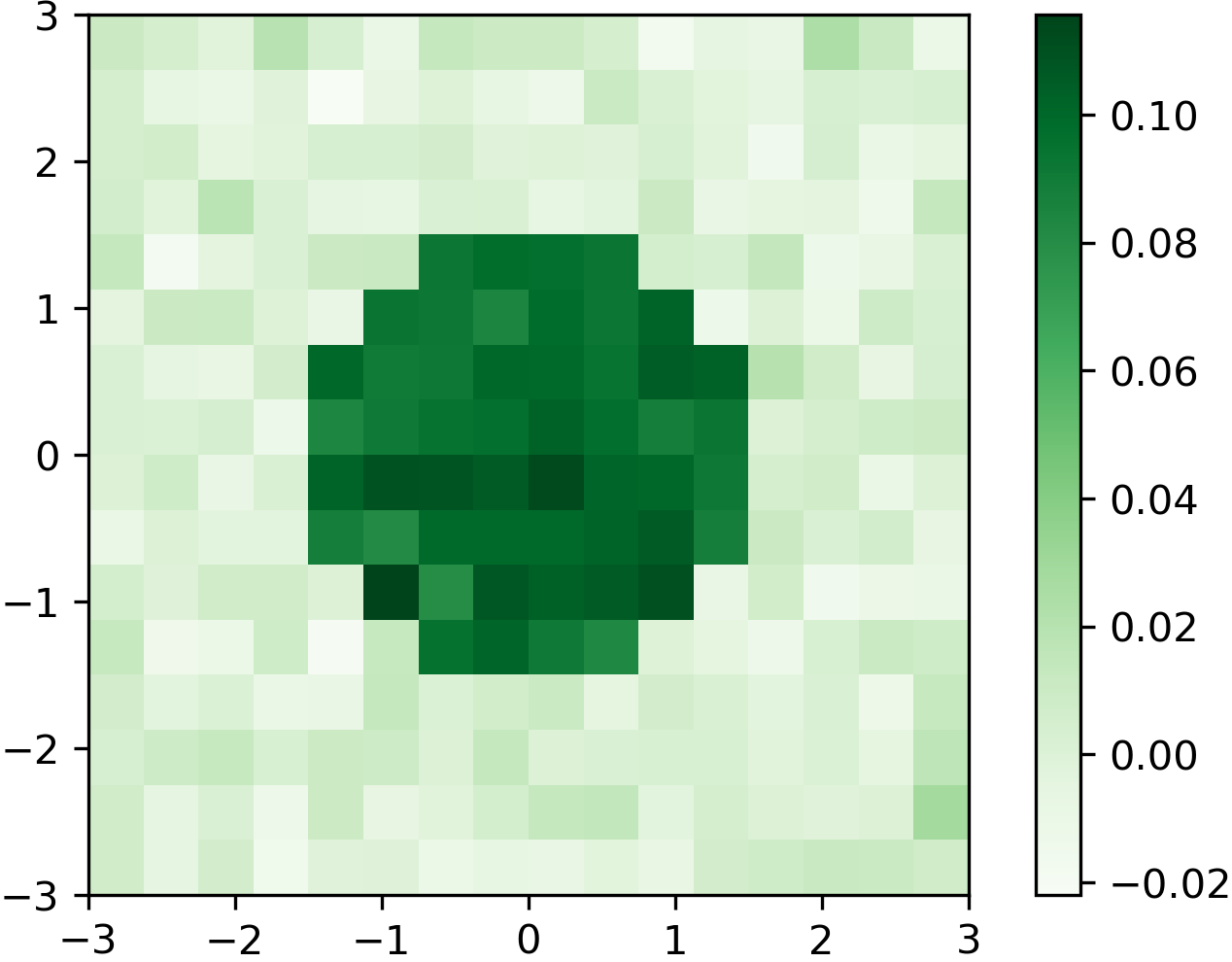}
  \subcaption{$B_1$, $k=7$, $\alpha=50$}
 \end{minipage}
 \begin{minipage}[b]{0.4\linewidth}
  \centering
  \includegraphics[keepaspectratio, scale=0.45]
  {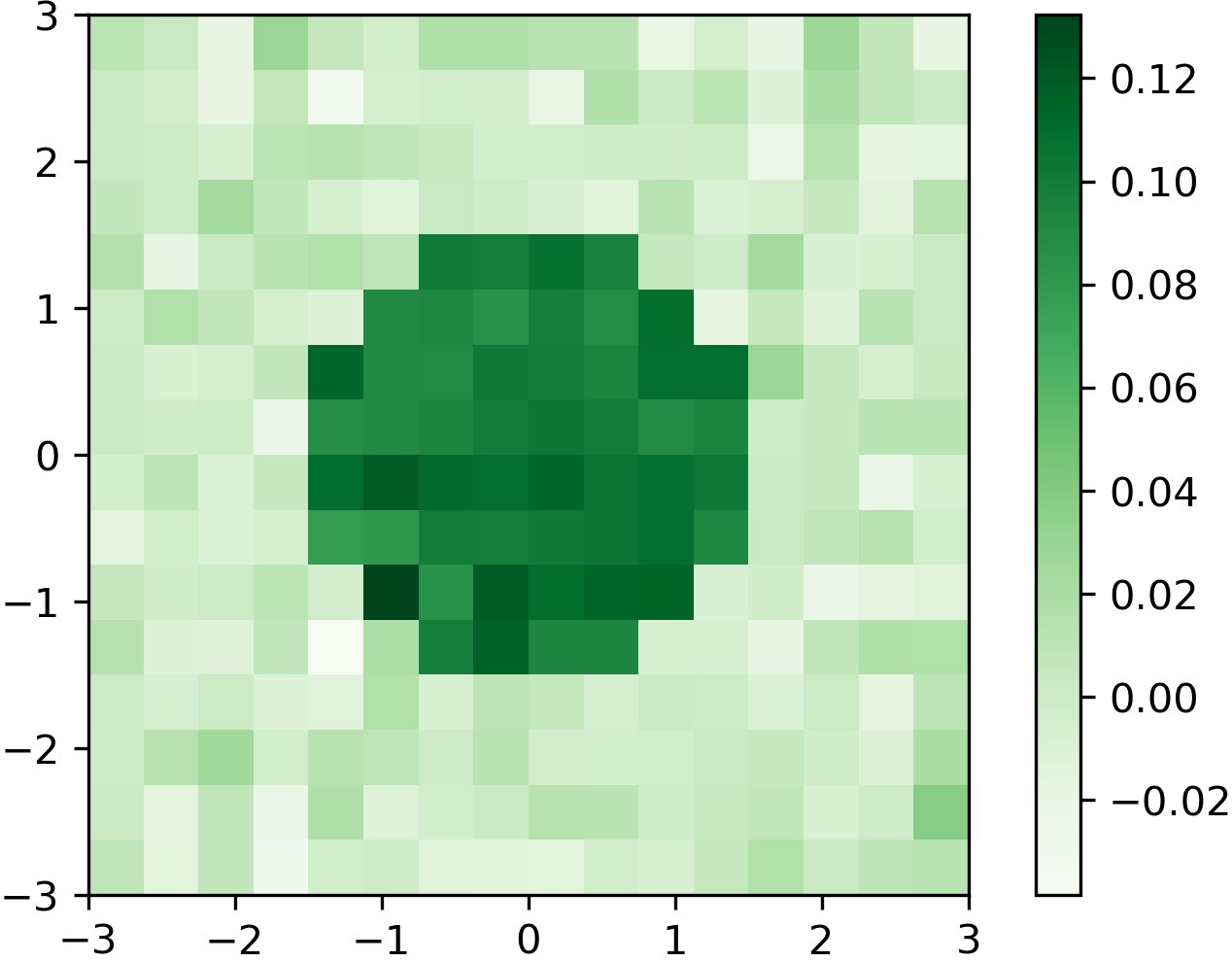}
  \subcaption{$B_1$, $k=7$, $\alpha=5$}
 \end{minipage}
 \begin{minipage}[b]{0.4\linewidth}
  \centering
  \includegraphics[keepaspectratio, scale=0.45]
  {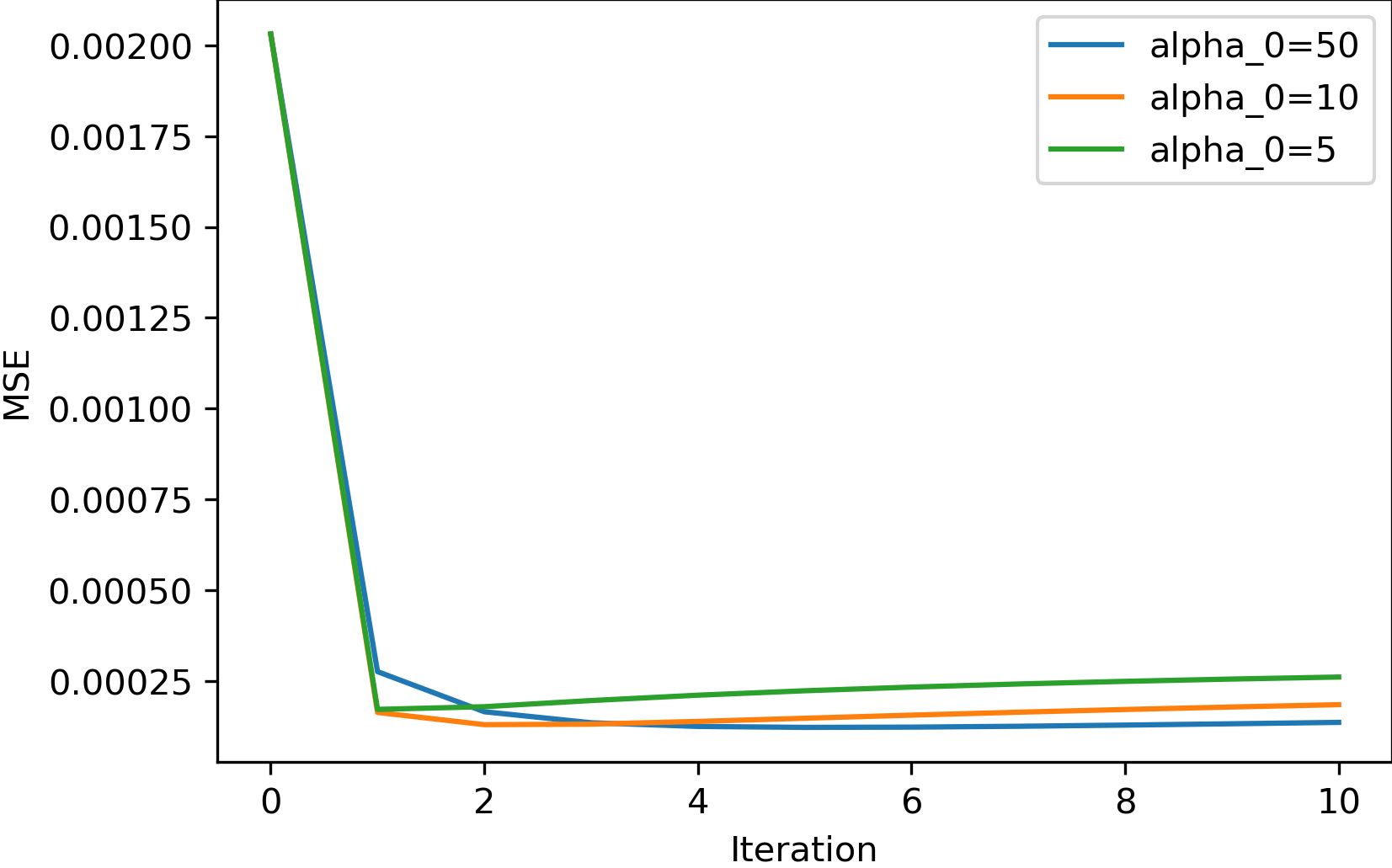}
  \subcaption{$B_1$, $k=7$, error graph}
 \end{minipage}
\end{tabular}

\begin{tabular}{c}
\hspace{-2.5cm}
 \begin{minipage}[b]{0.4\linewidth}
  \centering
  \includegraphics[keepaspectratio, scale=0.45]
  {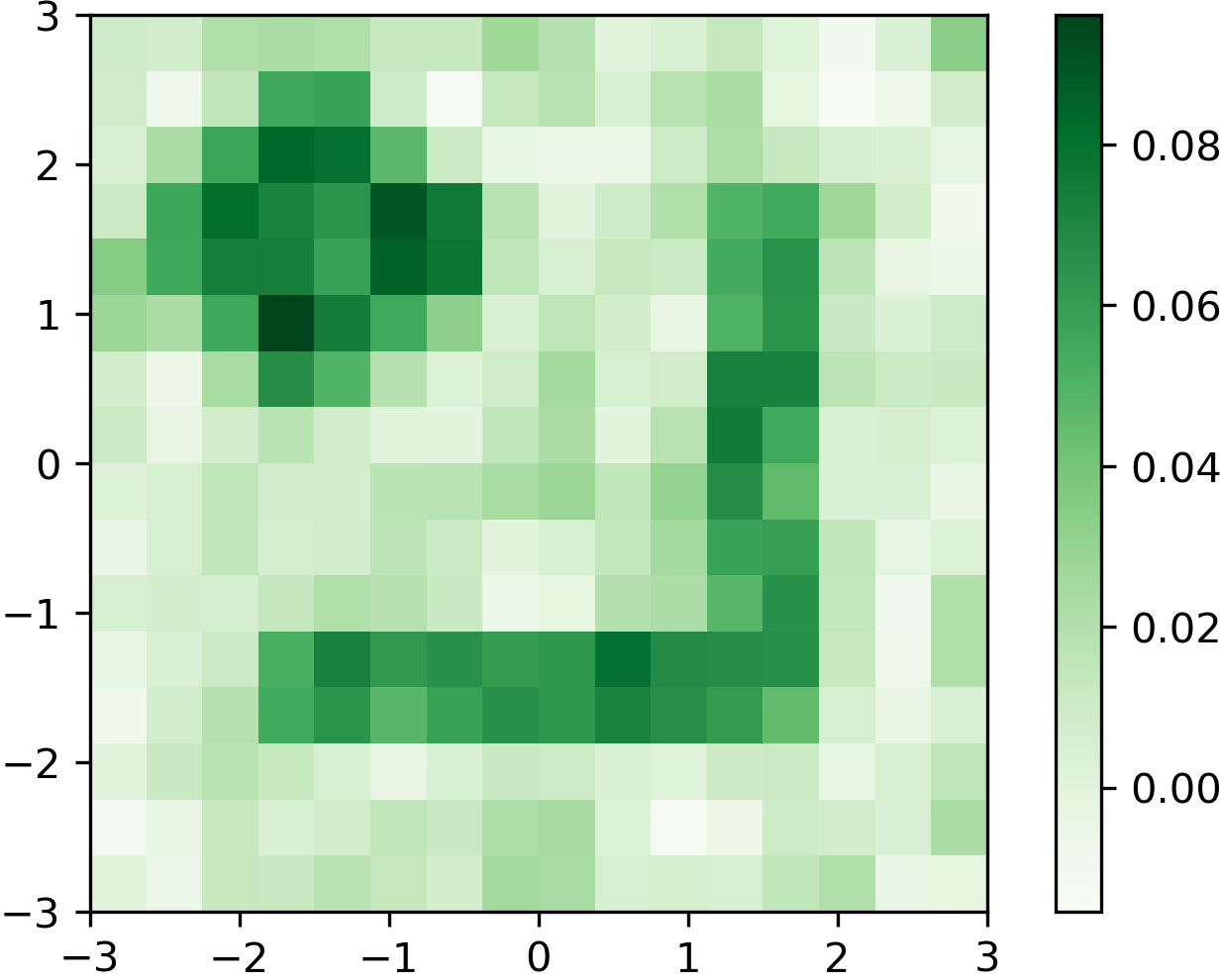}
  \subcaption{$B_2$, $k=3$, $\alpha=50$}
 \end{minipage}
 \begin{minipage}[b]{0.4\linewidth}
  \centering
  \includegraphics[keepaspectratio, scale=0.45]
  {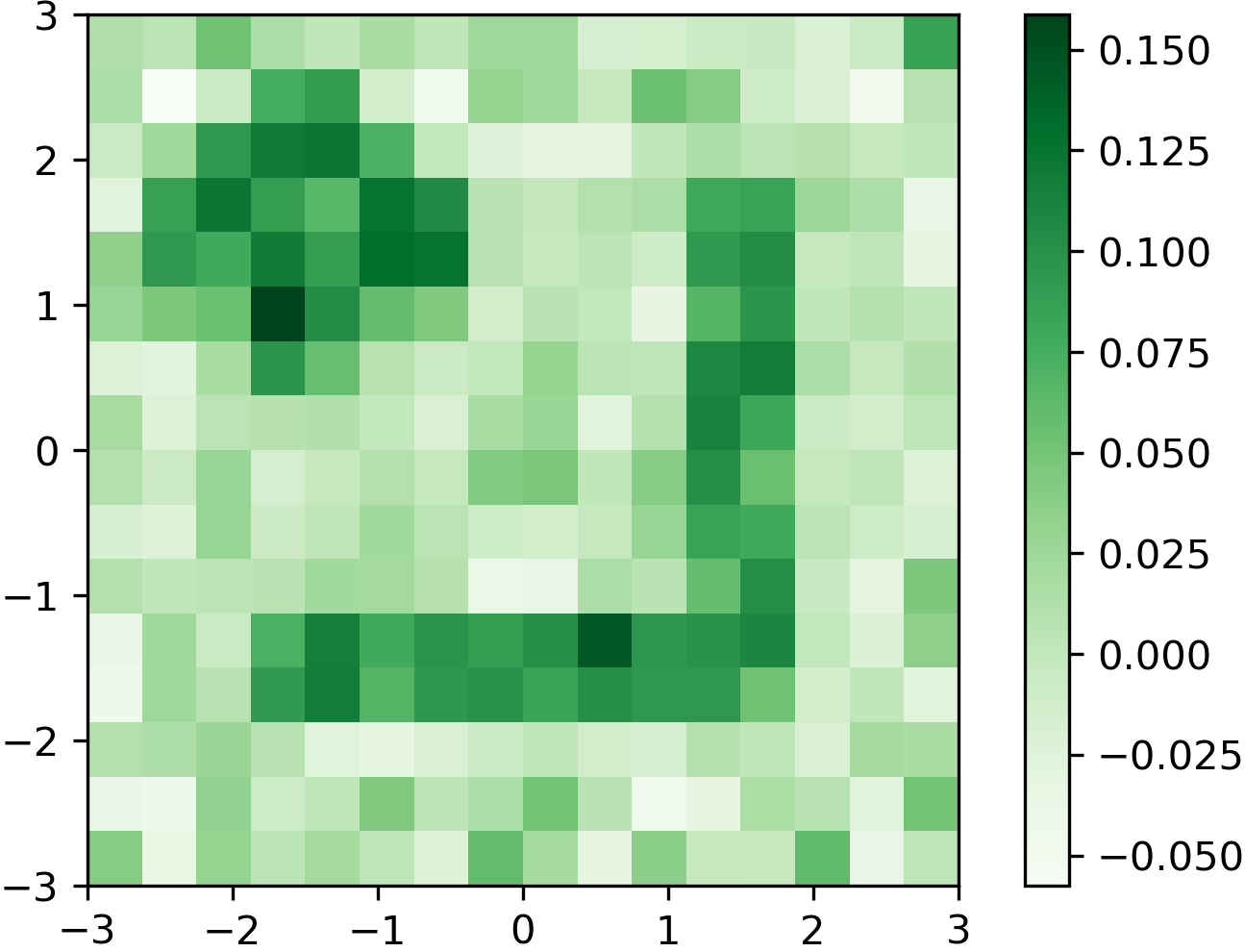}
  \subcaption{$B_2$, $k=3$, $\alpha=5$}
 \end{minipage}
 \begin{minipage}[b]{0.4\linewidth}
  \centering
  \includegraphics[keepaspectratio, scale=0.45]
  {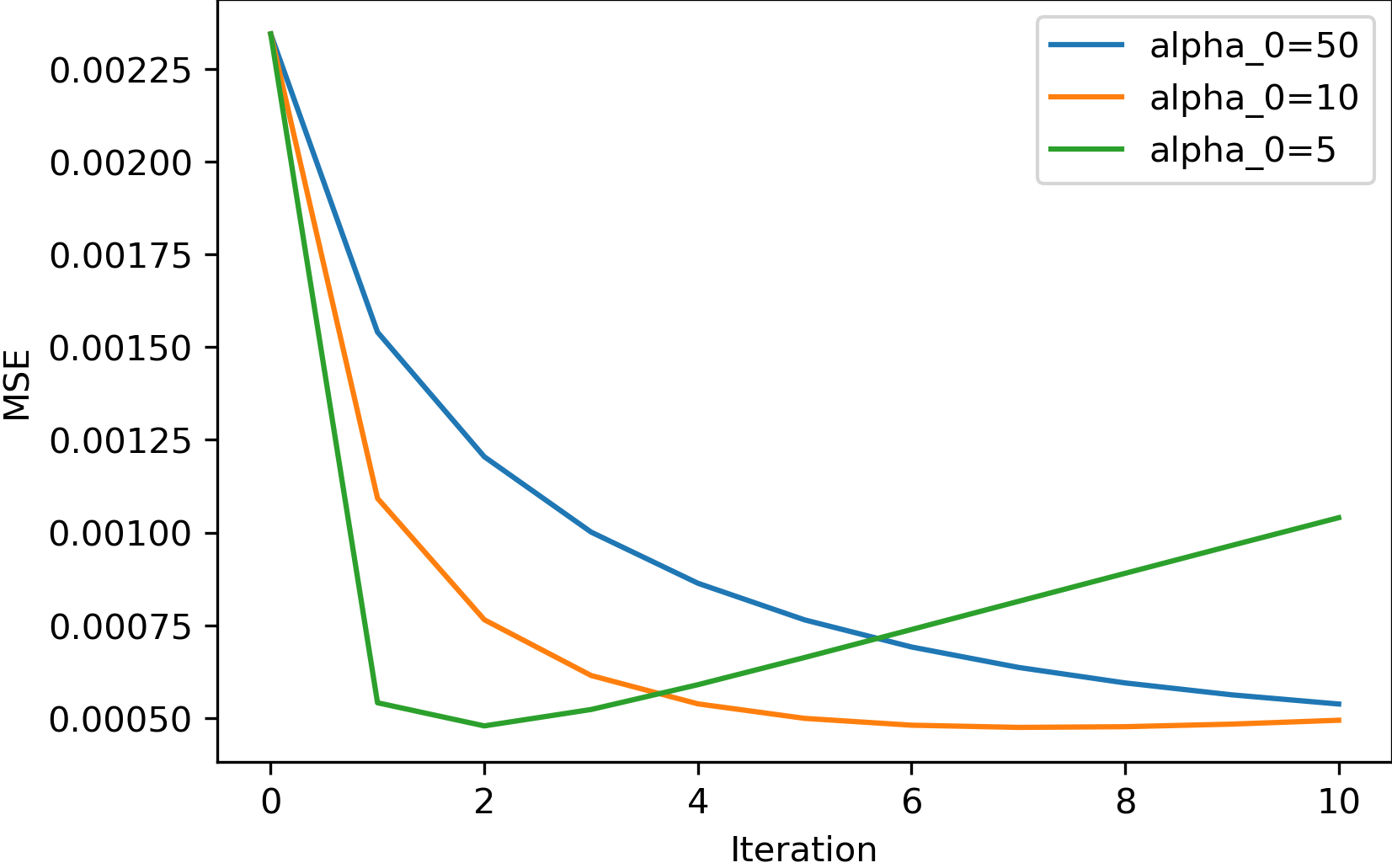}
  \subcaption{$B_2$, $k=3$, error graph}
 \end{minipage}
\end{tabular}

\begin{tabular}{c}
\hspace{-2.5cm}
 \begin{minipage}[b]{0.4\linewidth}
  \centering
  \includegraphics[keepaspectratio, scale=0.45]
  {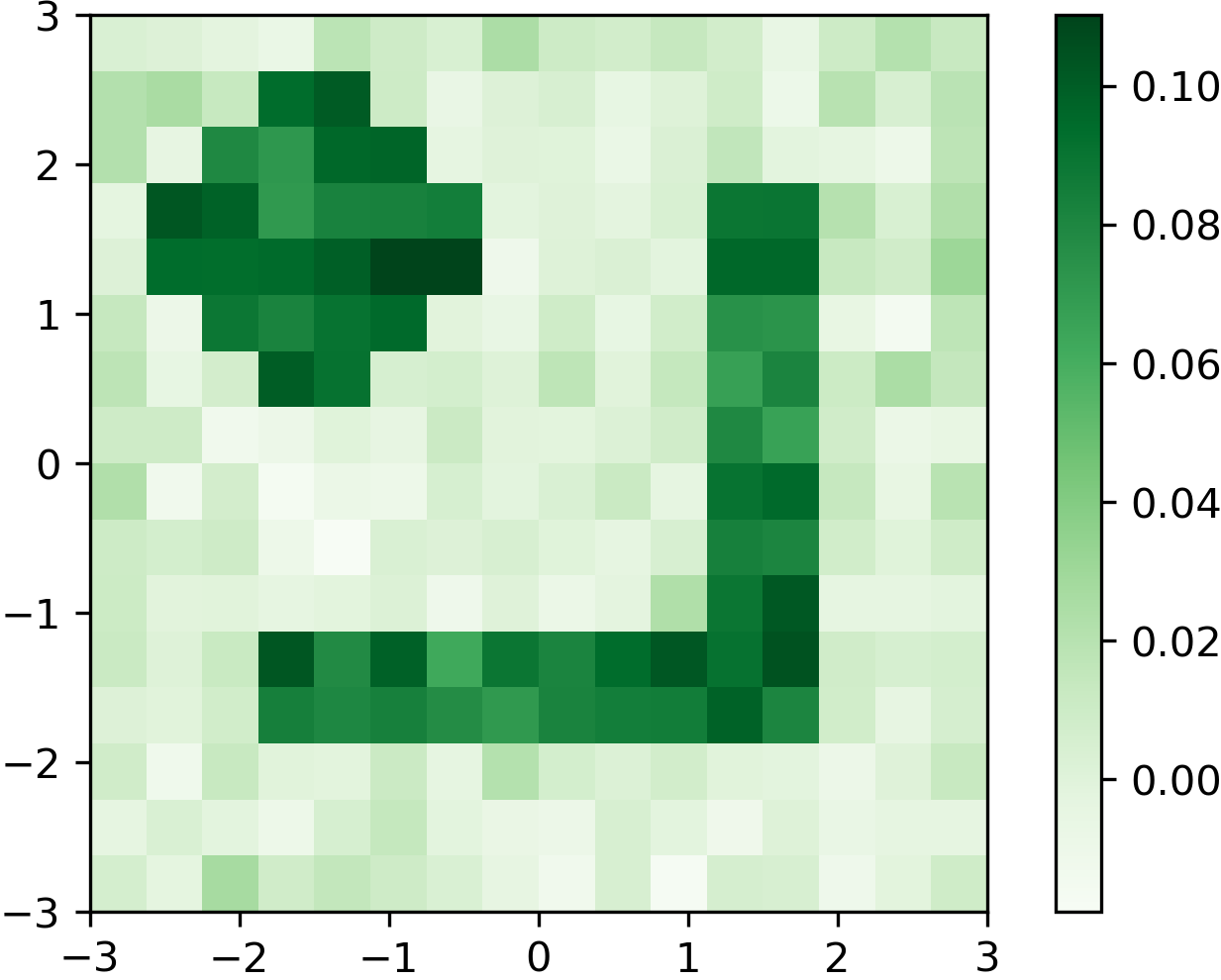}
  \subcaption{$B_2$, $k=7$, $\alpha=50$}
 \end{minipage}
 \begin{minipage}[b]{0.4\linewidth}
  \centering
  \includegraphics[keepaspectratio, scale=0.45]
  {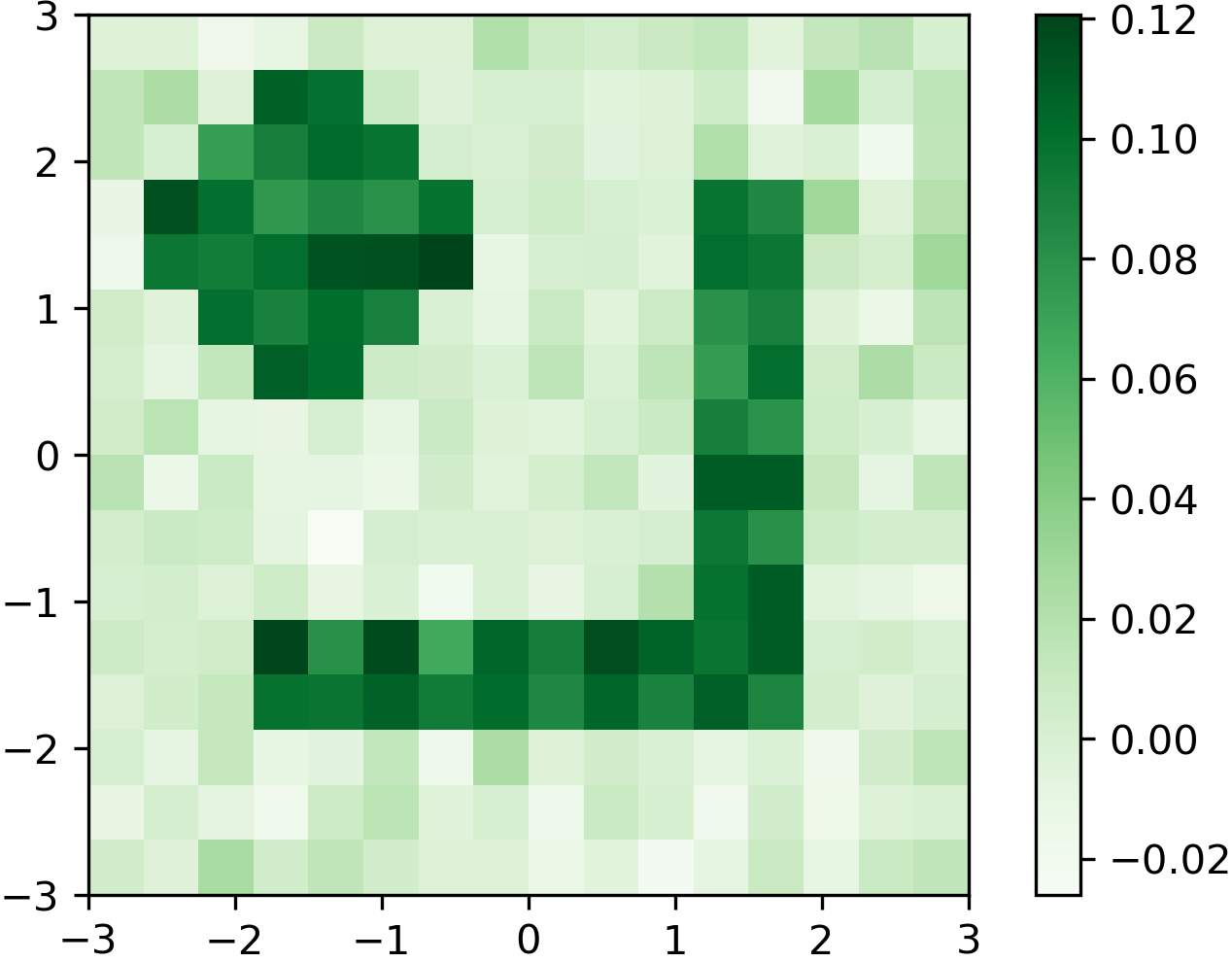}
  \subcaption{$B_2$, $k=7$, $\alpha=5$}
 \end{minipage}
 \begin{minipage}[b]{0.4\linewidth}
  \centering
  \includegraphics[keepaspectratio, scale=0.45]
  {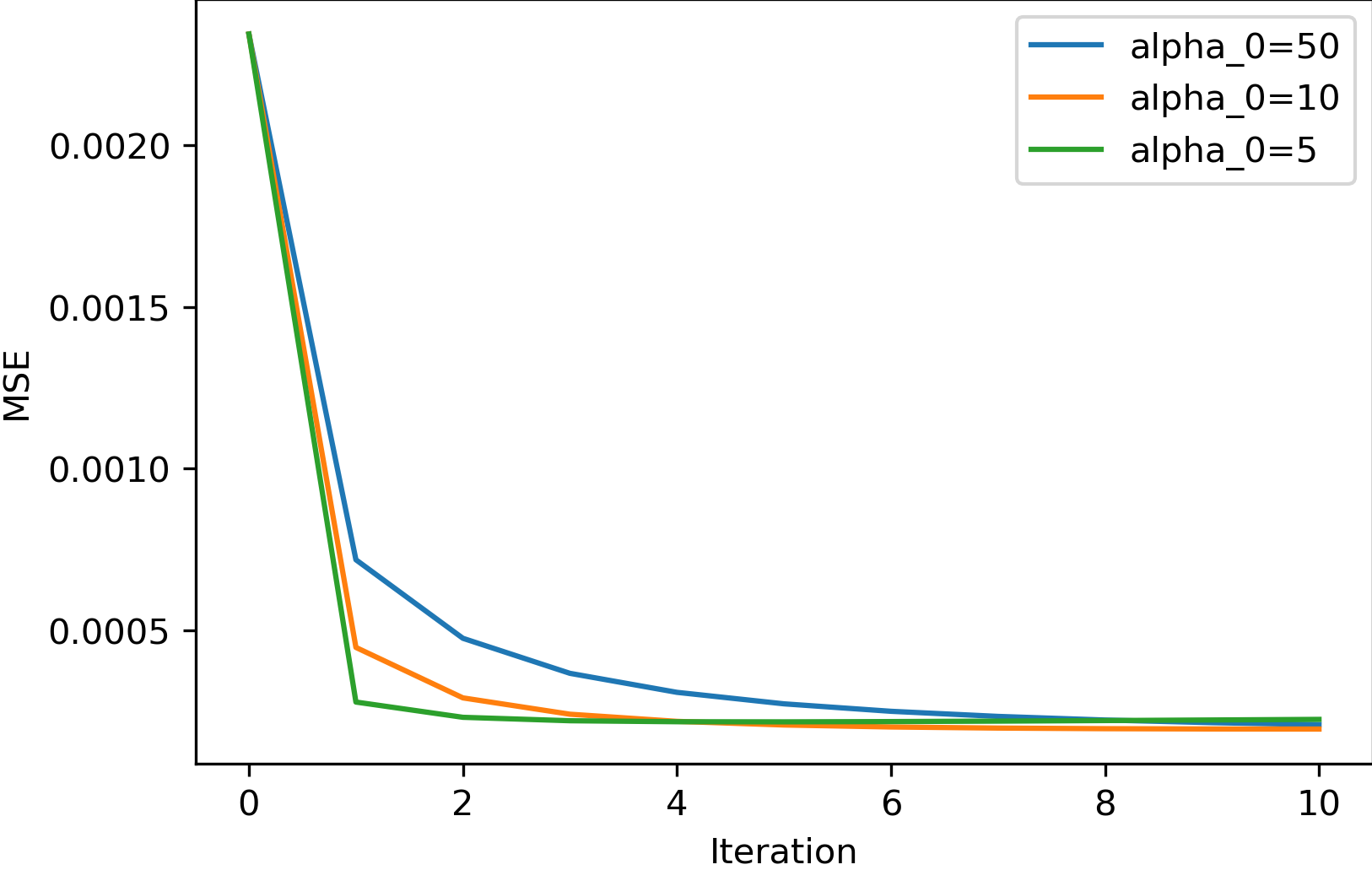}
  \subcaption{$B_2$, $k=7$, error graph}
 \end{minipage}
\end{tabular}
\caption{EKF (nosiy $\sigma=0.1$)}
\label{EKFn01}
\end{figure}

\end{document}